\documentclass[12pt,twoside,reqno]{amsart}
\usepackage[a4paper,textheight=24.5cm, textwidth=16cm, top=1in, right=2.5cm, footskip=8mm]{geometry}
\usepackage{mathrsfs,amssymb,amsmath}
\usepackage{indentfirst}

\usepackage{ifpdf}
\ifpdf
\usepackage[CJKbookmarks=true,
         hyperindex=true,
         pdfstartview=FitH,
         bookmarksnumbered=true,
         bookmarksopen=true,
         colorlinks=true,
         citecolor=blue,
         linkcolor=blue,
         urlcolor=blue,
         pdfauthor={Chengchun Hao},
         pdftitle={},
         pdfkeywords={},
         ]{hyperref}
\else
\usepackage[
         hypertex,
         hyperindex,
         linkcolor=blue,
         unicode,
         citecolor=blue%
         ]{hyperref}
\fi

\allowdisplaybreaks
\numberwithin{equation}{section}
\usepackage{amsthm}
\newtheorem{theorem}{Theorem}[section]

\newtheorem{lemma}[theorem]{Lemma}
\newtheorem{proposition}[theorem]{Proposition}
\theoremstyle{definition}

\theoremstyle{remark}
\newtheorem{remark}{Remark}[section]

\newcommand{\norm}[1]{\Vert#1\Vert}
\newcommand{\abs}[1]{\vert#1\vert}
\newcommand{\lnorm}[1]{\left\Vert#1\right\Vert}
\newcommand{\labs}[1]{\left\vert#1\right\vert}

\newcommand{\R}{\mathbb R}
\newcommand{\D}{\partial}
\newcommand{\eps}{\varepsilon}

\newcommand{\dv}{\mathrm{div}\,}
\newcommand{\pd}[2]{{#1}_{,#2}\,}
\newcommand{\ri}{\rho_0}

\newcommand{\nb}{\nabla}
\newcommand{\hd}{\overline{\D}}

\newcommand{\curl}{\mathrm{curl}\,}
\newcommand{\dist}{\mathrm{dist}\,}
\newcommand{\ls}{\leqslant}
\newcommand{\gs}{\geqslant}
\newcommand{\N}{\mathcal{N}}

\newcommand{\sgn}{\mathrm{sgn}}

\newcommand{\no}{\nonumber}

\newcommand{\dveta}{\mathrm{div}_{\!\eta}\,}
\newcommand{\curleta}{\mathrm{curl}_\eta\,}
\newcommand{\I}{\mathfrak{I}}
\newcommand{\ceil}[1]{\left\lceil#1\right\rceil}

\newcommand{\intt}{\int_0^T\!\!\!\!\int_\Omega }

\begin{document}

\title[FREE-BOUNDARY COMPRESSIBLE EULER SYSTEM IN VACUUM]{Remarks on the free boundary problem of compressible Euler equations in physical vacuum with general initial densities}%
\author{Chengchun Hao}
\address{Institute of Mathematics,
 Academy of Mathematics and Systems Science,
 and Hua Loo-Keng Key Laboratory of Mathematics,
  Chinese Academy of Sciences,
   Beijing 100190, China}
\email{hcc@amss.ac.cn}
\thanks{The author's work was partially supported by the National Natural Science Foundation of China (Grant No. 11171327) and by the Youth Innovation Promotion Association, Chinese Academy of Sciences.}

\begin{abstract}
  In this paper, we establish a priori estimates for the three-dimensional compressible Euler equations with moving physical vacuum boundary, the $\gamma$-gas law equation of state for $\gamma=2$ and the general initial density $\ri \in H^5$. Because of the degeneracy of the initial density, we investigate the estimates of the horizontal spatial and time derivatives and then obtain the estimates of the normal or full derivatives through the elliptic-type estimates. We derive a mixed space-time interpolation inequality which play a vital role in our energy estimates and obtain some extra estimates for the space-time derivatives of the velocity in $L^3$.
\end{abstract}


\maketitle

\section{Introduction}

In the present paper, we consider the following compressible Euler  equations
\begin{subequations}\label{eu1}
  \begin{align}
    &\D_t\rho+\dv(\rho u)=0  &&\text{ in } \Omega(t)\times(0,T],\label{eu1.1}\\
    &\D_t(\rho u)+ \dv(\rho u\otimes u)+\nb p=0  &&\text{ in } \Omega(t)\times(0,T], \label{eu1.2}\\
    &p=0\; &&\text{ on } {\Gamma_1}(t)\times(0,T],\label{eu1.5}\\
    &u_3=0\; &&\text{ on } {\Gamma_0}\times(0,T],\label{eq.cond.bdd}\\
    &\partial_t {\Gamma_1}(t)=\mathcal{V}({\Gamma_1}(t))=u\cdot \N &&\text{ in } (0,T],\label{eu1.7}\\
    &(\rho,u)=(\ri,u_0) &&\text{ in } \Omega\times\{t=0\} ,\label{eu1.8}\\
    &\Omega(0)=\Omega,\; {\Gamma_1}(0)={\Gamma_1},\label{eu1.9}
\end{align}
\end{subequations}
where $\rho$ denotes the density, the vector field $u=(u_1,u_2,u_3)\in\R^3$ denotes the Eulerian velocity field, and $p$ denotes the pressure function. $\nb=(\D_1,\D_2,\D_3)$ and $\dv$ are the usual gradient operator and spatial divergence where $\D_i=\D/\D x_i$. The open, bounded domain $\Omega(t)\subset \R^3$ denotes the changing volume occupied by the fluid, ${\Gamma_1}(t):=\D\Omega(t)$ denotes the moving vacuum boundary, $\mathcal{V}({\Gamma_1}(t))$ denotes the normal velocity of ${\Gamma_1}(t)$,  $\N$ denotes the outward unit normal vector to the boundary ${\Gamma_1}(t)$, and $\Gamma_0$ is a fixed boundary. The equation of the pressure $p(\rho)$ is given by
\begin{align}
  p(x,t)=C_\gamma\, \rho^\gamma(x,t),
\end{align}
where $\gamma$ is the adiabatic index and we will only consider the case $\gamma=2$ in this paper;  $C_\gamma$ is the adiabatic constant which we set to unity, i.e., $C_\gamma=1$; and
\begin{align}
  \rho>0 \quad \text{in } \Omega(t) \quad \text{and}\quad \rho=0 \; \text{on } {\Gamma_1}(t).
\end{align}

Equation \eqref{eu1.1} is the conservation of mass, \eqref{eu1.2} is the conservation of momentum. The boundary condition \eqref{eu1.5} states that pressure (and hence density) vanishes along the vacuum boundary, \eqref{eq.cond.bdd} describes that the normal component of the velocity vanishes on the fixed boundary $\Gamma_0$,  \eqref{eu1.7} indicates that the vacuum boundary is moving with the normal component of the fluid velocity, \eqref{eu1.8} and \eqref{eu1.9} are the initial conditions for the density, velocity, domain and boundary.

To avoid the use of local coordinate charts necessary for arbitrary geometries, for simplicity, we assume that the initial domain $\Omega\subset \R^3$ at time $t=0$ is given by
\begin{align*}
  \Omega=\left\{(x_1,x_2,x_3)\in\R^3 : (x_1,x_2)\in\mathbb{T}^2,\; x_3\in(0,1)\right\},
\end{align*}
where $\mathbb{T}^2$ denotes the $2$-torus, which can be thought of as the unit square with periodic boundary conditions. This permits the use of one global Cartesian coordinate system. At $t=0$, the reference vacuum boundary is the top boundary
\begin{align*}
  {\Gamma_1}=\{(x_1,x_2,x_3)\in\R^3 : (x_1,x_2)\in\mathbb{T}^2,\; x_3=1\},
\end{align*}
while the bottom boundary
\begin{align*}
  {\Gamma_0}=\{(x_1,x_2,x_3)\in\R^3 : (x_1,x_2)\in\mathbb{T}^2,\; x_3=0\}
\end{align*}
is fixed with the boundary condition \eqref{eq.cond.bdd}.

We set the unit normal vectors $N=(0,0,1)$ on $\Gamma_1$ and $N=(0,0,-1)$ on $\Gamma_0$.  We use the standard basis on $ \mathbb{R}^3  $: $e_1=(1,0,0)$, $e_2=(0,1,0)$ and $e_3=(0,0,1)$.  Similarly, the unit tangent vectors on $\Gamma=\Gamma_0\cup\Gamma_1$ are given by
 \begin{align*}
   T_1=(1,0,0) \quad \text{and} \quad T_2=(0,1,0).
 \end{align*}

Throughout the paper, we will use Einstein's summation convention. The $k^{\mathrm{th}}$-partial derivative of $F$ will be denoted by $\pd{F}{k}=\frac{\D F}{\D x^k}$. Then we have
\begin{align*}
  (\dv(\rho u\otimes u))^j=&\pd{(\rho u^iu^j)}{i}=\dv(\rho u)u^j+\rho u\cdot\nb u^j,
\end{align*}
which yields that \eqref{eu1.2} can be rewritten, in view of \eqref{eu1.1}, as
\begin{align*}
  \rho(\D_t u+u\cdot\nb u)+\nb p=0.
\end{align*}
Thus, the system \eqref{eu1} can be rewritten as
\begin{subequations}\label{eu2}
  \begin{align}
    &\D_t\rho+\dv(\rho u)=0  &&\text{ in } \Omega(t)\times(0,T],\label{eu2.1}\\
    &\rho(\D_t u+u\cdot\nb u)+\nb p=0 &&\text{ in } \Omega(t)\times(0,T], \label{eu2.2}\\
    &p=0\; &&\text{ on } {\Gamma_1}(t)\times(0,T],\label{eu2.6}\\
    &u_3=0\; &&\text{ on } {\Gamma_0}\times(0,T],\label{eu2.7}\\
    &\partial_t {\Gamma_1}(t)=u\cdot \N &&\text{ in } (0,T],\\
    &(\rho,u)=(\ri,u_0) &&\text{ in }\Omega\times \{t=0\},\\
    &\Omega(0)=\Omega,\; {\Gamma_1}(0)={\Gamma_1}.
\end{align}
\end{subequations}

With the sound speed given by $c:=\sqrt{\D p/\D\rho}$ and $N$ denoting the outward unit normal to ${\Gamma_1}$, the satisfaction of the condition
\begin{align}\label{eq.phyvacuum}
  -\infty<\frac{\D c_0^2}{\D N}\ls -\eps_0<0
\end{align}
in a small neighborhood of the boundary defines a physical vacuum boundary (cf. \cite{Liu00}), where $c_0=c|_{t=0}$ denotes the initial sound speed of the gas and $\eps_0>0$ is a constant. In other words, the pressure accelerates the boundary in the normal direction. The physical vacuum condition \eqref{eq.phyvacuum} for $\gamma=2$ is equivalent to the requirement
\begin{align}\label{eq.phyvacuum1}
  \frac{\D\ri  }{\D N}\ls -\frac{\eps_0}{2}< 0 \quad \text{ on } {\Gamma_1}.
\end{align}
Since $\ri>0$ in $\Omega$, \eqref{eq.phyvacuum1} implies that for some positive constant $C$ and $x\in\Omega$ near the vacuum boundary ${\Gamma_1}$,
\begin{align}\label{eq.ricon}
  \ri   (x)\gs C\dist(x,{\Gamma_1}),
\end{align}
where $\dist(x,{\Gamma_1})$ denotes the distance of $x$ away from $\Gamma_1$.

The moving boundary is characteristic because of the evolution law \eqref{eu1.7}, and the system of conservation laws is degenerate because of the appearance of the density function as a coefficient in the nonlinear wave equation which governs the dynamics of the divergence of the velocity of the gas. In turn, weighted estimates show that this wave equation indeed loses derivatives with respect to the uniformly hyperbolic non-degenerate case of a compressible liquid, wherein the density takes the value of a strictly positive constant on the moving boundary \cite{CF76}. The condition \eqref{eq.ricon} violates the uniform Kreiss-Lopatinskii condition \cite{Kreiss} because of resonant wave speeds at the vacuum boundary for the linearized problem. The methods developed for symmetric hyperbolic conservation laws would be extremely difficult to implement for this problem, wherein the degeneracy of the vacuum creates further difficulties for the linearized estimates.

Now, we transform the system \eqref{eu2} in terms of Lagrangian variables. Let $\eta(x,t)$ denote the ``position'' of the gas particle $x$ at time $t$. Thus,
\begin{align}\label{eq.eta}
  \D_t\eta=u\circ \eta \quad \text{for } t>0, \quad \text{and} \quad \eta(x,0)=x,
\end{align}
where $\circ$ denotes the composition, i.e., $(u\circ\eta)(x,t):=u(\eta(x,t),t)$.

We let
\begin{align*}
  v=u\circ\eta, \quad &f=\rho\circ\eta,  \quad
  A=(\nb \eta)^{-1}, \quad J=\det \nb\eta, \quad a=JA,
\end{align*}
where $A$ is the inverse of the deformation tensor
\begin{align*}
  \nb\eta=(\pd{\eta^i}{j})=\begin{pmatrix}
                             \pd{\eta^1}{1} & \pd{\eta^1}{2} & \pd{\eta^1}{3} \\
                             \pd{\eta^2}{1} & \pd{\eta^2}{2} & \pd{\eta^2}{3} \\
                             \pd{\eta^3}{1} & \pd{\eta^3}{2} & \pd{\eta^3}{3} \\
                           \end{pmatrix},
\end{align*}
$J$ is the Jacobian determinant and $a$ is the classical adjoint of $\nb\eta$, i.e., the transpose of the cofactor matrix of $\nb\eta$, explicitly,
\begin{align}\label{eq.a}
  a=\begin{pmatrix}
      \pd{\eta^2}{2}\pd{\eta^3}{3}-\pd{\eta^2}{3}\pd{\eta^3}{2} & \pd{\eta^1}{3}\pd{\eta^3}{2}-\pd{\eta^1}{2}\pd{\eta^3}{3} & \pd{\eta^1}{2}\pd{\eta^2}{3}-\pd{\eta^1}{3}\pd{\eta^2}{2} \\
      \pd{\eta^2}{3}\pd{\eta^3}{1}-\pd{\eta^2}{1}\pd{\eta^3}{3} & \pd{\eta^1}{1}\pd{\eta^3}{3}-\pd{\eta^1}{3}\pd{\eta^3}{1} & \pd{\eta^1}{3}\pd{\eta^2}{1}-\pd{\eta^1}{1}\pd{\eta^2}{3}\\
      \pd{\eta^2}{1}\pd{\eta^3}{2}-\pd{\eta^2}{2}\pd{\eta^3}{1} & \pd{\eta^1}{2}\pd{\eta^3}{1}-\pd{\eta^1}{1}\pd{\eta^3}{2} & \pd{\eta^1}{1}\pd{\eta^2}{2}-\pd{\eta^1}{2}\pd{\eta^2}{1} \\
    \end{pmatrix}.
\end{align}

Since on the fixed boundary $\Gamma_0$, $\eta^3=x_3=0$ so that according to \eqref{eq.a}, the components $a_1^3=a_2^3=0$ on $\Gamma_0$, and $v_3=0$ on $\Gamma_0$ due to $v\cdot(0,0,-1)=0$ where $(0,0,-1)$ is the outward unit normal vector to $\Gamma_0$, then the Lagrangian version of \eqref{eu2} can be written in the fixed reference domain $\Omega$ as
\begin{subequations}\label{eu3}
  \begin{align}
    &f_t+fA_i^j\pd{v^i}{j}=0  &&\text{ in } \Omega\times(0,T],\label{eu3.1}\\
    &fv_t^i+A_i^j\pd{f^2}{j}=0 &&\text{ in } \Omega\times(0,T], \label{eu3.2}\\
    &f=0\; &&\text{ on }{\Gamma_1}\times(0,T],\label{eu3.6}\\
    &a_1^3=a_2^3=0,\; v_3=0\; &&\text{ on }{\Gamma_0}\times(0,T],\label{eu3.7}\\
    &(\rho,v,\eta)=(\ri,u_0,e) &&\text{ in } \Omega\times\{t=0\},
\end{align}
\end{subequations}
where $e(x)=x$ denotes the identity map on $\Omega$.

From the derivative formula of determinants, we have
\begin{align}
  J_t=&JA^j_l\pd{v^l}{j}=a^j_l\pd{v^l}{j}.\label{eq.Jt}
\end{align}
It follows from \eqref{eu3.1} and \eqref{eq.Jt} that
\begin{align}
  f_t+fJ^{-1}J_t=0, \text{ or } \D_t(fJ)=0, \text{ i.e., } f=\ri J^{-1},
\end{align}
thus, the initial density function $\ri$ can be viewed as a parameter in compressible Euler equations.

Using the identity $A_i^j=J^{-1}a_i^j$, we write \eqref{eu3} as
\begin{subequations}\label{eu4}
  \begin{align}
    &\ri  v_t^i+a^j_i\pd{\left(\ri^2 J^{-2}\right)}{j}= 0 &&\text{ in } \Omega\times(0,T], \label{eu4.2}\\
    &\ri=0\; &&\text{ on } {\Gamma_1},\label{eu4.6}\\
    &a_1^3=a_2^3=0,\; v_3=0\; &&\text{ on }{\Gamma_0}\times(0,T],\label{eu4.7}\\
    &(v,\eta)=(u_0,e) &&\text{ in } \Omega\times\{t=0\},
\end{align}
\end{subequations}
with $\ri  (x)\gs C\dist(x,{\Gamma_1})$ for $x\in\Omega$ near ${\Gamma_1}$.

To understand the behavior of vacuum states is an important problem in gas and fluid dynamics. In particular, the physical vacuum, in which the boundary moves with a nontrivial finite normal
acceleration, naturally arises in the study of the motion of gaseous stars or shallow water \cite{JM10}.
Despite its importance, there are only few mathematical results available near vacuum.
The main difficulty lies in the fact that the physical systems become degenerate along the
vacuum boundary. The existence and uniqueness for the three-dimensional compressible Euler equations modeling a liquid rather than a gas was established in \cite{L2} where the density is positive on the vacuum boundary. Trakhinin provided an alternative proof for the existence of a compressible liquid, employing a solution strategy based on symmetric hyperbolic systems combined with the Nash-Moser iteration in \cite{Tr09}.

The local existence for the physical vacuum singularity can be found in the recent papers by Jang and Masmoudi \cite{JM09,JM10} and by Coutand and Shkoller \cite{CS11,CS12} for the one-dimensional and three-dimensional compressible gases.  Coutand, Lindblad and Shkoller \cite{CLS10} established a priori estimates based on time differentiated energy estimates and elliptic estimates for normal derivatives for $\gamma=2$ with $\ri\in H^4(\Omega)$ where the energy function was given by
\begin{align*}
  \bar{E}(t):=&\sum_{\ell=0}^4\norm{\D_t^{2\ell}\eta(t)}_{4-\ell}^2+\sum_{\ell=0}^4 \left[\norm{\ri \hd^{4-\ell} \D_t^{2\ell} \nb\eta(t)}_0^2+\norm{\sqrt{\ri}\hd^{4-\ell} \D_t^{2\ell } v(t)}_0^2\right]\no\\
  & +\sum_{\ell=0}^3\norm{\ri \D_t^{2\ell} J^{-2}(t)}_{4-\ell}^2+\norm{\curleta v(t)}_3^2+\norm{\ri\hd^4\curleta v(t)}_0^2.
\end{align*}

We will not attempt to address exhaustive references in this paper. For more related references, we refer the interested reader to \cite{CS12,JM10} and references therein for a nice history of the analysis of compressible Euler equations.

In the present paper, we will use a similar argument as in \cite{CLS10} to consider the cases of general initial densities. We will prove a mixed space-time interpolation inequality under the framework of Lebesgue spaces which will play a vital role in our energy estimates, rather than the framework of Lebesgue spaces for time but Sobolev spaces for spatial variables in bounded domain used in \cite{CLS10}, then we can obtain some extra estimates of space-time derivatives of $v(t)$ in $L^3(\Omega)$. In order to deal with some sub-higher order terms (e.g., \eqref{eq.why5order}) in the argument for the time derivatives, we have to investigate the fifth order energy estimates which are closed with themselves.

We now derive the physical energy of the system \eqref{eu4}. From \eqref{eq.Jt}, the Piola identity \eqref{eq.Piola} given in Section \ref{sec.4}, we get
  \begin{align*}
\frac{1}{2}\D_t(\ri |v|^2)=&\ri v_i v_t^i
    =-a^j_iv^i\pd{\left(\ri^2 J^{-2}\right)}{j}
    =-\pd{\left(a^j_iv^i\left(\ri^2 J^{-2}\right)\right)}{j}+J_t\ri^2 J^{-2}\\
    =&-  \ri^2 \D_t J^{-1}-\pd{(a^j_iv^i\ri^2 J^{-2})}{j}.
  \end{align*}
  Since $\ri=0$ on the boundary ${\Gamma_1}$ and $a^3_iv^i=0$ on $\Gamma_0$, integrating over $\Omega$ yields,  with the help of Gauss' theorem, that
\begin{align*}
    E_0(t):=&\int_\Omega \left(\frac{1}{2}\ri(x) |v(x,t)|^2+ \ri^2(x) J^{-1}(x,t)\right) dx\\
    =&\frac{1}{2}\norm{\sqrt{\ri}v}_0^2 + \norm{\ri J^{-1/2}}_0^2=E_0(0)
\end{align*}
conserves for all $t\gs 0$.

Although the physical energy is a conserved quantity, it is far too weak for the purposes of constructing solutions. Instead, we introduce the following $r^{\text{th}}$ order energy function
\begin{align*}
  E_r(t):=&\sum_{\ell=0}^r\Big[\norm{\D_t^{2\ell}\eta(t)}_{r-\ell}^2+ \norm{\ri^{1/2}\D_t^{2\ell}\hd^{r-\ell}v(t)}_0^2
  +\norm{\ri \D_t^{2\ell}\hd^{r-\ell}\nb \eta(t)}_0^2\Big]+\norm{\curleta v(t)}_{r-1}^2\\ &+\sum_{\ell=0}^{r-1}\Big[\norm{\D_t^{2\ell} v(t)}_{H_3^{r-1-\ell}(\Omega)}+\norm{\ri \D_t^{2\ell} J^{-2}(t)}_{r-\ell}^2\Big]+\norm{\ri\hd^r\curleta v(t)}_0^2.
\end{align*}

Now, we state our main result as follows.
\begin{theorem}
  Let $\gamma=2$. Suppose that $(\eta(t),v(t))$ is a smooth solution of the system \eqref{eq.eta} and \eqref{eu4} on a time interval $[0,T']$ satisfying the initial bounds $E_5(0)<\infty$ and that the initial density $\ri>0$ in $\Omega$ and $\ri  \in H^5(\Omega)$ satisfies the physical vacuum condition \eqref{eq.ricon}. Then there exists a $T=T(E_5(0))>0$ so small  that the energy function $E_5(t)$ constructed from the solution $(\eta(t),v(t))$ satisfies the a priori estimate
  \begin{align*}
    \sup_{[0,T]}E_5(t)\ls C(E_5(0)).
  \end{align*}
\end{theorem}

\begin{remark}
  The same arguments and the results hold true if the bottom boundary $\Gamma_0$ is also a moving vacuum boundary, i.e., by changing the boundary condition \eqref{eq.cond.bdd} into $p=0$ on $\Gamma_0(t)\times (0,T]$, which will not cause any additional difficulties except for the transformation of coordinates.
\end{remark}

\begin{remark}
  For the general cases $\gamma>1$ with general densities, we give some further remarks. We think that they are much more different from the special case $\gamma=2$. They need to reform the energy function in order to get a priori estimates. For the cases $\gamma>2$, it seems to be similar to the case $\gamma=2$ due to $\gamma-1>\gamma/2$ in view of the exponent of the weight $\ri^{\gamma/2}$ and $\ri^{\gamma-1}$ and weighted Sobolev embedding relations given in Section~\ref{sec.2}, but it is not easy to deal with the weight $\ri^{\gamma/2}$ in energy estimates in view of the higher order Hardy inequality. For the cases $1<\gamma<2$, one have to use the weight $\ri^{\gamma-1}$ instead of $\ri^{\gamma/2}$ in constructing the energy function according to the physical vacuum condition, the higher order Hardy inequality and  weighted Sobolev embedding relations, especially for the cases $3/2\ls \gamma<2$, however one must deal with many extra, important and difficult remainder integrals in the estimates of every horizontal, time or mixed derivatives. For the cases $1<\gamma<3/2$, it might be different from and difficult than the above cases.
\end{remark}

Throughout the paper, we will use the following notation: two-dimensional gradient vector or horizontal derivative $\hd=(\D_1,\D_2)$, the $H_p^s(\Omega)$ interior norm $\norm{\cdot}_{H_p^s(\Omega)}$, the $H^s(\Omega)$ interior norm $\norm{\cdot}_s$ when $p=2$, and the $H^s({\Gamma})$ boundary norm $\abs{\cdot}_s$. The component of a matrix $M$ at the $i^{\mathrm{th}}$ row and the $j^{\mathrm{th}}$ column will be denoted by $M^i_j$. Sometimes, we will use ``$\lesssim$'' to stand for ``$\ls C$'' with a generic constant $C$. For more notations, one can read the appendix.

The rest of this paper is organized as follows. We will give some preliminaries in Appendix~\ref{sec.pre}. Precisely, we introduce some notations and weighted Sobolev spaces in Section~\ref{sec.2}; we recall the higher-order Hardy-type inequality and Hodge's decomposition elliptic estimates in Section~\ref{sec.3}; we give the properties of the determinant $J$, the inverse of the deformation tensor $A$ and the transpose of the cofactor matrix $a$ in Section~\ref{sec.4}. We give a mixed space-time interpolation inequality in Section~\ref{sec.interpolation}, and derive the zero-th order energy estimates in Section~\ref{sec.5} and the curl estimates in Section~\ref{sec.6}. Since the standard energy method is very problematic due to the degeneracy of $\ri$, we first derive the estimates of the horizontal and time derivatives in Sections~\ref{sec.7}-\ref{sec.mix} and then obtain the estimates of normal or full derivatives through the elliptic-type estimates in Section~\ref{sec.8}. Finally, we will complete the proof of the a priori estimates in Section~\ref{sec.9}.

\section{A mixed space-time interpolation inequality}\label{sec.interpolation}

In this section, we prove a useful mixed space-time interpolation inequality which will play a vital role in our energy estimates.
\begin{proposition}[Mixed interpolation inequality] \label{prop.interpolation}
  Let $F(t,x)$ be a scalar or vector-valued function for $t\in [0,T]$, $T>0$ and $x\in \Omega\subset \R^3$. Assume that $F_t(0,\cdot)\in L^3(\Omega)$, $F\in L^\infty([0,T]; L^6(\Omega))$ and $F_{tt}\in L^\infty([0,T]; L^2(\Omega))$, then we have
  \begin{align}\label{eq.interpolation}
    \norm{F_t}_{L^3([0,T]\times\Omega)}^2\ls CT^{2/3}\Big[\norm{F_t(0)}_{L^3(\Omega)}^2+\sup_{[0,T]}\norm{F(t)}_{L^6(\Omega)} \norm{F_{tt}(t)}_{L^2(\Omega)}\Big],
  \end{align}
  where $C$ is a constant independent of $T$, $\Omega$ and $F$.
\end{proposition}

\begin{proof} Notice that $$2\abs{F_t}(\D_t\abs{F_t})F_t=\D_t(\abs{F_t}^2)F_t=2(F_t\cdot F_{tt})F_t\ls 2\abs{F_t}^2\abs{F_{tt}}$$
implies that $\abs{\D_t(\abs{F_t})F_t}\ls \abs{F_t}\abs{F_{tt}}$. Then, by the Fubini theorem, integration by parts with respect to time, the fundamental theorem of calculus, the H\"older inequality and the Minkowski inequality, we have
  \begin{align*}
&\norm{F_t}_{L^3([0,T]\times\Omega)}^3=\intt \abs{F_t}^3dxdt=\int_\Omega \int_0^T \abs{F_t}F_t\cdot F_t dtdx\\
    =&\int_\Omega \abs{F_t(T)}F_t(T)\cdot F(T)dx-\int_\Omega \abs{F_t(0)}F_t(0)\cdot F(0)dx-\int_\Omega \int_0^T \D_t(\abs{F_t}F_t)\cdot F dtdx\\
    =&\int_\Omega \left(\int_0^T\D_t(\abs{F_t}F_t)dt\right)\cdot F(T)dx+\int_\Omega \abs{F_t(0)}F_t(0)\cdot \left(\int_0^T F_t dt\right)dx\\
    &-\intt \abs{F_t}F_{tt}\cdot F dxdt-\intt \D_t(\abs{F_t}) F_t\cdot Fdxdt\\
    \ls& 2\int_\Omega \left(\int_0^T\abs{F_{tt}}\abs{F_t}dt\right) \abs{F(T)}dx+\int_\Omega \abs{F_t(0)}^2 \left(\int_0^T \abs{F_t} dt\right)dx+2\intt \abs{F_t}\abs{F_{tt}}\abs{F} dxdt\\
    \ls&C\norm{F(T)}_{L^6(\Omega)}\int_0^T\norm{\abs{F_{tt}}\abs{F_t}}_{L^{6/5}(\Omega)}dt +C\norm{F_t(0)}_{L^3(\Omega)}^2\int_0^T\norm{F_t}_{L^3(\Omega)}dt\\
    &+C\norm{F_t}_{L^3([0,T]\times\Omega)}\norm{F_{tt}}_{L^2([0,T] \times\Omega)} \norm{F}_{L^6([0,T]\times\Omega)}\\
    \ls &CT^{2/3}\norm{F_t}_{L^3([0,T]\times\Omega)}\left[ \norm{F_t(0)}_{L^3(\Omega)}^2+ \sup_{[0,T]}\norm{F}_{L^6(\Omega)} \norm{F_{tt}}_{L^2(\Omega)}\right],
  \end{align*}
  which implies the desired inequality by eliminating $\norm{F_t}_{L^3([0,T]\times\Omega)}$ from both sides of the inequality.
\end{proof}

\section{A priori assumption and the zero-th order energy estimates}\label{sec.5}

We assume that we have smooth solutions $\eta$ on a time interval $[0,T]$, and that for all such solutions, the time $T>0$ is taken sufficiently small so that for $t\in[0,T]$,
\begin{align}\label{priass.J}
  \frac{1}{2}\ls J(t)\ls \frac{3}{2}.
\end{align}
Once we establish the a priori bounds, we can ensure that our solution verifies  the assumption \eqref{priass.J} by means of the fundamental theorem of calculus. Then, by \eqref{eq.a}, Sobolev's embedding $H^2(\Omega)\subset L^\infty(\Omega)$, we have for $t\in[0,T]$,
\begin{align}
  \norm{\eta(t)}_{L^\infty(\Omega)}\ls& C\norm{\eta(t)}_2,\\
  \norm{a(t)}_{L^\infty(\Omega)}\ls& C\norm{\nb\eta(t)}_{L^\infty(\Omega)}^2\ls C\norm{\eta(t)}_{3}^2.
\end{align}
It follows from $a=JA$ and \eqref{priass.J} that
\begin{align}\label{eq.Ainfty}
  \norm{A(t)}_{L^\infty(\Omega)}\ls& \norm{J^{-1}(t)a(t)}_{L^\infty(\Omega)}
  \ls C\norm{\eta(t)}_{3}^2.
\end{align}

Now, we prove the following zero-th order energy estimates.
\begin{proposition}
  It holds for $r\gs 4$
  \begin{align*}
  &\sup_{[0,T]}\left[\norm{\ri^{1/2}v}_0^2+ \norm{\ri  J^{-1/2}}_0^2\right]
  \ls M_0+CTP(\sup_{[0,T]}E_r(t)).
\end{align*}
\end{proposition}

\begin{proof}
Taking the $L^2$-inner product of \eqref{eu4.2} with $v^i$ yields, by integration by parts, \eqref{eq.Piola}, \eqref{eu4.6}, \eqref{eu4.7}  and \eqref{eq.Jt}, that
\begin{align*}
 \frac{1}{2}\frac{d}{dt} \int_\Omega\ri |v|^2dx
 =&-\int_\Omega a^j_i\pd{\left(\ri^2  J^{-2}\right)}{j}v^idx
 =\int_\Omega\ri^2  J^{-2}a^j_i\pd{v^i}{j}dx -\int_{\Gamma_0} a^3_i\ri^2  J^{-2}v^idx_1dx_2\\
 =&- \frac{d}{dt}\int_\Omega \ri^2 J^{-1}dx.
\end{align*}
It follows, from the integration over $[0,T]$ and by H\"older's inequality, \eqref{priass.J} and \eqref{eq.a}, that
\begin{align*}
\frac{1}{2}\int_\Omega\ri |v|^2dx+ \int_\Omega \ri^2 J^{-1}dx
  =&\frac{1}{2}\int_\Omega\ri |u_0|^2dx+ \int_\Omega \ri^2 dx
  =\frac{1}{2}\norm{\ri^{1/2}u_0}_0^2 + \norm{\ri  }_0^2,
\end{align*}
which implies, by Cauchy's inequality, that
\begin{align*}
  &\frac{1}{4}\sup_{[0,T]}\norm{\ri^{1/2}v}_0^2+ \sup_{[0,T]}\norm{\ri  J^{-1/2}}_0^2
  \ls M_0+CTP(\sup_{[0,T]}E_r(t)).
\end{align*}
Thus, we complete the proof.
\end{proof}

\section{The curl estimates}\label{sec.6}

Taking the Lagrangian curl of \eqref{eu4.2} yields  that
\begin{align}\label{eq.curletavt0}
   \eps_{lji}A_j^s\pd{v_t^i}{s}=0, \quad\text{ or }\quad \curleta v_t=0.
\end{align}

We can obtain the following proposition.

\begin{proposition}\label{prop.curl}
  For all $t\in(0,T]$, we have for $r\gs 4$
\begin{align}
    \sum_{\ell=0}^{r-1}\norm{\curl\D_t^{2\ell} \eta(t)}_{r-1-\ell}^2+\sum_{\ell=0}^r\norm{\ri  \hd^{r-\ell}\curl\D_t^{2\ell} \eta(t)}_0^2
    \ls M_0+CTP(\sup_{[0,T]}E_r(t)).\label{eq.curl.1}
\end{align}
\end{proposition}

\begin{proof}
From \eqref{eq.curletavt0}, we get
\begin{align*}
  \D_t(\curleta v)=\eps_{\cdot ji}A_{tj}^s \pd{v^i}{s}.
\end{align*}
Integrating over $(0,t)$ yields
\begin{align}\label{eq.curl}
  \curleta v(t)=\curl u_0+\int_0^t \eps_{\cdot ji} A_{t'j}^s\pd{v^i}{s}dt',
\end{align}
and computing the $r$-th order horizontal derivatives of this relation yields
\begin{align*}
  \curleta \hd^r v(t)=&\hd^r(\curleta v(t))-\eps_{\cdot ji} \hd^r A_j^s \pd{v^i}{s}-\eps_{\cdot ji}\sum_{k=1}^{r-1}C_r^l\hd^{r-k}A^s_j\hd^k \pd{v^i}{s}\\
  =&\hd^r\curl u_0-\eps_{\cdot ji} \hd^r  A_j^s \pd{v^i}{s}-\eps_{\cdot ji}\sum_{k=1}^{r-1}C_r^l\hd^{r-k}A^s_j\hd^k \pd{v^i}{s}+\int_0^t \eps_{\cdot ji} \hd^r (A_{t'j}^s\pd{v^i}{s})dt'.
\end{align*}
Noticing that $\eta_t=v$, we have
\begin{align}
&\D_t(\curleta \hd^r \eta)=\curleta \hd^r  v+\eps_{\cdot ji}A_{tj}^s\hd^r  \pd{\eta^i}{s}\no\\
  =&\hd^r\curl u_0-\eps_{\cdot ji} \hd^r  A_j^s \pd{v^i}{s}-\eps_{\cdot ji}\sum_{k=1}^{r-1}C_r^l\hd^{r-k}A^s_j\hd^k \pd{v^i}{s}+\eps_{\cdot ji}A_{tj}^s\hd^r  \pd{\eta^i}{s}+\int_0^t \eps_{\cdot ji} \hd^r (A_{t'j}^s\pd{v^i}{s})dt'.
\end{align}
Integrating over $[0,t]$ yields
\begin{align*}
  \curleta \hd^r \eta
  =&t\hd^r\curl u_0+\int_0^t\eps_{\cdot ji} A_l^s\hd^r  \pd{\eta^l}{m}A^m_j \pd{v^i}{s}dt'-\sum_{k=1}^{r-1}C_r^l\int_0^t \eps_{\cdot ji}\hd^{r-k}A^s_j\hd^k \pd{v^i}{s}dt'\\
  &-\int_0^t\eps_{\cdot ji}A_l^s\pd{v^l}{m}A^m_j\hd^r  \pd{\eta^i}{s}dt'-\int_0^t \int_0^{t'}\eps_{\cdot ji} \hd^r  (A_p^s \pd{v^p}{q} A^q_j\pd{v^i}{s})dt''dt',\no
\end{align*}
and then by the fundamental theorem of calculus
\begin{align}\label{eq.curleta41}
  &\curleta \hd^r \eta
  =\curl \hd^r \eta+\int_0^t \eps_{\cdot ji}A^s_{tj}(t')dt'\hd^r  \pd{\eta^i}{s}.
\end{align}
It follows that
\begin{align*}
  &\hd^r\curl \eta
  =t\hd^r\curl  u_0+\int_0^t\eps_{\cdot ji} A_l^s\hd^r  \pd{\eta^l}{m}A^m_j \pd{v^i}{s}dt'-\sum_{k=1}^{r-1}C_r^k\int_0^t \eps_{\cdot ji}\hd^{r-k}A^s_j\hd^k \pd{v^i}{s}dt'\\
  &-\int_0^t\eps_{\cdot ji}A_l^s\pd{v^l}{m}A^m_j\hd^r  \pd{\eta^i}{s}dt'-\int_0^t \eps_{\cdot ji}A^s_{tj}(t')dt'\hd^r \pd{\eta^i}{s}dt'-\int_0^t \int_0^{t'}\eps_{\cdot ji} \hd^r  (A_p^s \pd{v^p}{q} A^q_j\pd{v^i}{s})dt''dt'.\no
\end{align*}
Notice that for $k=1$, integration by parts with respect to time and the fact $\hd\nb\eta(0)=0$ imply
\begin{align*}
&\int_0^t \eps_{\cdot ji}\hd^{{r-1}}A^s_j\hd \pd{v^i}{s}dt'=-\int_0^t \eps_{\cdot ji}\hd^{r-2}(A^s_p\hd\pd{\eta^p}{q}A^q_j)\hd \pd{v^i}{s}dt'\\
  =&-\int_0^t \eps_{\cdot ji}A^s_p\hd^{r-1}\pd{\eta^p}{q}A^q_j\hd \pd{v^i}{s}dt'-\sum_{l=0}^{r-3}C_{r-2}^l\int_0^t\hd^{r-2-l}(A^s_p A^q_j)\hd^{l+1}\pd{\eta^p}{q}\hd \pd{v^i}{s}dt'\\
  =&-\eps_{\cdot ji}A^s_p\hd^{r-1}\pd{\eta^p}{q}A^q_j\hd \pd{\eta^i}{s}+\int_0^t \eps_{\cdot ji}A^s_p\hd^{r-1}\pd{v^p}{q}A^q_j\hd \pd{\eta^i}{s}dt'\\
  &+\int_0^t \eps_{\cdot ji}\hd^{r-1}\pd{\eta^p}{q}\hd \pd{\eta^i}{s}\D_t(A^s_p A^q_j)dt'-\sum_{l=0}^{r-3}C_{r-2}^l\int_0^t\hd^{r-2-l}(A^s_p A^q_j)\hd^{l+1}\pd{\eta^p}{q}\hd \pd{v^i}{s}dt'.
\end{align*}
Since  other terms can be estimated easily, thus, it follows that
\begin{align*}
  \sup_{[0,T]}\int_\Omega \ri^2  \abs{\hd^r\curl \eta}^2dx
  \ls M_0+\delta\sup_{[0,T]}E_r(t)+CTP(\sup_{[0,T]}E_r(t)).
\end{align*}
The weighted estimates for the curl of $\D_t^{2m}\hd^{r-m} \eta$ can be obtained similarly.

From \eqref{eq.curl}, we see that
\begin{align*}
  \curleta \eta=t\curl u_0+\int_0^t\eps_{\cdot ji} A_{tj}^s\pd{\eta^i}{s}dt'+\int_0^t\int_0^{t'} \eps_{\cdot ji} A_{t'j}^s\pd{v^i}{s}dt''dt',
\end{align*}
and then by the fundamental theorem of calculus,
\begin{align*}
  \curl\eta =t\curl u_0-\eps_{\cdot ji}\int_0^t A_{tj}^sdt'\pd{\eta^i}{s}+\int_0^t\eps_{\cdot ji} A_{tj}^s\pd{\eta^i}{s}dt'+\int_0^t\int_0^{t'} \eps_{\cdot ji} A_{t'j}^s\pd{v^i}{s}dt''dt'.
\end{align*}
It follows that
\begin{align*}
  \norm{\curl\eta}_0^2\ls M_0+CTP(\sup_{[0,T]}E_r(t)),
\end{align*}
and
\begin{align*}
  \norm{\nb^{r-1}\curl\eta}_0^2\ls M_0+CTP(\sup_{[0,T]}E_r(t)),
\end{align*}
where we have used integration by parts with respect to time for the integrals involving $\nb^rv$ in order to  control them by $\norm{\eta}_r$ and $\norm{v}_{r-1}$. Thus, we get
\begin{align}
  \norm{\curl\eta}_{r-1}\ls M_0+CTP(\sup_{[0,T]}E_r(t)).
\end{align}

From \eqref{eq.curl}, we get, by the fundamental theorem of calculus,
\begin{align*}
  \curl v(t)=\curl u_0-\eps_{\cdot ji}\int_0^t A_{tj}^sdt'\pd{v^i}{s}+\eps_{\cdot ji}\int_0^t A_{tj}^s\pd{v^i}{s}dt',
\end{align*}
and, by taking the first order time derivative,
\begin{align*}
  \curl v_t(t)=-\eps_{\cdot ji}\int_0^t A_{tj}^sdt'\pd{v_t^i}{s}.
\end{align*}
Since $H^{r-2}(\Omega)$ is a multiplicative algebra for $r\gs 4$, we can directly estimate the $H^{r-2}(\Omega)$-norm of $\curl v_t$ to show that
\begin{align}
  \norm{\curl v_t(t)}_{r-2}^2\ls M_0+CTP(\sup_{[0,T]}E_r(t)).
\end{align}
The estimates for $\curl \D_t^{2m}\eta(t)$ in $H^{r-1-m}(\Omega)$ for $2\ls m\ls r-1$ follow from the same arguments.
\end{proof}

\section{The estimates for the horizontal derivatives}\label{sec.7}

We have the following estimates.

\begin{proposition}\label{prop.hori1}
  Let $r\in\{4,5\}$. For small $\delta>0$ and  the constant $M_0$ depending on $1/\delta$, it holds
  \begin{align*}
    &\sup_{[0,T]}\left[\norm{\ri^{1/2} \hd^r v(t)}_0^2+\norm{\ri  \hd^r\nb \eta(t)}_0^2+\norm{\ri\hd^r \dv \eta(t)}_0^2+\abs{\eta^\alpha}_{r-1/2}^2\right]\\
    \ls& M_0+\delta\sup_{[0,T]}E_r(t)+ CTP(\sup_{[0,T]} E_r(t)),
  \end{align*}
  where $\eta^\alpha=\eta\cdot T_\alpha$ for $\alpha=1,2$.
\end{proposition}

\begin{proof}
Letting $\hd^r$ act on \eqref{eu4.2}, then we have
\begin{align*}
\sum_{l=0}^rC_r^l \hd^{r-l}\ri \hd^l v_t^i&+\sum_{l=0}^rC_r^l \hd^{r-l}a^j_i\hd^l \pd{\left(\ri^2  J^{-2}\right)}{j}=0,
\end{align*}
where $C_r^l$ is the binomial coefficient. Taking the $L^2(\Omega)$-inner product with $\hd^r v^i$, we obtain
\begin{align}\label{eq.esthd4}
  \frac{1}{2}\frac{d}{dt}\int_\Omega \ri \abs{\hd^r v}^2dx +\sum_{m=0}^{3}\I_m
  =\sum_{m=4}^6\I_m.
\end{align}
Here,
\begin{align*}
  \I_1:=&\int_\Omega \hd^r a^j_i\pd{\left(\ri^2  J^{-2}\right)}{j}\hd^r v^idx, \quad
  \I_2:=\int_\Omega a^j_i\pd{\left(\hd^r\ri^2  J^{-2}\right)}{j}\hd^r v^idx,\\
  \I_3:=&\int_\Omega a^j_i\pd{\left(\ri^2  \hd^r J^{-2}\right)}{j}\hd^r v^idx,\quad
  \I_4:=-\sum_{l=0}^{r-1}C_r^l\int_\Omega \hd^{r-l}\ri \hd^l v_t^i\hd^r v^idx,\\
  \I_5:=&-\sum_{l=1}^{r-1}C_r^l\int_\Omega\hd^{r-l}a^j_i\hd^l \pd{\left(\ri^2  J^{-2}\right)}{j}\hd^r v^idx,\quad
  \I_6:=-\sum_{l=1}^{r-1}C_r^l\int_\Omega a^j_i\pd{\left(\hd^{r-l}\ri^2  \hd^l J^{-2}\right)}{j}\hd^r v^idx.
\end{align*}

\textbf{Step 1. Analysis of the integral $\I_1$}.
We use the identity \eqref{eq.Piola} to integrate by parts with respect to $x_j$  to find that
\begin{align*}
  \I_1=&-\int_\Omega \hd^r a^j_i\ri^2  J^{-2}\pd{\hd^r v^i}{j}dx +\int_{\Gamma_0} \hd^r a^3_i\hd^r v^i\ri^2  J^{-2}dx_1dx_2
  =-\int_\Omega \hd^r a^j_i\pd{\hd^r v^i}{j}\ri^2  J^{-2}dx,
\end{align*}
due to the boundary conditions \eqref{eu4.6} and \eqref{eu4.7}.

It follows from \eqref{eq.ahd} that
\begin{align}
  \I_1=&-\int_\Omega \hd^{r-1} \left(\hd \pd{\eta^l}{k}J^{-1}(a^j_ia^k_l-a^k_ia^j_l)\right)\pd{\hd^r v^i}{j}\ri^2  J^{-2}dx\no\\
  =&-\int_\Omega\hd^r \pd{\eta^l}{k} A^k_l\pd{\hd^r v^i}{j}A^j_i\ri^2  J^{-1}dx\label{eq.1a1}\\
  &+\int_\Omega \hd^r \pd{\eta^l}{k}A^k_i\pd{\hd^r v^i}{j}A^j_l\ri^2  J^{-1}dx \label{eq.1b1}\\
  &-\sum_{s=1}^{r-1} C_{r-1}^s\int_\Omega \hd^{r-s}\pd{\eta^l}{k}\hd^s\left(J^{-1}(a^j_ia^k_l-a^k_ia^j_l)\right)\pd{\hd^r v^i}{j}\ri^2  J^{-2}dx. \label{eq.1c1}
\end{align}
Since $v=\eta_t$, we get
\begin{align*}
  \eqref{eq.1a1}=&-\frac{1}{2}\frac{d}{dt}\int_\Omega\abs{\dveta\hd^r \eta}^2\ri^2  J^{-1}dx+\int_\Omega\hd^r \pd{\eta^l}{k} A^k_l\pd{\hd^r \eta^i}{j}\D_t A^j_i\ri^2 J^{-1}dx\\
  &+\frac{1}{2}\int_\Omega \abs{\dveta\hd^r \eta}^2\ri^2  \D_tJ^{-1}dx\\
  =&-\frac{1}{2}\frac{d}{dt}\int_\Omega\abs{\dveta\hd^r \eta}^2\ri^2  J^{-1}dx+\frac{1}{2}\int_\Omega\hd^r \pd{\eta^l}{k}\pd{\hd^r \eta^i}{j}\D_t (A^k_l A^j_i)\ri^2  J^{-1}dx\\
  &+\frac{1}{2}\int_\Omega \abs{\dveta\hd^r \eta}^2\ri^2  \D_tJ^{-1}dx.
\end{align*}

For the integral \eqref{eq.1b1}, since $v=\eta_t$, it holds
\begin{align*}
\hd^r\pd{\eta^l}{k}A_i^kA_l^j \hd^r \pd{v^i}{j}=&\D_t(\hd^r\pd{\eta^l}{k}A_i^k \hd^r \pd{\eta^i}{j}A_l^j)-\hd^r\pd{v^l}{k}A_i^k \hd^r \pd{\eta^i}{j}A_l^j-\hd^r\pd{\eta^l}{k} \hd^r \pd{\eta^i}{j}\D_t(A_i^k A_l^j).
\end{align*}
It follows from \eqref{eq.curletaF2} that
\begin{align}\label{eq.etavaa1}
  \hd^r\pd{\eta^l}{k}A_i^kA_l^j \hd^r \pd{v^i}{j}=&\frac{1}{2}\D_t\left[\abs{\nb_\eta \hd^r\eta}^2-\abs{\curleta \hd^r \eta}^2\right]-\frac{1}{2}\hd^r\pd{\eta^l}{k} \hd^r \pd{\eta^i}{j}\D_t(A_i^k A_l^j).
\end{align}
Thus, we have
\begin{align*}
  \eqref{eq.1b1}=&\frac{1}{2}\frac{d}{dt}\int_\Omega \left[\abs{\nb_\eta \hd^r\eta}^2-\abs{\curleta \hd^r \eta}^2\right]\ri^2  J^{-1}dx\\
  &+\frac1{2}\int_\Omega \left[\abs{\nb_\eta \hd^r\eta}^2-\abs{\curleta \hd^r \eta}^2\right]\ri^2 J^{-1}\dveta vdx\\
  &-\frac{1}{2}\int_\Omega \hd^r\pd{\eta^l}{k} \hd^r \pd{\eta^i}{j}\D_t(A_i^k A_l^j)\ri^2  J^{-1}dx.
\end{align*}
It follows that
\begin{align*}
  \int_0^T\I_1dt=&\frac{1}{2}\int_\Omega \left[\abs{\nb_\eta \hd^r\eta(T)}^2-\abs{\dveta \hd^r \eta(T)}^2-\abs{\curleta \hd^r \eta(T)}^2\right]\ri^2  J^{-1}(T)dx\\
&+\frac{2-1 }{2}\intt  \left[\abs{\nb_\eta \hd^r\eta}^2-\abs{\dveta \hd^r \eta}^2-\abs{\curleta \hd^r \eta}^2\right]\ri^2  J^{-1}\dveta vdxdt\\
  &-\frac{1}{2}\intt  \hd^r\pd{\eta^l}{k} \hd^r \pd{\eta^i}{j}\D_t(A_i^k A_l^j-A_i^jA_l^k)\ri^2  J^{-1}dxdt +\int_0^T\eqref{eq.1c1}dt.
\end{align*}

It is clear that
\begin{align*}
  &\labs{\intt  \left[\abs{\nb_\eta \hd^r\eta}^2-\abs{\dveta \hd^r \eta}^2-\abs{\curleta \hd^r \eta}^2\right]\ri^2  J^{-1}\dveta vdxdt}\\
  \ls& CT\sup_{[0,T]}\norm{\ri  \hd^r\nb\eta}_0^2\norm{v}_3\norm{\eta}_3^4
  \ls CTP(\sup_{[0,T]}E_r(t)),
\end{align*}
and
\begin{align*}
  &\labs{\intt  \hd^r\pd{\eta^l}{k} \hd^r \pd{\eta^i}{j}\D_t(A_i^k A_l^j-A_i^jA_l^k)\ri^2  J^{-1}dxdt}\\
  \ls& CT\sup_{[0,T]}\norm{\ri  \hd^r\nb\eta}_0^2\norm{v}_3\norm{\eta}_3^6\\
  \ls &CTP(\sup_{[0,T]}E_r(t)).
\end{align*}

Now, we analyze the integral $\int_0^T\eqref{eq.1c1}dt$. We will use integration by parts in time for the cases $s=1$ and $s=2$, while we have to use integration by parts with respect to spatial variables for the case $s=r-1$.

\emph{Case 1: $s=1$.}
From integration by parts with respect to time,  we get
\begin{align}
 &-(r-1)\intt  \hd^{r-1}\pd{\eta^l}{k}\hd \left(a^j_iA^k_l-A^k_ia^j_l\right)\pd{\hd^r v^i}{j}\ri^2  J^{-2}dxdt\label{eq.1c10}\\
  =&-(r-1)\int_\Omega\hd^{r-1}\pd{\eta^l}{k}\hd \left(a^j_iA^k_l-A^k_ia^j_l\right)\pd{\hd^r \eta^i}{j}\ri^2  J^{-2}dx\Big|_{t=T}\label{eq.1c11}\\
  &+(r-1)\intt  \hd^{r-1}\pd{v^l}{k}\hd \left(a^j_iA^k_l-A^k_ia^j_l\right)\pd{\hd^r \eta^i}{j}\ri^2  J^{-2}dxdt\label{eq.1c12}\\
  &+(r-1)\intt  \hd^{r-1}\pd{\eta^l}{k}\D_t\hd \left(a^j_iA^k_l-A^k_ia^j_l\right) \pd{\hd^r \eta^i}{j}\ri^2  J^{-2}dxdt\label{eq.1c13}\\
  &+(r-1)\intt  \hd^{r-1}\pd{\eta^l}{k}\hd \left(a^j_iA^k_l-A^k_ia^j_l\right)\pd{\hd^r \eta^i}{j}\ri^2 \D_t J^{-2}dxdt.\label{eq.1c14}
\end{align}

It is clear that, by H\"older's inequality and the fundamental theorem of calculus twice,
\begin{align*}
 \abs{\eqref{eq.1c11}}\ls& C \lnorm{\ri  \int_0^T\hd^{r-1}\D_t\nb\eta dt'}_0\norm{\ri  \hd^r\nb\eta(T)}_0\norm{\eta(T)}_4^7\\
  \ls & CT\left(\sup_{[0,T]}\lnorm{\ri \int_0^t\hd^{r-1}\D_t^2\nb\eta dt'}_0+\norm{\ri \hd^{r-1}\D_t\nb\eta(0)}_0\right)\norm{\ri   \hd^r\nb\eta(T)}_0\norm{\eta(T)}_4^7\\
  \ls & C T^2\sup_{[0,T]}\left(\norm{\ri  \hd^{r-1}\D_t^2\nb\eta }_0+\sum_{l=1}^2\norm{\ri  \hd^{r-1}\D_t^l\nb\eta(0)}_0\right) \norm{\ri \hd^r\nb\eta(T)}_0 \norm{\eta(T)}_4^7\\
  \ls& M_0+ CT^2P(\sup_{[0,T]}E_r(t)),
\end{align*}
\begin{align*}
  \abs{\eqref{eq.1c12}}\ls & CT\sup_{[0,T]}\left(\lnorm{\ri  \int_0^T\hd^{r-1}\D_t^2\nb\eta dt}_0+\norm{\ri  \hd^{r-1}\D_t\nb\eta(0)}_0\right)\norm{\ri  \hd^r\nb\eta}_0 \norm{\eta}_4^7\\
  \ls &CT^2\sup_{[0,T]}\left(\norm{\ri  \hd^{r-1}\D_t^2\nb\eta }_0+\sum_{l=1}^2\norm{\ri  \hd^{r-1}\D_t^l\nb\eta(0)}_0\right)\norm{\ri  \hd^r\nb\eta}_0 \norm{\eta}_4^7\\
  \ls &M_0+CT^2P(\sup_{[0,T]}E_r(t)),
\end{align*}
and
\begin{align*}
  \abs{\eqref{eq.1c14}}\ls & CT\norm{\ri  }_2 \sup_{[0,T]}\norm{\hd^{r-1}\nb\eta}_0P(\norm{\eta}_4)\norm{\nb v}_2\norm{\ri  \hd^r\nb\eta}_0
  \ls M_0+CT^2P(\sup_{[0,T]}E_r(t)).
\end{align*}

We can rewrite \eqref{eq.1c13} as, for $\beta\in\{1,2\}$
\begin{align}
  \eqref{eq.1c13}=&3\intt  \hd^{r-1}\pd{\eta^l}{\beta}\D_t\hd \left(a^3_iA^\beta_l-A^\beta_ia^3_l\right)J^{-2} \pd{\hd^r \eta^i}{3}\ri^2 dxdt\label{eq.1c15}\\
  &+3\intt  \hd^{r-1}\pd{\eta^l}{k}\D_t\hd \left(a^\beta_iA^k_l-A^k_ia^\beta_l\right)J^{-2} \pd{\hd^r \eta^i}{\beta}\ri^2 dxdt.\label{eq.1c16}
\end{align}
Obviously, we have from the H\"older inequality and Sobolev's embedding theorem that
\begin{align*}
  \abs{\eqref{eq.1c15}}\ls& CT\sup_{[0,T]}\norm{\ri  \hd^r\eta}_1P(\norm{\eta}_4, \norm{v}_3,\norm{\hd\nb v}_1)\norm{\ri  \hd^r\nb\eta}_0\\
  \ls &CT\sup_{[0,T]}\left(\norm{\ri }_3\norm{\eta}_r +\norm{\ri \hd^r\nb\eta}_0\right)P(\norm{\eta}_4, \norm{v}_3,\norm{\hd\nb v}_1)\norm{\ri  \hd^r\nb\eta}_0\\
  \ls& M_0+CTP(\sup_{[0,T]}E_r(t)).
\end{align*}
By integration by parts, it yields
\begin{align}
  \eqref{eq.1c16}=&-3\intt  \hd^{r-1}\pd{\eta^l}{k\beta}\D_t\hd \left(a^\beta_iA^k_l-A^k_ia^\beta_l\right) \hd^r \eta^i J^{-2}\ri^2 dxdt\label{eq.1c17}\\
  &-3\intt  \hd^{r-1}\pd{\eta^l}{k}\D_t\hd \pd{\left(a^\beta_iA^k_l-A^k_ia^\beta_l\right)}{\beta} \hd^r \eta^i J^{-2}\ri^2 dxdt\label{eq.1c18}\\
  &-3\intt  \hd^{r-1}\pd{\eta^l}{k}\D_t\hd \left(a^\beta_iA^k_l-A^k_ia^\beta_l\right) \hd^r \eta^i \pd{\left(J^{-2}\ri^2 \right)}{\beta}dxdt.\label{eq.1c19}
\end{align}
It is easy to see that
\begin{align*}
  \abs{\eqref{eq.1c17}}\ls& CT\sup_{[0,T]} \norm{\ri \hd^r\nb\eta}_0\norm{\hd\nb v }_1P(\norm{\eta}_4)\norm{\ri \hd^r\eta}_1 \ls M_0+CTP(\sup_{[0,T]}E_r(t)),
\end{align*}
and
\begin{align*}
  \eqref{eq.1c19}\ls& CT\sup_{[0,T]} \norm{\hd^{r-1}\nb\eta}_0\norm{\hd\nb v }_1P(\norm{\eta}_4)\norm{\ri \hd^r\eta}_1 \left(\norm{ \ri }_{L^\infty(\Omega)}+\norm{\hd \ri }_{L^\infty(\Omega)}\right)\\
   \ls& M_0+CTP(\sup_{[0,T]}E_r(t)).
\end{align*}

In order to estimate \eqref{eq.1c18}, we first consider the estimate of the $\norm{D^{r-1} v}_{L^3([0,T]\times\Omega)}$ where $D^{r-1}$ denotes all the derivatives $\D^{\theta}$ for the multi-index $\theta=(\theta_1,\theta_2,\theta_3)$ and $0\ls\abs{\theta}\ls r-1$. By Proposition \ref{prop.interpolation} with $F=D^{r-1}\eta$ and the Sobolev embedding theorem, we have
\begin{align*}
  \norm{D^{r-1} v}_{L^3([0,T]\times\Omega)}^2
  \ls &CT^{2/3}\Big[\norm{D^{r-1}v(0)}_{L^3(\Omega)}+\sup_{[0,T]}\norm{D^{r-1}v_t}_{L^2(\Omega)} \norm{D^{r-1}\eta}_{L^6(\Omega)}\Big]\\
  \ls&M_0+CT^{2/3}\sup_{[0,T]}\norm{v_t}_{r-1}\norm{\eta}_r\\
  \ls &M_0+CT^{2/3}P(\sup_{[0,T]}E_r(t)).
\end{align*}
Thus, we obtain
\begin{align}\label{eq.vL3est}
  \norm{D^{r-1} v}_{L^3([0,T]\times\Omega)}^2\ls &M_0+CT^{2/3}P(\sup_{[0,T]}E_r(t)).
\end{align}

By the H\"older inequality, the Sobolev embedding theorem, the Cauchy inequality and \eqref{eq.vL3est}, we easily get
\begin{align*}
  \abs{\eqref{eq.1c18}}\ls &C T^{2/3}\norm{\ri }_2\norm{\hd^2\nb v}_{L^3([0,T]\times\Omega)}\sup_{[0,T]}\norm{\hd^3\nb\eta}_0 \norm{\ri \hd^4\eta}_1\\
  \ls&CT^{1/3}\norm{\hd^2\nb v}_{L^3([0,T]\times\Omega)}^2+CT\norm{\ri }_2^2\sup_{[0,T]}\norm{\hd^3\nb\eta}_0^2 \norm{\ri \hd^4\eta}_1^2\\
  \ls&CT^{1/3}\left(M_0+CT^{2/3}P(\sup_{[0,T]}E_r(t))\right)+M_0+CT P(\sup_{[0,T]}E_r(t))\\
  \ls&M_0+CT P(\sup_{[0,T]}E_r(t)).
\end{align*}
Hence, we obtain
\begin{align*}
  \abs{\eqref{eq.1c10}}\ls&M_0+CT P(\sup_{[0,T]}E_r(t)).
\end{align*}

\emph{Case 2: $s=2$}. By integration by parts with respect to time, it yields
\begin{align}
  &-C_{r-1}^2\intt  \hd^{r-2}\pd{\eta^l}{k}\hd^2(A^j_ia^k_l-a^k_iA^j_l)\pd{\hd^r v^i}{j}\ri^2  J^{-2}dxdt\label{eq.1c158}\\
  =&-C_{r-1}^2\int_\Omega \hd^{r-2}\pd{\eta^l}{k}\hd^2(A^j_ia^k_l-a^k_iA^j_l)\pd{\hd^r \eta^i}{j}\ri^2  J^{-2}dx\Big|_{t=T}\label{eq.1c159}\\
  &+C_{r-1}^2\intt  \hd^{r-2}\pd{v^l}{k}\hd^2(A^j_ia^k_l-a^k_iA^j_l)\pd{\hd^r \eta^i}{j}\ri^2  J^{-2}dxdt\label{eq.1c160}\\
  &+C_{r-1}^2\intt  \hd^{r-2}\pd{\eta^l}{k}\hd^2\D_t(A^j_ia^k_l-a^k_iA^j_l)\pd{\hd^r \eta^i}{j}\ri^2  J^{-2}dxdt\label{eq.1c161}\\
  &+C_{r-1}^2\intt  \hd^{r-2}\pd{\eta^l}{k}\hd^2(A^j_ia^k_l-a^k_iA^j_l)\pd{\hd^r \eta^i}{j}\ri^2  \D_tJ^{-2}dxdt.\label{eq.1c162}
\end{align}
Applying the fundamental theorem of calculus yields for a small $\delta>0$
\begin{align*}
  \abs{\eqref{eq.1c159}}\ls & C\norm{\ri }_2\norm{\hd^{r-2}\nb\eta(T)}_1P(\norm{\eta(T)}_4)\norm{\nb\eta(T)}_1 \norm{\ri \hd^r\nb\eta(T)}_0\\
  \ls &C\norm{\ri }_2P(\norm{\eta(T)}_4,\norm{\eta(T)}_r)\left(\lnorm{\int_0^T\nb v dt}_1+\norm{\nb\eta(0)}_1\right) \norm{\ri \hd^r\nb\eta(T)}_0\\
  \ls &M_0+\delta\sup_{[0,T]}E_r(t) +CTP(\sup_{[0,T]}E_r(t)).
\end{align*}
Similar to \eqref{eq.1c18}, we can get
\begin{align*}
  \abs{\eqref{eq.1c160}}+\abs{\eqref{eq.1c161}}\ls&M_0+\delta\sup_{[0,T]}E_r(t)+CT P(\sup_{[0,T]}E_r(t)).
\end{align*}

By an $L^6$-$L^6$-$L^6$-$L^2$ H\"older inequality and the Sobolev embedding theorem, we get
\begin{align*}
  \abs{\eqref{eq.1c162}} \ls& M_0+CTP(\sup_{[0,T]}E_r(t)).
\end{align*}
Therefore, we have obtained
\begin{align*}
  \abs{\eqref{eq.1c158}} \ls& M_0+\delta\sup_{[0,T]}E_r(t)+CTP(\sup_{[0,T]}E_r(t)).
\end{align*}

\emph{Case 3: $s=r-1$}. We write the space-time integral as, for $\beta\in\{1,2\}$
\begin{align}
  &-\intt  \hd \pd{\eta^l}{k}\hd^{r-1}(a^j_iA^k_l-A^k_ia^j_l)\pd{\hd^r v^i}{j}\ri^2  J^{-2}dxdt\label{eq.1c117}\\
  =&-\intt  \hd \pd{\eta^l}{k}\hd^{r-1}(a^\beta_iA^k_l-A^k_ia^\beta_l)\pd{\hd^r v^i}{\beta}\ri^2  J^{-2}dxdt\label{eq.1c118}\\
  &-\intt  \hd \pd{\eta^l}{\beta}\hd^{r-1}(a^3_iA^\beta_l-A^\beta_ia^3_l)\pd{\hd^r v^i}{3}\ri^2  J^{-2}dxdt.\label{eq.1c119}
\end{align}
By integration by parts with respect to $x_\beta$, we have
\begin{align}
  \eqref{eq.1c118}=&\intt  \hd \pd{\eta^l}{k\beta}\hd^{r-1}(a^\beta_iA^k_l-A^k_ia^\beta_l)\hd^r v^i\ri^2  J^{-2}dxdt\label{eq.1c120}\\
  &+\intt  \hd \pd{\eta^l}{k}\hd^{r-1}(\pd{a^\beta_i}{\beta}A^k_l-A^k_i\pd{a^\beta_l}{\beta})\hd^r v^i\ri^2  J^{-2}dxdt\label{eq.1c121}\\
  &+\intt  \hd \pd{\eta^l}{k}\hd^{r-1}(a^\beta_i \pd{A^k_l}{\beta}-\pd{A^k_i}{\beta}a^\beta_l)\hd^r v^i\ri^2  J^{-2}dxdt\label{eq.1c122}\\
  &+\intt  \hd \pd{\eta^l}{k}\hd^{r-1}(a^\beta_iA^k_l-A^k_ia^\beta_l)\hd^r v^i\pd{(\ri^2)}{\beta} J^{-2}dxdt\label{eq.1c123}\\
  &+\intt  \hd \pd{\eta^l}{k}\hd^{r-1}(a^\beta_iA^k_l-A^k_ia^\beta_l)\hd^r v^i\ri^2 \pd{J^{-2}}{\beta} dxdt.\label{eq.1c124}
\end{align}
From \eqref{eq.ahd}, it follows that
\begin{align}
  \eqref{eq.1c120}=&\intt  \hd \pd{\eta^l}{k\beta}\hd^{r-2}\left[\pd{\hd\eta^p}{q}J A^\beta_i(A^q_pA^k_l-A^k_pA^q_l)\right] \hd^r v^i\ri^2  J^{-2}dxdt\label{eq.1c125}\\
  &+\intt  \hd \pd{\eta^l}{k\beta}\hd^{r-2}\left[\pd{\hd\eta^p}{q}J A^\beta_l(A^k_pA^q_i-A^q_pA^k_i)\right] \hd^r v^i\ri^2  J^{-2}dxdt\label{eq.1c126}\\
  &+\intt  \hd \pd{\eta^l}{k\beta}\hd^{r-2}\left[\pd{\hd\eta^p}{q}J A^\beta_p(A^k_iA^q_l-A^q_iA^k_l)\right] \hd^r v^i\ri^2  J^{-2}dxdt.\label{eq.1c127}
\end{align}
We can write
\begin{align}
  \eqref{eq.1c125}=&\intt  \hd \pd{\eta^l}{k\beta}\hd^{r-2}\left[\pd{\hd\eta^p}{\alpha}J A^\beta_i(A^\alpha_pA^k_l-A^k_pA^\alpha_l)\right] \hd^r v^i\ri^2  J^{-2}dxdt\no\\
  &+\intt  \hd \pd{\eta^l}{\alpha\beta}\hd^{r-2}\left[\pd{\hd\eta^p}{3}J A^\beta_i(A^3_pA^\alpha_l-A^\alpha_pA^3_l)\right] \hd^r v^i\ri^2  J^{-2}dxdt\no\\
  =&\intt  \hd \pd{\eta^l}{k\beta}\pd{\hd^{r-1}\eta^p}{\alpha} A^\beta_i(A^\alpha_pA^k_l-A^k_pA^\alpha_l) \hd^r v^i\ri^2  J^{-1}dxdt\label{eq.1c128}\\
  &+\sum_{m=0}^{r-3}C_{r-2}^m\intt  \hd \pd{\eta^l}{k\beta}\pd{\hd^{m+1}\eta^p}{\alpha}\hd^{r-2-m}\left[J A^\beta_i(A^\alpha_pA^k_l-A^k_pA^\alpha_l)\right]\hd^r v^i\ri^2  J^{-2}dxdt\label{eq.1c129}\\
  &+\intt  \hd \pd{\eta^l}{\alpha\beta}\pd{\hd^{r-1}\eta^p}{3} A^\beta_i(A^3_pA^\alpha_l-A^\alpha_pA^3_l) \hd^r v^i\ri^2  J^{-1}dxdt\label{eq.1c130}\\
  &+\sum_{m=0}^{r-3}C_{r-2}^m\intt  \hd \pd{\eta^l}{\alpha\beta}\pd{\hd^{m+1}\eta^p}{3}\hd^{r-2-m}\left[J A^\beta_i(A^3_pA^\alpha_l-A^\alpha_pA^3_l)\right] \hd^r v^i\ri^2  J^{-2}dxdt.\label{eq.1c131}
\end{align}
It is easy to see that
\begin{align*}
  \abs{\eqref{eq.1c128}}\ls &CT\norm{\ri }_2^{1/2}\sup_{[0,T]}\norm{\hd^2\nb\eta}_1 \norm{\ri \hd^r\eta}_1P(\norm{\eta}_3) \norm{\ri^{1/2}\hd^rv}_0
  \ls M_0+CTP(\sup_{[0,T]}E_r(t)),
\end{align*}
and by an $L^6$-$L^6$-$L^6$-$L^2$ H\"older inequality and the Sobolev embedding theorem,
\begin{align*}
  \abs{\eqref{eq.1c129}}+\abs{\eqref{eq.1c131}}\ls &CT\norm{\ri }_2^{3/2}\sup_{[0,T]}P(\norm{\eta}_4,\norm{\eta}_r)\norm{\ri^{1/2}\hd^r v}_0
  \ls M_0+CTP(\sup_{[0,T]}E_r(t)).
\end{align*}
By using integration by parts with respect to time, we have
\begin{align}
  \eqref{eq.1c130}=&\int_\Omega \hd \pd{\eta^l}{\alpha\beta}\pd{\hd^{r-1}\eta^p}{3} A^\beta_i(A^3_pA^\alpha_l-A^\alpha_pA^3_l) \hd^r \eta^i\ri^2  J^{-1}dx\Big|_{t=T}\label{eq.1c132}\\
  &-\intt  \hd \pd{v^l}{\alpha\beta}\pd{\hd^{r-1}\eta^p}{3} A^\beta_i(A^3_pA^\alpha_l-A^\alpha_pA^3_l) \hd^r \eta^i\ri^2  J^{-1}dxdt\label{eq.1c133}\\
  &-\intt  \hd \pd{\eta^l}{\alpha\beta}\pd{\hd^{r-1}v^p}{3} A^\beta_i(A^3_pA^\alpha_l-A^\alpha_pA^3_l) \hd^r \eta^i\ri^2  J^{-1}dxdt\label{eq.1c134}\\
  &-\intt  \hd \pd{\eta^l}{\alpha\beta}\pd{\hd^{r-1}\eta^p}{3} \D_t\left[A^\beta_i(A^3_pA^\alpha_l-A^\alpha_pA^3_l)J^{-1}\right] \hd^r \eta^i\ri^2  dxdt.\label{eq.1c135}
\end{align}
Obviously, it yields by using the fundamental theorem of calculus twice
\begin{align*}
  &\eqref{eq.1c132}\ls C\left(\norm{\ri }_3\norm{\eta}_3+\lnorm{\ri \int_0^T\hd^3\D_t\nb\eta dt}_0\right)P(\norm{\eta(T)}_3,\norm{\eta(T)}_r)\norm{\ri \hd^r\eta(T)}_1\\
  \ls &C\norm{\ri }_3\norm{\eta}_3\\
  &+CT\left( \norm{\ri \hd^3\D_t\nb\eta(0)}_0+ \lnorm{\ri \int_0^T\hd^3\D_t^2\nb\eta dt}_0\right) P(\norm{\eta(T)}_3,\norm{\eta(T)}_r)\norm{\ri \hd^r\eta(T)}_1\\
  \ls &M_0+CTP(\sup_{[0,T]}E_r(t)).
\end{align*}
By the H\"older inequality, the Sobolev embedding theorem and the fundamental theorem of calculus, we get
\begin{align*}
  &\eqref{eq.1c133}\ls CT\sup_{[0,T]}\norm{\ri \hd^3\D_t\eta}_1\norm{\hd^{r-1}\nb\eta}_0 \norm{\ri \hd^r\eta}_1 P(\norm{\eta}_3)\\
  \ls&CT\sup_{[0,T]}\left(\norm{\ri }_3\norm{v}_3 +\norm{\ri \hd^3\D_t\nb\eta(0)}_0+\lnorm{\ri \int_0^T \hd^3\D_t^2\nb\eta}_0\right) \norm{\ri \hd^r\eta}_1 P(\norm{\eta}_3,\norm{\eta}_r)\\
  \ls &M_0+CTP(\sup_{[0,T]}E_r(t)).
\end{align*}
Similarly, we have
\begin{align*}
  \eqref{eq.1c134}\ls& CT\sup_{[0,T]}\norm{\hd^3 \eta}_1\norm{\ri \hd^{r-1}\D_t\nb\eta}_0 \norm{\ri \hd^r\eta}_1 P(\norm{\eta}_3)
  \ls M_0+CTP(\sup_{[0,T]}E_r(t)),
\end{align*}
and
\begin{align*}
  \eqref{eq.1c135}\ls& CT\norm{\ri }_2\sup_{[0,T]}\norm{\hd^3 \eta}_1\norm{\hd^{r-1}\nb\eta}_0 \norm{\nb v}_1 \norm{\ri \hd^r\eta}_1 P(\norm{\eta}_3)
  \ls M_0+CTP(\sup_{[0,T]}E_r(t)).
\end{align*}
We can deal with \eqref{eq.1c126} and \eqref{eq.1c127} as the same arguments as for \eqref{eq.1c125}. Thus, we obtain
\begin{align}
  \abs{\eqref{eq.1c120}}\ls M_0+CTP(\sup_{[0,T]}E_r(t)).
\end{align}

We write
\begin{align}
  \eqref{eq.1c121}=&\intt  \hd \pd{\eta^l}{k}\hd^{r-1}\pd{\eta^p}{q\beta}A^q_p(A^\beta_iA^k_l-A^k_iA^\beta_l) \hd^r v^i\ri^2  J^{-1}dxdt\label{eq.1c136}\\
  &+\intt  \hd \pd{\eta^l}{k}\hd^{r-1}\pd{\eta^p}{\alpha\beta}A^\beta_p(A^k_iA^\alpha_l-A^\alpha_iA^k_l) \hd^r v^i\ri^2  J^{-1}dxdt\label{eq.1c137}\\
  &+\intt  \hd \pd{\eta^l}{\alpha}\hd^{r-1}\pd{\eta^p}{3\beta}A^\beta_p(A^\alpha_iA^3_l-A^3_iA^\alpha_l) \hd^r v^i\ri^2  J^{-1}dxdt\label{eq.1c138}\\
  &+\sum_{m=0}^{r-2}C_{r-1}^m\intt  \hd \pd{\eta^l}{k}\hd^m\pd{\eta^p}{q\beta}\hd^{r-1-m}\left[JA^q_p(A^\beta_iA^k_l-A^k_iA^\beta_l) \right] \hd^r v^i\ri^2  J^{-2}dxdt\label{eq.1c139}\\
  &+\sum_{m=0}^{r-2}C_{r-1}^m\intt \hd \pd{\eta^l}{k}\hd^m\pd{\eta^p}{\alpha\beta}\hd^{r-1-m}\left[JA^\beta_p(A^k_iA^\alpha_l-A^\alpha_iA^k_l) \right]\hd^r v^i\ri^2  J^{-2}dxdt\label{eq.1c140}\\
  &+\sum_{m=0}^{r-2}C_{r-1}^m\intt \hd \pd{\eta^l}{\alpha}\hd^m\pd{\eta^p}{3\beta}\hd^{r-1-m}\left[JA^\beta_p(A^\alpha_iA^3_l-A^3_iA^\alpha_l) \right]\hd^r v^i\ri^2  J^{-2}dxdt.\label{eq.1c141}
\end{align}
It is easy to see that
\begin{align*}
  \abs{\eqref{eq.1c136}}+\abs{\eqref{eq.1c137}}+\abs{\eqref{eq.1c138}}\ls&CT
  \norm{\ri }_2^{1/2}\sup_{[0,T]} \norm{\eta}_4\norm{\ri \hd^r\nb\eta}_0
  P(\norm{\eta}_3)\norm{\ri^{1/2}\hd^rv}_0\\
  \ls& M_0+CTP(\sup_{[0,T]}E_r(t)).
\end{align*}
By the H\"older inequality and the Sobolev embedding theorem, we have
\begin{align*}
  \abs{\eqref{eq.1c139}}+\abs{\eqref{eq.1c140}}+\abs{\eqref{eq.1c141}}
  \ls &CT\norm{\ri }_2^{3/2}\sup_{[0,T]}\norm{\ri^{1/2}\hd^rv}_0 P(\norm{\eta}_4,\norm{\eta}_r)\\
  \ls& M_0+CTP(\sup_{[0,T]}E_r(t)).
\end{align*}
Hence,
\begin{align}
  \abs{\eqref{eq.1c121}}\ls &M_0+CTP(\sup_{[0,T]}E_r(t)).
\end{align}

By \eqref{eq.Ad}, we have
\begin{align}
  \eqref{eq.1c122}=&\intt  \hd \pd{\eta^l}{k}\hd^{r-1}\left(\pd{\eta^p}{q\beta}JA^k_p(A^q_iA^\beta_l-A^\beta_iA^q_l)\right) \hd^r v^i\ri^2  J^{-2}dxdt\no\\
  =&\intt  \hd \pd{\eta^l}{k}\hd^{r-1}\pd{\eta^p}{q\beta}A^k_p(A^q_iA^\beta_l-A^\beta_iA^q_l)\hd^r v^i\ri^2  J^{-1}dxdt\label{eq.1c142}\\
  &+\sum_{m=0}^{r-2}C_{r-1}^m\intt  \hd \pd{\eta^l}{k}\hd^m\pd{\eta^p}{q\beta}\hd^{r-1-m} \left(JA^k_p(A^q_iA^\beta_l-A^\beta_iA^q_l)\right)\hd^r v^i\ri^2  J^{-2}dxdt.\label{eq.1c143}
\end{align}
By the H\"older inequality and the Sobolev embedding theorem, we have
\begin{align*}
  \abs{\eqref{eq.1c142}}\ls&CT\sup_{[0,T]}\norm{\eta}_4\norm{\ri \hd^r\nb\eta}_0 P(\norm{\eta}_3)\norm{\ri^{1/2}\hd^rv}_0\norm{\ri }_2^{1/2}
  \ls M_0+CTP(\sup_{[0,T]}E_r(t)),
\end{align*}
and similar to \eqref{eq.1c139}
\begin{align*}
  \abs{\eqref{eq.1c143}}\ls&M_0+CTP(\sup_{[0,T]}E_r(t)).
\end{align*}
Thus,
\begin{align}
  \abs{\eqref{eq.1c122}}\ls &M_0+CTP(\sup_{[0,T]}E_r(t)).
\end{align}

For \eqref{eq.1c123} and \eqref{eq.1c124}, it is easy to have
\begin{align*}
  \abs{\eqref{eq.1c123}}+\abs{\eqref{eq.1c124}}
  \ls&CT(\norm{\ri }_3^{3/2}+\norm{\ri }_2^{3/2})\sup_{[0,T]} \norm{\ri^{1/2}\hd^rv}_0P(\norm{\eta}_4,\norm{\eta}_r)\\
  \ls& M_0+CTP(\sup_{[0,T]}E_r(t)).
\end{align*}
Hence,
\begin{align}
  \abs{\eqref{eq.1c118}}\ls &M_0+CTP(\sup_{[0,T]}E_r(t)).
\end{align}

By integration by parts with respect to time, it holds
\begin{align}
  \eqref{eq.1c119}=&-\int_\Omega \hd \pd{\eta^l}{\beta}\hd^{r-1}(a^3_iA^\beta_l-A^\beta_ia^3_l)\pd{\hd^r \eta^i}{3}\ri^2  J^{-2}dx\Big|_{t=T}\label{eq.1c144}\\
  &+\intt  \hd \pd{v^l}{\beta}\hd^{r-1}(a^3_iA^\beta_l-A^\beta_ia^3_l)\pd{\hd^r \eta^i}{3}\ri^2  J^{-2}dxdt\label{eq.1c145}\\
  &+\intt  \hd \pd{\eta^l}{\beta}\hd^{r-1}\D_t(a^3_iA^\beta_l-A^\beta_ia^3_l)\pd{\hd^r \eta^i}{3}\ri^2  J^{-2}dxdt\label{eq.1c146}\\
  &+\intt  \hd \pd{\eta^l}{\beta}\hd^{r-1}(a^3_iA^\beta_l-A^\beta_ia^3_l)\pd{\hd^r \eta^i}{3}\ri^2  \D_tJ^{-2}dxdt.\label{eq.1c147}
\end{align}

It is easy to see that by applying the fundamental theorem of calculus three times
\begin{align*}
  \abs{\eqref{eq.1c144}}\ls&C\norm{\eta(T)}_4\norm{\ri \hd^r\nb\eta(T)}_0 P(\norm{\eta}_3) \\
  &\cdot\left(\lnorm{\ri \int_0^T\hd^{r-1}\D_t\nb\eta dt}_0 +\norm{\ri }_2\norm{\hd^{r-2}\nb\eta}_1\lnorm{\int_0^T\hd\nb v dt}_1\right.\\
  &\qquad+\norm{\ri }_2\norm{\hd^{r-3}\nb\eta}_1\lnorm{\int_0^T\hd^2\nb vdt}_{L^3(\Omega)}\\
  &\qquad\left.+(1-\sgn(5-r))\norm{\ri }_2\norm{\hd^3\nb\eta }_1\lnorm{\int_0^T \hd^{r-4}\nb v dt}_1\right)\\
  \ls &M_0+CTP(\sup_{[0,T]}E_r(t)).
\end{align*}

By \eqref{eq.ahd} and \eqref{eq.Ad}, we have
\begin{align}
  \eqref{eq.1c145}=&\intt  \hd \pd{v^l}{\beta}\hd^{r-2}\left(\hd\pd{\eta^p}{\alpha} JA^\beta_l(A^3_iA^\alpha_p-A^\alpha_ia^3_p)\right)\pd{\hd^r \eta^i}{3}\ri^2  J^{-2}dxdt\label{eq.1c148}\\
  &+\intt  \hd \pd{v^l}{\beta}\hd^{r-2}\left(\hd\pd{\eta^p}{\alpha} JA^\beta_p(A^3_lA^\alpha_i-A^\alpha_la^3_i)\right)\pd{\hd^r \eta^i}{3}\ri^2  J^{-2}dxdt\label{eq.1c149}\\
  &+\intt  \hd \pd{v^l}{\beta}\hd^{r-2}\left(\hd\pd{\eta^p}{\alpha} JA^\beta_i(A^3_pA^\alpha_l-A^\alpha_pa^3_l)\right)\pd{\hd^r \eta^i}{3}\ri^2  J^{-2}dxdt.\label{eq.1c150}
\end{align}
We split \eqref{eq.1c148} into two integrals, i.e.,
\begin{align}
  \eqref{eq.1c148}=&\intt  \hd \pd{v^l}{\beta}\hd^{r-1}\pd{\eta^p}{\alpha} JA^\beta_l(A^3_iA^\alpha_p-A^\alpha_ia^3_p)\pd{\hd^r \eta^i}{3}\ri^2  J^{-1}dxdt\label{eq.1c151}\\
  &+\sum_{m=0}^{r-3}C_{r-2}^m\intt  \hd \pd{v^l}{\beta}\hd^{m+1}\pd{\eta^p}{\alpha}\hd^{r-2-m}\left( JA^\beta_l(A^3_iA^\alpha_p-A^\alpha_ia^3_p)\right)\pd{\hd^r \eta^i}{3}\ri^2  J^{-2}dxdt.\label{eq.1c152}
\end{align}
Obviously, we see that
\begin{align*}
  \abs{\eqref{eq.1c151}}\ls&CT\sup_{[0,T]}\norm{\hd^2v}_1\norm{\ri \hd^r\eta}_1 P(\norm{\eta}_3)\norm{\ri \hd^r\nb\eta}_0
  \ls M_0+CTP(\sup_{[0,T]}E_r(t)),
\end{align*}
and
\begin{align*}
  \abs{\eqref{eq.1c152}}\ls &CT\norm{\ri }_2\sup_{[0,T]} \norm{\hd^2 v}_1P(\norm{\eta}_r)\norm{\ri \hd^r\nb\eta}_0
  \ls M_0+CTP(\sup_{[0,T]}E_r(t)).
\end{align*}
Both \eqref{eq.1c149} and \eqref{eq.1c150} can be dealt with as the same argument as for \eqref{eq.1c148}. Thus, we obtain
\begin{align*}
  \abs{\eqref{eq.1c145}}\ls &M_0+CTP(\sup_{[0,T]}E_r(t)).
\end{align*}

By \eqref{eq.adt} and \eqref{eq.At}, we have
\begin{align}
  \eqref{eq.1c146}=&\intt  \hd \pd{\eta^l}{\beta}\hd^{r-1}\left[\pd{v^p}{\alpha}JA^\beta_l(A^3_iA^\alpha_p-A^\alpha_ia^3_p) \right]\pd{\hd^r \eta^i}{3}\ri^2  J^{-2}dxdt\label{eq.1c153}\\
  &+\intt  \hd \pd{\eta^l}{\beta}\hd^{r-1}\left[\pd{v^p}{\alpha}JA^\beta_p(A^3_lA^\alpha_i-A^\alpha_la^3_i) \right]\pd{\hd^r \eta^i}{3}\ri^2  J^{-2}dxdt\label{eq.1c154}\\
  &+\intt  \hd \pd{\eta^l}{\beta}\hd^{r-1}\left[\pd{v^p}{\alpha}JA^\beta_i(A^3_pA^\alpha_l-A^\alpha_pa^3_l) \right]\pd{\hd^r \eta^i}{3}\ri^2  J^{-2}dxdt.\label{eq.1c155}
\end{align}
We write
\begin{align}
  \eqref{eq.1c153}=&\intt  \hd \pd{\eta^l}{\beta}\hd^{r-1}\pd{v^p}{\alpha}A^\beta_l(A^3_iA^\alpha_p-A^\alpha_ia^3_p) \pd{\hd^r \eta^i}{3}\ri^2  J^{-1}dxdt\label{eq.1c156}\\
  &+\sum_{m=0}^{r-2}C_{r-1}^m\intt  \hd \pd{\eta^l}{\beta}\hd^m\pd{v^p}{\alpha}\hd^{r-1-m}\left[JA^\beta_l(A^3_iA^\alpha_p -A^\alpha_ia^3_p) \right]\pd{\hd^r \eta^i}{3}\ri^2  J^{-2}dxdt.\label{eq.1c157}
\end{align}
Then, by the same arguments as for \eqref{eq.1c138} and \eqref{eq.1c141}, we can estimate \eqref{eq.1c156} and \eqref{eq.1c157}, and then
\begin{align*}
  \abs{\eqref{eq.1c146}}\ls &M_0+CTP(\sup_{[0,T]}E_r(t)).
\end{align*}
It is easy to see that \eqref{eq.1c147} has the same bounds. Thus, we obtain the estimates of \eqref{eq.1c119} and then of \eqref{eq.1c117}, i.e.,
\begin{align*}
  \abs{\eqref{eq.1c117}}\ls &M_0+CTP(\sup_{[0,T]}E_r(t)).
\end{align*}

\emph{Case 4: $s=r-2$ and $r=5$.} Integration by parts with respect to time gives
\begin{align*}
  &\intt \hd^{2}\pd{\eta^l}{k}\hd^{3}\left(J^{-1}(a^j_ia^k_l-a^k_ia^j_l)\right)\pd{\hd^5 v^i}{j}\ri^2  J^{-2}dxdt\\
  =&\int_\Omega \hd^{2}\pd{\eta^l}{k}\hd^{3}\left(J^{-1}(a^j_ia^k_l-a^k_ia^j_l)\right)\pd{\hd^5 \eta^i}{j}\ri^2  J^{-2}dx\Big|_0^T\\
  &-\intt \hd^{2}\pd{v^l}{k}\hd^{3}\left(J^{-1}(a^j_ia^k_l-a^k_ia^j_l)\right)\pd{\hd^5 \eta^i}{j}\ri^2  J^{-2}dxdt\\
  &-\intt \hd^{2}\pd{\eta^l}{k}\D_t\hd^{3}\left(J^{-1}(a^j_ia^k_l-a^k_ia^j_l)\right)\pd{\hd^5 \eta^i}{j}\ri^2  J^{-2}dxdt\\
  &-\intt \hd^{2}\pd{\eta^l}{k}\hd^{3}\left(J^{-1}(a^j_ia^k_l-a^k_ia^j_l)\right)\pd{\hd^5 \eta^i}{j}\ri^2  \D_t J^{-2}dxdt,
\end{align*}
which can be controlled by $M_0+CTP(\sup_{[0,T]}E_5(t))$ from the H\"older inequality and the fundamental theorem of calculus.

Therefore, we obtain
\begin{align*}
  \labs{\int_0^T\eqref{eq.1c1}dt}\ls &M_0+\delta\sup_{[0,T]}E_r(t)+CTP(\sup_{[0,T]}E_r(t)).
\end{align*}

\textbf{Step 2. Analysis of the integral $\I_2$}.
Similar to those of $\I_1$, by \eqref{eu4.6}, \eqref{eu4.7} and $\eta_t=v$, we have
\begin{align*}
  \I_2=&-\int_\Omega a^j_i\hd^r\ri^2  J^{-2}\pd{\hd^r v^i}{j}dx +\int_{\Gamma_0}a^3_i\hd^r\ri^2  J^{-2}\hd^r v^idx_1dx_2
  =-\int_\Omega \hd^r\ri^2  A^j_i J^{-1}\pd{\hd^r \eta_t^i}{j}dx.
\end{align*}

Integration by parts shows that
\begin{align*}
  \int_0^T \I_2(t)dt=&\intt  \hd^r\ri^2   (A^j_i J^{-1})_t\pd{\hd^r \eta^i}{j}dx dt-\int_\Omega \hd^r\ri^2   A^j_i J^{-1}\pd{\hd^r \eta^i}{j}dx\Big|_{t=T}\\
  =&\intt  \hd^r\ri^2   (A^j_i J^{-1})_t\pd{\hd^r \eta^i}{j}dx dt\\
  &-\int_\Omega \hd^r\ri^2 \left( \delta_i^j+\int_0^T (A^j_i J^{-1})_tdt\right)\pd{\hd^r \eta^i}{j}(T) dx,
\end{align*}
which yields, by H\"older's inequality and Sobolev's embedding theorem, that
\begin{align*}
  \labs{\int_0^T \I_2(t)dt}\ls &C\lnorm{\frac{\hd^r\ri^2 }{\ri}}_{0}\left(T\sup_{[0,T]} \norm{\ri \hd^r\nb\eta}_0 \norm{v}_3\norm{\eta}_3^4+\norm{\ri \hd^r\dv\eta}_0\right)\\
  \ls &M_0+\delta\sup_{[0,T]}E_r(t)+CTP(\sup_{[0,T]}E_r(t)),
\end{align*}
where we require $\ri \in H^{\max(4,r)}(\Omega)$ because, for $r\ls 5$,
\begin{align}
  \lnorm{\frac{\hd^r\ri^2 }{\ri}}_{0}\ls&\sum_{m=0}^rC_r^m\lnorm{\frac{\hd^m \ri\hd^{r-m}\ri}{\ri}}_0\ls 2\sum_{m=0}^{\ceil{(r-1)/2}} C_r^m\lnorm{\frac{\hd^m \ri\hd^{r-m}\ri}{\ri}}_0\no\\
  \ls&C\norm{\hd^r\ri}_0+C\lnorm{\frac{\hd\ri}{\ri}}_2\norm{\hd^{r-1}\ri}_0 +C\lnorm{\frac{\hd^2\ri}{\ri}}_0\norm{\hd^{r-2}\ri}_2\no\\
  \ls&C\norm{\ri}_r+C\norm{\ri}_4\norm{\ri}_{r-1}+C\norm{\ri}_3\norm{\ri}_r\no\\
  \ls &C\norm{\ri}_{\max(4,r)},
\end{align}
 by the higher order Hardy inequality.

\textbf{Step 3. Analysis of the integral $\I_3$}.
 Similar to those of $\I_1$, by \eqref{eu4.6}, \eqref{eu4.7}, \eqref{eq.Jd} and $\eta_t=v$, we have
\begin{align}
  \I_3=&-\int_\Omega\ri^2   \hd^r J^{-2}a^j_i\pd{\hd^r v^i}{j}dx+\int_{\Gamma_0}\ri^2   \hd^r J^{-2}a^3_i\hd^r v^i dx_1dx_2\no\\
  =&2\int_\Omega \ri^2   \hd^{r-1}( J^{-3}\hd J)a^j_i\pd{\hd^r v^i}{j}dx\no\\
  =&2\int_\Omega\hd^{r} J A^j_i\pd{\hd^r v^i}{j}\ri^2   J^{-2}dx\label{eq.3a.1}\\
  &+2\sum_{s=0}^{r-2} C_{r-1}^s \int_\Omega \hd^{r-1-s} J^{-3}\hd^{s+1} J a^j_i\pd{\hd^r v^i}{j}\ri^2  dx.\label{eq.3b.1}
\end{align}
Due to $v=\eta_t$,  we get
\begin{align}
  \eqref{eq.3a.1}=&2\int_\Omega \hd^r\pd{\eta^k}{l}A^l_k \pd{\hd^r v^i}{j}A^j_i\ri^2   J^{-1} dx\label{eq.3a0}\\
  &+2\sum_{s=0}^{r-2} C_{r-1}^s\int_\Omega\hd^{r-1-s} a_k^l\hd^{s+1} \pd{\eta^k}{l}\pd{\hd^r v^i}{j}A^j_i \ri^2   J^{-2}dx\label{eq.3a3}\\
  =& \frac{d}{dt}\int_\Omega \hd^r\pd{\eta^k}{l}A^l_k \pd{\hd^r \eta^i}{j}A^j_i \ri^2   J^{-1}dx+\eqref{eq.3a3}\no\\
  &+ \int_\Omega \hd^r\pd{\eta^k}{l} \pd{\hd^r \eta^i}{j}A^l_k A^j_i \dveta v \ri^2  J^{-1}dx\label{eq.3a1}\\
  &- \int_\Omega \hd^r\pd{\eta^k}{l} \pd{\hd^r \eta^i}{j}(A^l_k A^j_i)_t \ri^2   J^{-1}dx.\label{eq.3a2}
\end{align}
Noticing that $\hd^r\dv\eta(0)=0$, integrating over $[0,T]$ yields
\begin{align*}
  \int_0^T \eqref{eq.3a.1} dt'=&\int_\Omega \abs{\dveta\hd^r\eta(T)}^2 \ri^2   J^{-1}(T)dx+\int_0^T[\eqref{eq.3a1}+\eqref{eq.3a2}+\eqref{eq.3a3}]dt.
\end{align*}

By H\"older's inequality and Sobolev's embedding theorem, we obtain
\begin{align*}
  \labs{\int_0^T\eqref{eq.3a1}+\eqref{eq.3a2}dt} \ls & CT\sup_{[0,T]}\norm{\ri \hd^r\nb\eta}_0^2\norm{v}_3\norm{\eta}_4^6
  \ls CTP(\sup_{[0,T]}E_r(t)).
\end{align*}
By integration by parts,  we have
\begin{align}
\int_0^T\eqref{eq.3a3}=&2\sum_{s=0}^{r-2} C_{r-1}^s\intt\hd^{r-1-s} a_k^l\hd^{s+1} \pd{\eta^k}{l}\pd{\hd^r v^i}{j}\ri^2   J^{-2}A^j_i dxdt\no\\
  =&-2\sum_{s=0}^{r-2} C_{r-1}^s\intt\hd^{r-s} a_k^l\hd^{s+1} \pd{\eta^k}{l}\pd{\hd^{r-1} v^i}{j}\ri^2 J^{-2}A^j_idxdt\label{eq.3a3.1}\\
  &-2\sum_{s=0}^{r-2} C_{r-1}^s\intt\hd^{r-1-s} a_k^l\hd^{s+2} \pd{\eta^k}{l}\pd{\hd^{r-1} v^i}{j}\ri^2 J^{-2}A^j_idxdt \label{eq.3a3.2}\\
  &-2\sum_{s=0}^{r-2} C_{r-1}^s\intt\hd^{r-1-s} a_k^l\hd^{s+1} \pd{\eta^k}{l}\pd{\hd^{r-1} v^i}{j}\hd(\ri^2 J^{-2}A^j_i)dxdt.\label{eq.3a3.3}
\end{align}

We first consider \eqref{eq.3a3.1} and split it into four cases.

\emph{Case 1: $s=0$.} By an $L^2$-$L^\infty$-$L^2$ H\"older inequality, the Sobolev embedding theorem and the fundamental theorem of calculus for the norm $\norm{\ri \hd^{r-1} \nb v}_0$, we can easily get
\begin{align*}
  &\labs{2\intt\hd^{r} a_k^l\hd \pd{\eta^k}{l}\pd{\hd^{r-1} v^i}{j}\ri^2 J^{-2}A^j_idxdt}
  \ls M_0+CTP(\sup_{[0,T]}E_r(t)).
\end{align*}

\emph{Case 2: $s=1$.} Integration by parts yields
\begin{align}
  &-2C_{r-1}^1\intt\hd^{r-1}  a_k^l\hd^2 \pd{\eta^k}{l}\pd{\hd^{r-1} v^i}{j}\ri^2 J^{-2}A^j_idxdt\label{eq.3a3.4}\\
  =& 2C_{r-1}^1\intt\hd^r a_k^l\hd^2 \pd{\eta^k}{l}\pd{\hd^{r-2} v^i}{j}\ri^2 J^{-2}A^j_idxdt\label{eq.3a3.5}\\
  &+ 2C_{r-1}^1\intt\hd^{r-1} a_k^l\hd^3 \pd{\eta^k}{l}\pd{\hd^{r-2} v^i}{j}\ri^2 J^{-2}A^j_idxdt\label{eq.3a3.6}\\
  &+ 2C_{r-1}^1\intt\hd^{r-1} a_k^l\hd^2 \pd{\eta^k}{l}\pd{\hd^{r-2} v^i}{j}\hd(\ri^2 J^{-2}A^j_i)dxdt.\label{eq.3a3.7}
\end{align}
By an $L^2$-$L^6$-$L^3$ H\"older inequality, \eqref{eq.vL3est} and the Sobolev embedding theorem, we get
\begin{align*}
  \abs{\eqref{eq.3a3.5}}+\abs{\eqref{eq.3a3.7}}\ls M_0+CTP(\sup_{[0,T]}E_r(t)).
\end{align*}

By using \eqref{eq.ahd}, it holds
\begin{align}
  &\eqref{eq.3a3.6}
  = 2C_{r-1}^1\intt\hd^{r-2}\left[ \pd{\hd\eta^p}{q} J(A^l_kA^q_p-A^q_kA^l_p)\right] \pd{\hd^3\eta^k}{l}\pd{\hd^{r-2} v^i}{j}\ri^2 J^{-2}A^j_i dxdt\no\\
  =&2C_{r-1}^1\intt  \pd{\hd^{r-1}\eta^p}{\beta} (A^l_kA^\beta_p-A^\beta_kA^l_p) \pd{\hd^3\eta^k}{l}\pd{\hd^{r-2} v^i}{j}\ri^2 J^{-1}A^j_i dxdt\label{eq.3a3.8}\\
  &+2C_{r-1}^1\intt  \pd{\hd^{r-1}\eta^p}{3} (A^\beta_kA^3_p-A^3_kA^\beta_p) \pd{\hd^3\eta^k}{\beta}\pd{\hd^{r-2} v^i}{j}\ri^2 J^{-1}A^j_i dxdt\label{eq.3a3.9}\\
  &+2C_{r-1}^1\sum_{m=0}^{r-3}C_{r-2}^m\intt \pd{\hd^{m+1}\eta^p}{\beta} \hd^{r-2-m}\left[J(A^l_kA^\beta_p-A^\beta_kA^l_p)\right]\label{eq.3a3.10}\\
  &\qquad\qquad\qquad\qquad \cdot\pd{\hd^3\eta^k}{l}\pd{\hd^{r-2} v^i}{j}\ri^2 J^{-2}A^j_i dxdt\no\\
  &+2C_{r-1}^1\sum_{m=0}^{r-3}C_{r-2}^m\intt \pd{\hd^{m+1}\eta^p}{3} \hd^{r-2-m}\left[J(A^\beta_kA^3_p-A^3_kA^\beta_p)\right]\label{eq.3a3.11}\\
  &\qquad\qquad\qquad\qquad \cdot \pd{\hd^3\eta^k}{\beta}\pd{\hd^{r-2} v^i}{j}\ri^2 J^{-2}A^j_i dxdt.\no
\end{align}
By using the $L^6$-$L^2$-$L^3$ H\"older inequality for \eqref{eq.3a3.8} and $L^2$-$L^6$-$L^3$ H\"older inequality for \eqref{eq.3a3.9} on higher order terms, together with \eqref{eq.vL3est} and the Sobolev embedding theorem, we get
\begin{align*}
  \abs{\eqref{eq.3a3.8}}+\abs{\eqref{eq.3a3.9}}\ls M_0 +CTP(\sup_{[0,T]}E_r(t)).
\end{align*}
From the $L^\infty$-$L^6$-$L^2$-$L^3$ H\"older inequality for the case $m=0$ and $L^6$-$L^\infty$-$L^2$-$L^3$ H\"older inequality for the case $1\ls m\ls r-3$ on higher order terms, together with \eqref{eq.vL3est} and the Sobolev embedding theorem, we can get the same bounds for \eqref{eq.3a3.10} and \eqref{eq.3a3.11}. Then, we obtain
\begin{align*}
  \abs{\eqref{eq.3a3.4}}\ls M_0+CTP(\sup_{[0,T]}E_r(t)).
\end{align*}

\emph{Case 3: $s=2$.} By integration by parts, it yields
\begin{align}
  &-2C_{r-1}^2\intt\hd^{r-2} a_k^l\hd^3 \pd{\eta^k}{l}\pd{\hd^{r-1} v^i}{j}\ri^2 J^{-2}A^j_idxdt\label{eq.3a3.12}\\
  =&2C_{r-1}^2\intt\hd^{r-1} a_k^l\hd^3 \pd{\eta^k}{l}\pd{\hd^{r-2} v^i}{j}\ri^2 J^{-2}A^j_idxdt\label{eq.3a3.13}\\
  &+2C_{r-1}^2\intt\hd^{r-2} a_k^l\hd^4 \pd{\eta^k}{l}\pd{\hd^{r-2} v^i}{j}\ri^2 J^{-2}A^j_idxdt\label{eq.3a3.14}\\
  &+2C_{r-1}^2\intt\hd^{r-2} a_k^l\hd^3 \pd{\eta^k}{l}\pd{\hd^{r-2} v^i}{j}\hd(\ri^2 J^{-2}A^j_i)dxdt.\label{eq.3a3.15}
\end{align}
It is clear that $C_{r-1}^1\eqref{eq.3a3.13}=C_{r-1}^2\eqref{eq.3a3.6}$ while the latter has been just estimated. By an $L^6$-$L^2$-$L^3$ H\"older inequality for higher order terms, \eqref{eq.vL3est} and the Sobolev embedding theorem, we get
\begin{align*}
  \abs{\eqref{eq.3a3.14}}+\abs{\eqref{eq.3a3.15}}\ls M_0 +CTP(\sup_{[0,T]}E_r(t)).
\end{align*}
That is, \eqref{eq.3a3.12} has the same bound $M_0 +CTP(\sup_{[0,T]}E_r(t))$.

\emph{Case 4: $s=r-2$ with $r=5$.} It is easy to see that
\begin{align}
  \intt\hd^2 a_k^l\hd^4 \pd{\eta^k}{l}\pd{\hd^4 v^i}{j}\ri^2 J^{-2}A^j_idxdt
\end{align}
can be bounded by the desired bound in view of H\"older's inequality and the fundamental theorem of calculus.

Next, we consider \eqref{eq.3a3.2}. Since for the case $s=0$
\begin{align*}
  -2\intt\hd^{r-1} a_k^l\hd^{2} \pd{\eta^k}{l}\pd{\hd^{r-1} v^i}{j}\ri^2 J^{-2}A^j_idxdt=\frac{1}{r-1}\eqref{eq.3a3.4},
\end{align*}
and for the case $s=1$
\begin{align*}
  -2C_{r-1}^1\intt\hd^{r-2} a_k^l\hd^{3} \pd{\eta^k}{l}\pd{\hd^{r-1} v^i}{j}\ri^2 J^{-2}A^j_idxdt=\frac{C_{r-1}^1}{C_{r-1}^2}\eqref{eq.3a3.12},
\end{align*}
we have the desired bounds. For the case $s=2$, we have, by H\"older's inequality and the Sobolev embedding theorem and the fundamental theorem of calculus, that
\begin{align*}
  &\labs{-2C_{r-1}^2\intt\hd^{r-3}  a_k^l\hd^{4} \pd{\eta^k}{l}\pd{\hd^{r-1} v^i}{j}\ri^2 J^{-2}A^j_idxdt}\\
  \ls&CT\sup_{[0,T]}\norm{\eta}_rP(\norm{\eta}_3)\norm{\ri \hd^4\nb\eta}_0 \left(\lnorm{\ri \int_0^T\hd^{r-1}\D_t^2\nb \eta dt}_0+\norm{\ri \hd^{r-1}\nb v(0)}_0\right)\\
  \ls& M_0+CTP(\sup_{[0,T]}E_r(t)).
\end{align*}
For the case $s=r-2$ and $r=5$,
\begin{align*}
  &\labs{-2 C_{r-1}^{r-2}\intt\hd  a_k^l\hd^{r} \pd{\eta^k}{l}\pd{\hd^{r-1} v^i}{j}\ri^2 J^{-2}A^j_idxdt}\\
  \ls&CT\sup_{[0,T]}\norm{\eta}_4P(\norm{\eta}_3)\norm{\ri \hd^r\nb\eta}_0 \left(\lnorm{\ri \int_0^T\hd^{r-1}\D_t^2\nb \eta dt}_0+\norm{\ri \hd^{r-1}\nb v(0)}_0\right)\\
  \ls& M_0+CTP(\sup_{[0,T]}E_r(t)).
\end{align*}

For \eqref{eq.3a3.3},  by the H\"older inequality, the Sobolev embedding theorem and the fundamental theorem of calculus, we can easily obtain the desired bounds. Thus,
\begin{align*}
  \labs{\int_0^T\eqref{eq.3a3}}\ls M_0+CTP(\sup_{[0,T]}E_r(t)).
\end{align*}

Now, we turn to the estimates of $\int_0^T\eqref{eq.3b.1}dt$.

\emph{Case 1: $s=0$.} By \eqref{eq.Jd} and integration by parts, we see that
\begin{align}
  &2 \intt \hd^{r-1} J^{-3}\hd  J a^j_i\pd{\hd^r v^i}{j}\ri^2 dxdt\label{eq.3b.2}\\
  =&-6\intt \hd^{r-2}( J^{-3}\pd{\hd\eta^l}{k}A^k_l)  \pd{\hd\eta^p}{q}A^q_p  A^j_i\pd{\hd^r v^i}{j}\ri^2 J^2 dxdt\no\\
  =&-6\intt \pd{\hd^{r-1}\eta^l}{k} \pd{\hd\eta^p}{q}A^k_l A^q_p  A^j_i\pd{\hd^r v^i}{j}\ri^2 J^{-1} dxdt\no\\
  &-6\sum_{m=0}^{r-3} C_{r-2}^m\intt \pd{\hd^{m+1}\eta^l}{k}\hd^{r-2-m}( J^{-3}A^k_l)  \pd{\hd\eta^p}{q}A^q_p  A^j_i\pd{\hd^r v^i}{j}\ri^2 J^2 dxdt\no\\
  =&6\intt \pd{\hd^r\eta^l}{k} \pd{\hd\eta^p}{q}\pd{\hd^{r-1} v^i}{j}\ri^2 J^{-1}A^k_l A^q_p  A^j_i dxdt\label{eq.3b.3}\\
  &+6\intt \pd{\hd^{r-1}\eta^l}{k} \pd{\hd^2\eta^p}{q}\pd{\hd^{r-1} v^i}{j}\ri^2 J^{-1}A^k_l A^q_p  A^j_i dxdt\label{eq.3b.4}\\
  &+6\intt \pd{\hd^{r-1}\eta^l}{k} \pd{\hd\eta^p}{q}\pd{\hd^{r-1} v^i}{j}\hd\left(\ri^2 J^{-1}A^k_l A^q_p  A^j_i\right) dxdt\label{eq.3b.5}\\
  &+6\sum_{m=0}^1 C_2^m\intt \pd{\hd^{m+2}\eta^l}{k}\hd^{r-2-m}( J^{-3}A^k_l)  \pd{\hd\eta^p}{q}\pd{\hd^{r-1} v^i}{j}\ri^2 J^2 A^q_p  A^j_i dxdt\label{eq.3b.6}\\
  &+6\sum_{m=0}^{r-3} C_{r-2}^m\intt \pd{\hd^{m+1}\eta^l}{k}\hd^{r-1-m}( J^{-3}A^k_l)  \pd{\hd\eta^p}{q}\pd{\hd^{r-1} v^i}{j}\ri^2 J^2 A^q_p  A^j_i dxdt\label{eq.3b.7}\\
  &+6\sum_{m=0}^{r-3} C_{r-2}^m\intt \pd{\hd^{m+1}\eta^l}{k}\hd^{r-2-m}( J^{-3}A^k_l)  \pd{\hd^2\eta^p}{q}\pd{\hd^{r-1} v^i}{j}\ri^2 J^2 A^q_p  A^j_i dxdt\label{eq.3b.8}\\
  &+6\sum_{m=0}^{r-3} C_{r-2}^m\intt \pd{\hd^{m+1}\eta^l}{k}\hd^{r-2-m}( J^{-3}A^k_l)  \pd{\hd\eta^p}{q}\pd{\hd^{r-1} v^i}{j}\hd\left(\ri^2 J^2 A^q_p  A^j_i\right) dxdt.\label{eq.3b.9}
\end{align}

By the H\"older inequality, the Sobolev embedding theorem and the fundamental theorem of calculus, we can easily obtain the desired bound
$$M_0+\delta \sup_{[0,T]}E_r(t)+CTP(\sup_{[0,T]}E_r(t)),$$
for \eqref{eq.3b.3} and \eqref{eq.3b.5}-\eqref{eq.3b.9}.

For \eqref{eq.3b.4}, we use integration by parts with respect to time to get
\begin{align}
  \eqref{eq.3b.4}=&6\int_\Omega \pd{\hd^{r-1}\eta^l}{k} \pd{\hd^2\eta^p}{q}\pd{\hd^{r-1} \eta^i}{j}\ri^2 J^{-1}A^k_l A^q_p  A^j_i dx\Big|_{t=T}\label{eq.3b.10}\\
  &-6\intt \pd{\hd^{r-1}v^l}{k} \pd{\hd^2\eta^p}{q}\pd{\hd^{r-1} \eta^i}{j}\ri^2 J^{-1}A^k_l A^q_p  A^j_i dxdt\label{eq.3b.11}\\
  &-6\intt \pd{\hd^{r-1}\eta^l}{k} \pd{\hd^2v^p}{q}\pd{\hd^{r-1} \eta^i}{j}\ri^2 J^{-1}A^k_l A^q_p  A^j_i dxdt\label{eq.3b.12}\\
  &-6\intt \pd{\hd^{r-1}\eta^l}{k} \pd{\hd^2\eta^p}{q}\pd{\hd^{r-1} \eta^i}{j}\D_t\left(\ri^2 J^{-1}A^k_l A^q_p  A^j_i\right) dxdt.\label{eq.3b.13}
\end{align}
It is clear that $\eqref{eq.3b.11}=-\eqref{eq.3b.4}$. Thus,
\begin{align}
  \eqref{eq.3b.4}=\frac{1}{2}\left(\eqref{eq.3b.10}+\eqref{eq.3b.12}+\eqref{eq.3b.13}\right). \label{eq.3b.14}
\end{align}
Moreover, integration by parts yields
\begin{align}
  \eqref{eq.3b.12}=&6\intt \pd{\hd^r\eta^l}{k} \pd{\hd^2v^p}{q}\pd{\hd^{r-2} \eta^i}{j}\ri^2 J^{-1}A^k_l A^q_p  A^j_i dxdt\label{eq.3b.15}\\
  &+6\intt \pd{\hd^{r-1}\eta^l}{k} \pd{\hd^3v^p}{q}\pd{\hd^{r-2} \eta^i}{j}\ri^2 J^{-1}A^k_l A^q_p  A^j_i dxdt\label{eq.3b.16}\\
  &+6\intt \pd{\hd^{r-1}\eta^l}{k} \pd{\hd^2v^p}{q}\pd{\hd^{r-2} \eta^i}{j}\hd\left(\ri^2 J^{-1}A^k_l A^q_p  A^j_i\right) dxdt.\label{eq.3b.17}
\end{align}
Since $\eqref{eq.3b.16}=\eqref{eq.3b.4}$, we have from \eqref{eq.3b.14}
\begin{align}
  \eqref{eq.3b.4}=\eqref{eq.3b.10}+\eqref{eq.3b.13}+\eqref{eq.3b.15}+\eqref{eq.3b.17}. \label{eq.3b.18}
\end{align}

By integration by parts, we get for $t=T$
\begin{align*}
  \eqref{eq.3b.10}=&-12\int_\Omega \pd{\hd^r\eta^l}{k} \pd{\hd\eta^p}{q}\pd{\hd^{r-1} \eta^i}{j}\ri^2 J^{-1}A^k_l A^q_p  A^j_i dx\\
  &-6\int_\Omega \pd{\hd^{r-1}\eta^l}{k} \pd{\hd\eta^p}{q}\pd{\hd^{r-1} \eta^i}{j}\hd\left(\ri^2 J^{-1}A^k_l A^q_p  A^j_i\right) dx,
\end{align*}
which yields the desired bound for \eqref{eq.3b.10} by using the H\"older inequality, the Sobolev embedding theorem and the fundamental theorem of calculus.

Integration by parts implies
\begin{align}
  \eqref{eq.3b.13}=&6\intt \pd{\hd^{r-2}\eta^l}{k} \pd{\hd^3\eta^p}{q}\pd{\hd^{r-1} \eta^i}{j}\D_t\left(\ri^2 J^{-1}A^k_l A^q_p  A^j_i\right) dxdt\label{eq.3b.19}\\
  &  +6\intt \pd{\hd^{r-2}\eta^l}{k} \pd{\hd^2\eta^p}{q}\pd{\hd^r \eta^i}{j}\D_t\left(\ri^2 J^{-1}A^k_l A^q_p  A^j_i\right) dxdt\label{eq.3b.20}\\
  & + 6\intt \pd{\hd^{r-2}\eta^l}{k} \pd{\hd^2\eta^p}{q}\pd{\hd^{r-1} \eta^i}{j}\D_t\hd\left(\ri^2 J^{-1}A^k_l A^q_p  A^j_i\right) dxdt.\label{eq.3b.21}
\end{align}
Due to $\eqref{eq.3b.19}=-\eqref{eq.3b.13}$, it follows that
\begin{align*}
  \eqref{eq.3b.13}=\frac{1}{2}(\eqref{eq.3b.20}+\eqref{eq.3b.21}),
\end{align*}
which implies the desired bound for \eqref{eq.3b.13} by using the H\"older inequality and the Sobolev embedding theorem. It is clear that both \eqref{eq.3b.15} and \eqref{eq.3b.13} have the same bound in view of the H\"older inequality and the Sobolev embedding theorem. Thus, we have obtained
\begin{align*}
  \abs{\eqref{eq.3b.2}}\ls  M_0+\delta \sup_{[0,T]}E_r(t)+CTP(\sup_{[0,T]}E_r(t)).
\end{align*}

\emph{Case 2: $s=1$.} From \eqref{eq.Jd}, it follows that
\begin{align}
  &2C_{r-1}^1 \intt \hd^{r-2} J^{-3}\hd^2 J a^j_i\pd{\hd^r v^i}{j}\ri^2 dxdt\label{eq.3b.22}\\
  =&-6C_{r-1}^1  \intt \hd^{r-3}\left( J^{-3}A^k_l\pd{\hd\eta^l}{k}\right)\hd\left( JA^q_p\pd{\hd\eta^p}{q}\right) a^j_i\pd{\hd^r v^i}{j}\ri^2 dxdt\no\\
  =&-6C_{r-1}^1 \sum_{m=0}^{r-3}C_{r-3}^m \intt \hd^{r-3-m}\left( J^{-3}A^k_l\right)\pd{\hd^{m+1}\eta^l}{k}\hd\left( JA^q_p\right)\label{eq.3b.23}\\
  &\qquad\qquad\qquad\qquad\qquad\qquad\qquad\cdot\pd{\hd\eta^p}{q} \pd{\hd^r v^i}{j}\ri^2 JA^j_i dxdt\no\\
  &-6C_{r-1}^1  \sum_{m=0}^{r-3}C_{r-3}^m\intt \hd^{r-3-m}\left( J^{-3}A^k_l\right)\pd{\hd^{m+1}\eta^l}{k} \pd{\hd^2\eta^p}{q} \pd{\hd^r v^i}{j}\ri^2 J^2 A^j_i A^q_p dxdt.\label{eq.3b.24}
\end{align}
By using integration by parts with respect to $\hd$ for \eqref{eq.3b.23} and the cases $m=0, r-3$ in \eqref{eq.3b.24}, and integration by parts with respect to time for the case $m=1$ in \eqref{eq.3b.24}, together with the H\"older inequality, the Sobolev embedding theorem, \eqref{eq.vL3est} and the fundamental theorem of calculus, we obtain
\begin{align*}
  \abs{\eqref{eq.3b.22}}\ls  M_0+\delta \sup_{[0,T]}E_r(t)+CTP(\sup_{[0,T]}E_r(t)).
\end{align*}

\emph{Case 3: $s=2$.} By \eqref{eq.Jd} and integration by parts, it yields
\begin{align}
  &2C_{r-1}^2 \intt \hd^{r-3} J^{-3}\hd^3 J \pd{\hd^r v^i}{j}\ri^2 JA^j_i dxdt\label{eq.3b.27}\\
  =&-6C_{r-1}^2 \sum_{n=0}^{r-4}C_{r-4}^n\intt\pd{\hd^{n+1}\eta^l}{k}\hd^{r-4-n}J^{-3} \pd{\hd^3\eta^p}{q}\pd{\hd^r v^i}{j} A^q_p A^j_i A^k_l \ri^2 J^2dxdt\no\\
  &-6C_{r-1}^2\sum_{n=0}^{r-4}C_{r-4}^n \sum_{m=0}^1C_2^m\intt\pd{\hd^{n+1}\eta^l}{k}\hd^{r-4-n}J^{-3} \pd{\hd^{m+1}\eta^p}{q}\hd^{2-m}\left( JA^q_p\right)\no\\
   &\qquad\qquad\qquad\qquad\qquad\qquad\qquad\qquad\qquad\cdot\pd{\hd^r v^i}{j} A^j_i A^k_l\ri^2 J dxdt\no\\
  =&6C_{r-1}^2 \sum_{n=0}^{r-4}C_{r-4}^n\intt\pd{\hd^{n+2}\eta^l}{k}\hd^{r-4-n}J^{-3} \pd{\hd^3\eta^p}{q}\pd{\hd^{r-1} v^i}{j} A^q_p A^j_i A^k_l \ri^2 J^2dxdt\label{eq.3b.28}\\
  &+6C_{r-1}^2 \sum_{n=0}^{r-4}C_{r-4}^n\intt\pd{\hd^{n+1}\eta^l}{k}\hd^{r-4-n}J^{-3} \pd{\hd^4\eta^p}{q}\pd{\hd^{r-1} v^i}{j} A^q_p A^j_i A^k_l \ri^2 J^2dxdt\label{eq.3b.29}\\
  &+6C_{r-1}^2 \sum_{n=0}^{r-4}C_{r-4}^n\intt\pd{\hd^{n+1}\eta^l}{k} \pd{\hd^3\eta^p}{q}\pd{\hd^{r-1} v^i}{j} \hd(\hd^{r-4-n}J^{-3}A^q_p A^j_i A^k_l \ri^2 J^2)dxdt\label{eq.3b.30}\\
  &+6C_{r-1}^2\sum_{n=0}^{r-4}C_{r-4}^n \sum_{m=0}^1C_2^m\intt\pd{\hd^{n+2}\eta^l}{k}\hd^{r-4-n}J^{-3} \pd{\hd^{m+1}\eta^p}{q}\label{eq.3b.31}\\
  &\qquad\qquad\qquad\qquad\qquad\qquad\qquad\qquad\quad\cdot\hd^{2-m}\left( JA^q_p\right) \pd{\hd^{r-1} v^i}{j} A^j_i A^k_l\ri^2 J dxdt\no\\
  &+6C_{r-1}^2\sum_{n=0}^{r-4}C_{r-4}^n \sum_{m=0}^1C_2^m\intt\pd{\hd^{n+1}\eta^l}{k}\hd^{r-4-n}J^{-3} \pd{\hd^{m+2}\eta^p}{q}\label{eq.3b.32}\\
  &\qquad\qquad\qquad\qquad\qquad\qquad\qquad\quad\qquad\cdot\hd^{2-m}\left( JA^q_p\right) \pd{\hd^{r-1} v^i}{j} A^j_i A^k_l\ri^2 J dxdt\no\\
  &+6C_{r-1}^2\sum_{n=0}^{r-4}C_{r-4}^n \sum_{m=0}^1C_2^m\intt\pd{\hd^{n+1}\eta^l}{k}\hd^{r-4-n}J^{-3} \pd{\hd^{m+1}\eta^p}{q}\label{eq.3b.33}\\
  &\qquad\qquad\qquad\qquad\qquad\qquad\qquad\quad\qquad\cdot\hd^{3-m}\left( JA^q_p\right) \pd{\hd^{r-1} v^i}{j} A^j_i A^k_l\ri^2 J dxdt\no\\
  &+6C_{r-1}^2\sum_{n=0}^{r-4}C_{r-4}^n \sum_{m=0}^1C_2^m\intt\pd{\hd^{n+1}\eta^l}{k} \pd{\hd^{m+1}\eta^p}{q}\hd^{2-m}\left( JA^q_p\right)\pd{\hd^{r-1} v^i}{j}\label{eq.3b.34}\\
  &\qquad\qquad\qquad\qquad\qquad\qquad\qquad\quad\qquad\cdot\hd(\hd^{r-4-n}J^{-3} A^j_i A^k_l\ri^2 J) dxdt.\no
\end{align}
Due to the case $n=0$ in \eqref{eq.3b.28} is equal to $C_{r-1}^2\eqref{eq.3b.4}$, we have obtained the desired bound for this case of \eqref{eq.3b.28}. For \eqref{eq.3b.29}-\eqref{eq.3b.34} and the case $n=r-4$ with $r=5$ of \eqref{eq.3b.28}, we can easily use the H\"older inequality, the Sobolev embedding theorem and the fundamental theorem of calculus to get the desired bounds. Then, so does \eqref{eq.3b.27}.

\emph{Case 4: $s=r-2$ with $r=5$.} Integration by parts with respect to time yields
\begin{align*}
  &\intt \hd J^{-3}\hd^{4} J a^j_i\pd{\hd^5 v^i}{j}\ri^2  dxdt\\
  =&\int_\Omega \hd J^{-3}\hd^{4} J a^j_i\pd{\hd^5 \eta^i}{j}\ri^2  dx\Big|_0^T-\intt \D_t(\hd J^{-3}a^j_i)\hd^{4} J\pd{\hd^5 \eta^i}{j}\ri^2  dxdt\\
  &-\intt \hd J^{-3}\D_t\hd^{4} J a^j_i\pd{\hd^5 \eta^i}{j}\ri^2  dxdt-\intt \hd J^{-3}\hd^{4} J a^j_i\pd{\hd^5 \eta^i}{j}\ri^2  dxdt,
\end{align*}
which can be easily controlled by the desired bound in view of the H\"older inequality and the fundamental theorem of calculus.

Therefore, combining with four cases, we obtain
\begin{align}
 \labs{\int_0^T\eqref{eq.3b.1}dt} \ls  M_0+\delta \sup_{[0,T]}E_r(t)+CTP(\sup_{[0,T]}E_r(t)).
\end{align}

\textbf{Step 4. Analysis of the remainder $\I_4$.}

For $l=0$, we have from the integration by parts in time, the fundamental theorem of calculus, the H\"older inequality, the Sobolev embedding theorem and the Cauchy inequality that
\begin{align*}
  \labs{\intt  \hd^r\ri  v_t^i\hd^r v^idxdt}\ls&\labs{\intt  \hd^r\ri  v_{tt}^i\hd^r \eta^idxdt}+\labs{\int_\Omega \hd^r\ri  v_t^i\hd^r \eta^idx\Big|_0^T}\\
  \ls& CT \norm{\hd^r\ri}_0 \sup_{[0,T]}\norm{\hd^r \eta}_0\norm{v_{tt}}_2\\
  \ls&M_0+\delta\sup_{[0,T]}E_r(t)+CTP(\sup_{[0,T]}E_r(t)),
\end{align*}
where we need the condition $\ri \in H^{r}(\Omega)$.

For $l=1$, we have, at a similar way, that
\begin{align*}
  \labs{\intt  \hd^{r-1}\ri \hd  v_t^i \hd^r v^idxdt}
  \ls &CT\norm{\ri }_{L^\infty(\Omega)}^{1/2} \lnorm{\frac{\hd^{r-1}\ri }{\ri }}_0 \sup_{[0,T]}\norm{\ri^{1/2}\hd^r v}_0\norm{v_t}_3\\
  \ls &\norm{\ri }_2^2 \norm{\ri }_r^4+\delta \sup_{[0,T]}\norm{\ri^{1/2}\hd^r v}_0^2 +CT^4\norm{v_t}_3^4\\
  \ls&M_0+\delta\sup_{[0,T]}E_r(t)+CTP(\sup_{[0,T]}E_r(t)).
\end{align*}

For $l=2$, we get, by the H\"older inequality and the Sobolev embedding theorem,
\begin{align*}
  \labs{\intt  \hd^{r-2}\ri \hd^2  v_t^i \hd^r v^idxdt}
  \ls &CT\norm{\ri }_{L^\infty(\Omega)}^{1/2} \lnorm{\frac{\hd^{r-2}\ri }{\ri }}_1 \sup_{[0,T]}\norm{\ri^{1/2}\hd^r v}_0\norm{\hd^2 v_t}_{1}\\
  \ls &\norm{\ri }_2^2\norm{\ri }_{r}^4+\delta \sup_{[0,T]}\norm{\ri^{1/2}\hd^r v}_0^2 +CT^4\norm{v_t}_3^4\\
  \ls&M_0+\delta\sup_{[0,T]}E_r(t)+CTP(\sup_{[0,T]}E_r(t)).
\end{align*}

For $l=3$, we obtain, by the H\"older inequality and the Sobolev embedding theorem, that
\begin{align*}
  \labs{\intt  \hd^{r-3}\ri \hd^3  v_t^i \hd^r v^idxdt}
  \ls &CT\norm{\ri }_{L^\infty(\Omega)}^{1/2} \lnorm{\frac{\hd^{r-3}\ri }{\ri }}_2 \sup_{[0,T]}\norm{\ri^{1/2}\hd^r v}_0\norm{\hd^3 v_t}_{0}\\
  \ls &\norm{\ri }_2^2\norm{\ri }_r^4+\delta \sup_{[0,T]}\norm{\ri^{1/2}\hd^r v}_0^2 +CT^4\norm{v_t}_3^4\\
  \ls&M_0+\delta\sup_{[0,T]}E_r(t)+CTP(\sup_{[0,T]}E_r(t)).
\end{align*}

For $l=4$ with $r=5$, we have, by the H\"older inequality and the Sobolev embedding theorem, that
\begin{align*}
  \labs{\intt  \hd^{r-4}\ri \hd^4  v_t^i \hd^r v^idxdt}
  \ls &CT\norm{\ri }_{L^\infty(\Omega)}^{1/2} \lnorm{\frac{\hd^{r-4}\ri }{\ri }}_2 \sup_{[0,T]}\norm{\ri^{1/2}\hd^r v}_0\norm{\hd^4 v_t}_{0}\\
  \ls &\norm{\ri }_2^2\norm{\ri }_{r-1}^4+\delta \sup_{[0,T]}\norm{\ri^{1/2}\hd^r v}_0^2 +CT^4\norm{v_t}_4^4\\
  \ls&M_0+\delta\sup_{[0,T]}E_r(t)+CTP(\sup_{[0,T]}E_r(t)).
\end{align*}

\textbf{Step 5. Analysis of the remainders $\I_5$ and $\I_6$.}
By integration by parts, \eqref{eq.Piola}, \eqref{eu4.6} and \eqref{eu4.7} for $l=1,\cdots, r-1$, we have
\begin{align*}
 \int_0^T \I_5dt=\sum_{l=1}^{r-1} C_r^l\intt  \hd^{r-l} a_i^j\hd^l(\ri^2  J^{-2})\hd^r\pd{v^i}{j}dxdt.
\end{align*}
which can be written as, by integration by parts with respect to time,
\begin{align}
  &\sum_{l=1}^{r-1} C_r^l\int_\Omega \hd^{r-l} a_i^j\hd^l(\ri^2  J^{-2})\hd^r\pd{\eta^i}{j}dx\Big|_{t=T}\label{eq.I5.01}\\
  &-\sum_{l=1}^{r-1} C_r^l\intt \hd^{r-l} a_{ti}^j\hd^l(\ri^2  J^{-2})\hd^r\pd{\eta^i}{j}dxdt\label{eq.I5.02}\\
  &-\sum_{l=1}^{r-1} C_r^l\intt  \hd^{r-l} a_i^j\hd^l(\ri^2  \D_t J^{-2})\hd^r\pd{\eta^i}{j}dxdt.\label{eq.I5.03}
\end{align}

\emph{Case 1: $l=1$.} By using the fundamental theorem of calculus twice and \eqref{eq.Ad}, we get for \eqref{eq.I5.01}
\begin{align*}
  &\int_\Omega \hd^{r-1} a^j_i(T)\hd(\ri^2  J^{-2}(T))\hd^r\pd{\eta^i}{j}(T)dx\\
  =&\int_\Omega \hd^{r-1} a^j_i(T) \hd\ri^2  J^{-2}(T)\hd^r\pd{\eta^i}{j}(T)dx\\
  &-2\int_\Omega \hd^{r-1} a^j_i(T)\ri^2  J^{-2}(T)A^p_q(T)\hd\pd{\eta^q}{p}(T)\hd^r\pd{\eta^i}{j}(T)dx\\
  =&\int_\Omega \int_0^T\hd^{r-1} a^j_{ti}dt J^{-2}(T)\hd\ri^2  \hd^r\pd{\eta^i}{j}(T)dx-2\int_\Omega\int_0^T\hd^{r-1} a^j_{ti}dt J^{-2} A^p_{q}\hd\pd{\eta^q}{p}\ri^2  \hd^r\pd{\eta^i}{j}(T)dx\\
  \ls&CT\norm{\ri }_3 P(\sup_{[0,T]}E_r(t))\left(1+\norm{\ri \hd^{r-1}\nb v_t}_0\right)\norm{\ri \hd^r\nb\eta}_0\\
  \ls &M_0+\delta\sup_{[0,T]}E_r(t)+CTP(\sup_{[0,T]}E_r(t)).
\end{align*}
Similarly, we have for \eqref{eq.I5.02}
\begin{align*}
  \labs{\intt \hd^{r-1} a_{ti}^j\hd (\ri^2  J^{-2})\hd^r\pd{\eta^i}{j}dxdt}\ls M_0+\delta\sup_{[0,T]}E_r(t)+CTP(\sup_{[0,T]}E_r(t)).
\end{align*}
For \eqref{eq.I5.03}, it is harder to be controlled than \eqref{eq.I5.01} and \eqref{eq.I5.02}. We write it as
\begin{align}
&-\intt  \hd^{r-1} a_i^j\hd (\ri^2  \D_t J^{-2})\hd^r\pd{\eta^i}{j}dxdt\no\\
=&2\intt  \hd^{r-1} a_i^j \hd\ri^2   J^{-2}A^p_q\pd{v^q}{p}\hd^r\pd{\eta^i}{j}dxdt\label{eq.I5.1}\\
  &+2\intt  \hd^{r-1} a_i^j \ri^2  \hd( J^{-2}A^p_q)\pd{v^q}{p}\hd^r\pd{\eta^i}{j}dxdt\label{eq.I5.2}\\
  &+2\intt  \hd^{r-1} a_i^\beta \ri^2  J^{-2}A^p_q\pd{\hd v^q}{p}\hd^r\pd{\eta^i}{\beta}dxdt\label{eq.I5.3}\\
  &+2\intt  \hd^{r-1} a_i^3 \ri^2  J^{-2}A^p_q\pd{\hd v^q}{p}\hd^r\pd{\eta^i}{3}dxdt.\label{eq.I5.4}
\end{align}
It is easy to see that \eqref{eq.I5.1} and \eqref{eq.I5.2} are bounded by $$M_0+\delta\sup_{[0,T]}E_r(t)+CTP(\sup_{[0,T]}E_r(t)).$$
By integration by parts, H\"older's inequality, \eqref{eq.vL3est} and Sobolev's embedding theorem, it holds
\begin{align*}
  \eqref{eq.I5.3}=&-2\intt \pd{\hd^{r-1} a^\beta_i}{\beta} A^p_{q}\hd\pd{v^q}{p}\hd^r\eta^i\ri^2  J^{-2} dx dt\\
  &-2\int_0^T \int_\Omega\hd^{r-1} a^\beta_i A^p_{q}\hd\pd{v^q}{p} \hd^r\eta^i\ri^2 \pd{J^{-2}}{\beta}dxdt\\
  &-2\intt \hd^{r-1} a^\beta_i \pd{A^p_{q}}{\beta}\hd\pd{v^q}{p} \hd^r\eta^i\ri^2  J^{-2}dx dt\\
  &-2\int_0^T \int_\Omega\hd^{r-1} a^\beta_i A^p_{q}\hd\pd{v^q}{p\beta} \hd^r\eta^i\ri^2  J^{-2}dx dt\\
  &-2\intt \hd^{r-1} a^\beta_i A^p_{q}\hd\pd{v^q}{p} \hd^r\eta^i\pd{(\ri^2)}{\beta}J^{-2}dx dt\\
  \lesssim& \int_0^T\norm{\ri \hd^r a}_0\norm{\eta}_2^2\norm{\hd\nb v}_1\norm{\ri \hd^r\eta}_1dt\\
  &+\int_0^T\norm{\ri }_2\norm{\hd^{r-1} a}_0\norm{\eta}_3^4\norm{\eta}_r\norm{\hd\nb v}_1\norm{\ri \hd^r\eta}_1dt\\
  &+\int_0^T\norm{\ri }_2\norm{\hd^{r-1} a}_0\norm{\eta}_3^2\norm{\hd\nb v}_1\norm{\ri \hd^r\eta}_1dt\\
  &+\int_0^T\norm{\ri }_2\norm{\hd^{r-1} a}_0\norm{\eta}_3^2\norm{\hd^2\nb v}_{L^3(\Omega)}\norm{\ri \hd^r\eta}_1dt\\
  &+ \int_0^T\norm{\ri }_3\norm{\hd^{r-1} a}_0\norm{A}_2\norm{\hd\nb v}_1\norm{\ri \hd^r\eta}_1dt\\
  \ls &M_0+\delta\sup_{[0,T]}E_r(t)+CTP(\sup_{[0,T]}E_r(t)).
\end{align*}

For \eqref{eq.I5.4}, we can easily get the desired bound by an $L^6$-$L^3$-$L^2$ H\"older's inequality and the Sobolev embedding theorem because each component of $a_i^3$ is quadratic in $\hd \eta$ due to \eqref{eq.a}.

\emph{Case 2: $l=2$.} By the fundamental theorem of calculus, H\"older's inequality and the Sobolev embedding theorem, we can see that
\begin{align*}
  \labs{\int_\Omega \hd^{r-2} a^j_i \hd^2(\ri^2 J^{-2})\hd^r\pd{\eta^i}{j}dx\Big|_{t=T}}\ls M_0+\delta\sup_{[0,T]}E_r(t)+CTP(\sup_{[0,T]}E_r(t)).
\end{align*}
From \eqref{eq.adt}, the H\"older inequality, the Sobolev embedding theorem and \eqref{eq.vL3est}, it yields
\begin{align*}
  &-\intt  \hd^{r-2} a^j_{ti} \hd^2(\ri^2 J^{-2})\hd^r\pd{\eta^i}{j}dxdt\\
  =&\intt \hd^{r-2}\left(J\pd{v^q}{p}(A^j_qA^p_i- A^p_qA^j_i)\right)\hd^2(\ri^2  J^{-2})\hd^4\pd{\eta^i}{j}dxdt\\
  =&-2\intt  J\hd^{r-2}\pd{v^q}{p}(A^j_qA^p_i -A^p_qA^j_i)\ri^2  J^{-2}A^k_l\pd{\hd^2\eta^l}{k}\hd^4\pd{\eta^i}{j}dxdt+\text{remainders}\\
  \ls &C\norm{\ri }_2\int_0^T\norm{\hd^{r-2}\nb v}_{L^3(\Omega)}\norm{\eta}_r^7 \norm{\ri \hd^r\nb\eta}_0dt+\text{remainders}\\
  \ls &M_0+\delta\sup_{[0,T]}E_r(t)+CTP(\sup_{[0,T]}E_r(t)),
\end{align*}
since the remainders can be easily controlled by the desired bound.

Similarly, we can get the bound for the last integral
\begin{align*}
  \labs{\intt  \hd^{r-2} a^j_{i} \hd^2(\ri^2 \D_tJ^{-2})\hd^r\pd{\eta^i}{j}dxdt}\ls M_0+\delta\sup_{[0,T]}E_r(t)+CTP(\sup_{[0,T]}E_r(t)).
\end{align*}

\emph{Case 3: $l=3$.} By using the H\"older inequality, the Sobolev embedding theorem and the fundamental theorem of calculus, the spatial integral \eqref{eq.I5.01} can be bounded by $M_0+\delta\sup_{[0,T]}E_r(t)+CTP(\sup_{[0,T]}E_r(t))$. Similarly, we can get the same bound for the first space-time integral \eqref{eq.I5.02}. Since the norm $\|\ri  \D_t^2 J^{-2}\|_3$ is contained in the energy function $E_r(t)$, the last space-time integral \eqref{eq.I5.03} have the same bound by the H\"older inequality, the Sobolev embedding theorem and the fundamental theorem of calculus.

\emph{Case 4: $l=r-1$ and $r=5$.} They can be easily controlled by the desired bound, especially with the help of the fundamental theorem of calculus for \eqref{eq.I5.03}.

We can deal with the integrals in $\I_6$ by using a similar argument and we omit the details for simplicity.

\textbf{Step 6. Summing identities.}
Integrating \eqref{eq.esthd4} over $[0,T]$, by H\"older's inequality and Cauchy's inequality,  we have, for sufficiently small $T$ such that
\begin{align*}
  &\frac{1}{2}\int_\Omega \ri \abs{\hd^r v}^2dx+\frac{1}{2}\int_\Omega \left[\abs{\nb_\eta \hd^r\eta}^2+\abs{\dveta \hd^r \eta}^2-\abs{\curleta \hd^r \eta}^2\right]\ri^2   J^{-1}dx\\
  =&\frac{1}{2}\int_\Omega \ri \abs{\hd^r u_0}^2dx-\frac1{2}\intt  \left[\abs{\nb_\eta \hd^r\eta}^2+\abs{\dveta \hd^r \eta}^2-\abs{\curleta \hd^r \eta}^2\right]\ri^2   J^{-1}\dveta vdxdt\\
  &+\frac{1}{2}\intt  \hd^r\pd{\eta^l}{k} \hd^r \pd{\eta^i}{j}\D_t(A_i^k A_l^j-A_i^jA_l^k)\ri^2   J^{-1}dxdt-\intt  \hd^r\ri^2   (A^j_i J^{-1})_t\pd{\hd^r \eta^i}{j}dx dt\\
  &+\int_\Omega \hd^r\ri^2   \dveta\hd^r \eta(T)J^{-1}dx-\int_0^T[\eqref{eq.1c1}+\eqref{eq.3a2}+\eqref{eq.3a3} +\eqref{eq.3b.1}]dt\\
  &+\int_0^T[\I_{4}+\I_{5}+\I_{6}]dt\\
  \ls& M_0+\delta\sup_{[0,T]}E_r(t)+CTP(\sup_{[0,T]}E_r(t)).
\end{align*}

By the fundamental theorem of calculus, we have
\begin{align*}
  &\int_\Omega \abs{\nb_\eta \hd^r\eta}^2\ri^2  J^{-1}dx
  =\int_\Omega (\nb_\eta \hd^r\eta)_i^j(\nb_\eta \hd^r\eta)_i^j\ri^2  J^{-1}dx
  =\int_\Omega \hd^r\pd{\eta^j}{s}A^s_i\hd^r\pd{\eta^j}{p}A^p_i\ri^2  J^{-1}dx\no\\
  =&\int_\Omega \ri^2 \left(\int_0^t (A^s_i J^{-1})_t(t')dt'+A^s_i(0)\right)\hd^r\pd{\eta^j}{s} \left(\int_0^t A^p_{ti}(t')dt'+A^p_i(0)\right)\hd^r\pd{\eta^j}{p}dx\no\\
  =&\int_\Omega \ri^2 \abs{\hd^r\nb\eta}^2dx+\int_\Omega \ri^2 \left(\int_0^t (A^s_i J^{-1})_t(t')dt'\right)\hd^r\pd{\eta^j}{s}\left(\int_0^t A^p_{ti}(t')dt'\right)\hd^r\pd{\eta^j}{p}dx\no\\
  &+\int_\Omega \ri^2 (\nb\hd^r\eta)_i^j\left(\int_0^t A^p_{ti}(t')dt'\right)\hd^r\pd{\eta^j}{p}dx\\
  &+\int_\Omega \ri^2 (\nb\hd^r\eta)_i^j\left(\int_0^t (A^s_i J^{-1})_t(t')dt'\right)\hd^r\pd{\eta^j}{s}dx,
\end{align*}
which yields
\begin{align*}
  &\sup_{[0,T]}\labs{\norm{\ri \nb_\eta\hd^r \eta}_0^2-\norm{\ri \hd^r\nb \eta}_0^2}\\
  \ls &CT^2\sup_{[0,T]}\norm{\ri \hd^r\nb \eta}_0^2\norm{v}_{r-1}^2\norm{\eta}_r^8+CT\sup_{[0,T]}\norm{\ri \hd^r\nb \eta}_0^2\norm{v}_{r-1}\norm{\eta}_r^4.
\end{align*}
Thus, taking $T$ so small that $CT\norm{v}_{r-1}\norm{\eta}_r^4\ls 1/6$, we obtain
\begin{align}\label{eq.nbeta4}
  \frac{1}{2}\sup_{[0,T]}\norm{\ri \hd^r\nb \eta}_0\ls \sup_{[0,T]}\norm{\ri \nb_\eta\hd^r \eta}_0\ls \frac{3}{2}\sup_{[0,T]}\norm{\ri \hd^r\nb \eta}_0.
\end{align}

Similarly, we have
\begin{align}\label{eq.dveta4}
  &\int_\Omega \abs{\dveta\hd^r\eta}^2 \ri^2  J^{-1}dx\\
  =&\int_\Omega \ri^2  \left(\hd^r\dv\eta+\hd^r\pd{\eta^k}{l}\int_0^t (J^{-1} A^l_{k})_t dt'\right)\left(\hd^r\dv\eta+\pd{\hd^r \eta^i}{j}\int_0^t A^j_{ti}dt'\right) dx\no\\
  =&\int_\Omega \ri^2  \abs{\hd^r\dv\eta}^2 dx+\int_\Omega \ri^2  \hd^r\dv\eta\hd^r\pd{\eta^i}{j}\left(\int_0^t A^j_{ti}dt'\right)dx\no\\
  &+\int_\Omega \ri^2  \hd^r\dv\eta\hd^r\pd{\eta^k}{l}\left(\int_0^t (J^{-1}A^l_{k})_tdt'\right)dx\no\\
  &+\int_\Omega\ri^2  J^{-1}\hd^r\pd{\eta^k}{l}\left(\int_0^t A^l_{tk}dt'\right)\pd{\hd^r \eta^i}{j}\left(\int_0^t A^j_{ti}dt'\right)dx,\no
\end{align}
and
\begin{align}\label{eq.curleta4}
  &\int_\Omega \abs{\curleta\hd^r\eta}^2\ri^2  J^{-1} dx\\
  =&\int_\Omega \ri^2 \left(\hd^r\curl \eta+\eps_{\cdot ij}\pd{\hd^r\eta^j}{s}\int_0^t(A_i^s J^{-1})_tdt'\right)\left(\hd^r\curl \eta+\eps_{\cdot kl}\pd{\hd^r\eta^l}{p}\int_0^t A_{tk}^pdt'\right)dx\no\\
  =&\int_\Omega \ri^2  \abs{\hd^r\curl\eta}^2  dx+\eps_{\cdot ij}\int_\Omega \ri^2 \hd^r\curl \eta \pd{\hd^r\eta^j}{s}\left(\int_0^t(A_i^s J^{-1})_tdt'\right)dx\no\\
  &+\eps_{\cdot kl}\int_\Omega \ri^2 \hd^r\curl \eta \pd{\hd^r\eta^l}{p}\left(\int_0^t A_{tk}^pdt'\right)dx\no\\
  &+\eps_{\cdot ij}\eps_{\cdot kl}\int_\Omega \ri^2 \pd{\hd^r\eta^j}{s}\left(\int_0^t(A_i^s J^{-1})_tdt'\right)\pd{\hd^r\eta^l}{p}\left(\int_0^t A_{tk}^pdt'\right)dx,\no
\end{align}
which yields for a sufficiently small $T$ that
\begin{align}\label{eq.dvequi}
  \frac{1}{2}\sup_{[0,T]}\norm{\ri \hd^r\dv \eta}_0\ls &\sup_{[0,T]}\norm{\ri \dveta\hd^r \eta}_0\ls \frac{3}{2}\sup_{[0,T]}\norm{\ri \hd^r\dv \eta}_0,\\
  \frac{1}{2}\sup_{[0,T]}\norm{\ri \hd^r\curl \eta}_0\ls &\sup_{[0,T]}\norm{\ri \curleta\hd^r \eta}_0\ls \frac{3}{2}\sup_{[0,T]}\norm{\ri \hd^r\curl \eta}_0.\label{eq.curlequi}
\end{align}
Thus, we obtain the desired inner estimates with the help of curl estimates.

\textbf{Step 7: Boundary estimates.}
By Lemma \ref{lem.tangential.trace}, \eqref{w-em}, the fundamental theorem of calculus and H\"older's inequality,  we obtain
\begin{align*}
\abs{\hd^r \eta^\alpha}_{-1/2}^2\lesssim& \norm{\hd^r \eta}_0^2+\norm{\curl\hd^{r-1}\eta}_0^2\\
  \lesssim &\norm{\ri \hd^r \eta}_0^2+\norm{\ri \nb\hd^r \eta}_0^2+\norm{\curl\hd^{r-1}\eta}_0^2\\
  \lesssim & \lnorm{\ri \int_0^t\hd^r v dt}_0^2+\norm{\ri \nb \hd^r \eta}_0^2+\norm{\curl\hd^{r-1}\eta}_0^2\\
  \lesssim &\norm{\ri}_2 T^2\sup_{[0,T]} \norm{\ri^{1/2}\hd^r v}_0^2+\norm{\ri \nb\hd^r \eta}_0^2+\norm{\curl \hd^{r-1}\eta}_0^2,
\end{align*}
which implies the desired estimates from curl estimates.
\end{proof}

\section{The estimates for the time derivatives}\label{sec.2r}

We have the following estimates.

\begin{proposition}
  Let $r\in\{4,5\}$. Then, for a small $\delta>0$ and the constant $M_0$ depending on $1/\delta$, we have
  \begin{align*}
    &\sup_{[0,T]}\Big[\norm{\ri^{1/2}\D_t^{2r} v}_0^2+\norm{\ri  \D_t^{2r}\nb \eta}_0^2 +\norm{\ri  \D_t^{2r}\dv \eta}_0^2\Big]
    \ls  M_0+\delta\sup_{[0,T]}E_5(t)+CTP(\sup_{[0,T]}E_5(t)).
  \end{align*}
\end{proposition}

\begin{proof}
Acting $\D_t^{2r}$ on \eqref{eu4.2}, and taking the $L^2(\Omega)$-inner product with $\D_t^{2r}v^i$, we obtain
\begin{align*}
&\frac{1}{2}\frac{d}{dt}\int_\Omega \ri |\D_t^{2r} v|^2dx+\I_1+\I_2 =\I_3,
\end{align*}
where
\begin{align*}
  \I_1:=&\int_\Omega \D_t^{2r}a^j_i\pd{\left(\ri^2  J^{-2}\right)}{j}\D_t^{2r}v^idx,\qquad
  \I_2:=\int_\Omega a^j_i\pd{\left(\ri^2  \D_t^{2r} J^{-2}\right)}{j}\D_t^{2r}v^idx,\\
  \I_3:=&-\sum_{l=1}^{2r-1}C_{2r}^l\int_\Omega \D_t^{2r-l}a_i^j \pd{\left(\ri^2  \D_t^l J^{-2}\right)}{j} \D_t^{2r}v^idx.
\end{align*}

\textbf{Step 1. Analysis of the integral $\I_1$.} Noticing that $\D_t^{2r}a^3_i\D_t^{2r}v^i=0$ on $\Gamma_0$, integration by parts gives that
\begin{align}
  \I_1=&-\int_\Omega \D_t^{2r}a^j_i \pd{\D_t^{2r}v^i}{j}\ri^2  J^{-2}dx+\int_{\Gamma_0}\D_t^{2r}a^3_i\D_t^{2r}v^i \ri^2  J^{-2} dx_1dx_2\no\\
  =&-\int_\Omega \D_t^{2r-1}(\pd{v^l}{k} J^{-1}(a_i^ja_l^k-a_i^ka_l^j))\pd{\D_t^{2r}v^i}{j}\ri^2  J^{-2}dx\no\\
  =&-\int_\Omega \D_t^{2r-1}\pd{v^l}{k}\pd{\D_t^{2r}v^i}{j}A_l^k A_i^j\ri^2  J^{-1}dx\label{eq.2r.1}\\
  &+\int_\Omega \D_t^{2r-1}\pd{v^l}{k}\pd{\D_t^{2r}v^i}{j}A_i^kA_l^j\ri^2  J^{-1} dx\label{eq.2r.2}\\
  &-\sum_{s=1}^{2r-1}C_{2r-1}^s\int_\Omega \D_t^{2r-1-s}\pd{v^l}{k}\D_t^s(J^{-1}(a_i^ja_l^k-a_i^ka_l^j))\pd{\D_t^{2r}v^i}{j}\ri^2  J^{-2}dx.\label{eq.2r.3}
\end{align}
Then, we have
\begin{align*}
  \eqref{eq.2r.1}=&-\frac{1}{2}\frac{d}{dt}\int_\Omega \abs{\dveta \D_t^{2r-1} v}^2\ri^2  J^{-1}dx+\frac{1}{2}\int_\Omega \D_t^{2r-1}\pd{v^l}{k} \pd{\D_t^{2r-1}v^i}{j}\D_t(A_l^k A_{i}^j)\ri^2  J^{-1}dx\\
  &+\frac{1}{2}\int_\Omega \abs{\dveta \D_t^{2r-1} v}^2\ri^2 \D_t J^{-1}dx.
\end{align*}
It follows from \eqref{eq.curletaF2} that
\begin{align*}
  \D_t^{2r-1}\pd{v^l}{k} A_i^kA_l^j\pd{\D_t^{2r}v^i}{j}=&\frac{1}{2}\D_t\left[\abs{\nb_\eta \D_t^{2r-1}v}^2-\abs{\curleta \D_t^{2r-1}v}^2\right]-\frac{1}{2}\D_t^{2r-1}\pd{v^l}{k} \pd{\D_t^{2r-1}v^i}{j}\D_t(A_i^kA_l^j).
\end{align*}
Thus, we have
\begin{align*}
  \eqref{eq.2r.2}=&\frac{1}{2}\frac{d}{dt}\int_\Omega \left[\abs{\nb_\eta \D_t^{2r-1}v}^2-\abs{\curleta \D_t^{2r-1}v}^2\right]\ri^2  J^{-1}dx\\
  &-\frac{1}{2}\int_\Omega \left[\abs{\nb_\eta \D_t^{2r-1}v}^2-\abs{\curleta \D_t^{2r-1}v}^2\right]\ri^2  \D_t J^{-1}dx\\
  &-\frac{1}{2}\int_\Omega \D_t^{2r-1}\pd{v^l}{k} \pd{\D_t^{2r-1}v^i}{j}\D_t(A_i^kA_l^j)\ri^2  J^{-1}dx.
\end{align*}

\textbf{Step 2. Analysis of the integral $\I_2$}.
Similar to those of $\I_1$, by noticing that $a^3_i\D_t^{2r}v^i=0$ on $\Gamma_0$, we get
\begin{align}
  \I_2=&-\int_\Omega a^j_i\D_t^{2r}\pd{v^i}{j}\ri^2  \D_t^{2r} J^{-2}dx
  +\int_{\Gamma_0} a^3_i\D_t^{2r}v^i\ri^2  \D_t^{2r} J^{-2}dx_1dx_2\no\\
  =&2\int_\Omega A^j_i\D_t^{2r}\pd{v^i}{j}A_l^k \D_t^{2r-1}\pd{v^l}{k}\ri^2   J^{-1}dx\no\\
  &+2\sum_{s=0}^{2r-2}C_{2r-1}^s\int_\Omega A^j_i\D_t^{2r}\pd{v^i}{j}\ri^2  J\D_t^{2r-1-s} (J^{-2}A_l^k)\D_t^s \pd{v^l}{k}dx\label{eq.2r.4}\\
  =&-2 \cdot\eqref{eq.2r.1}+\eqref{eq.2r.4}.\no
\end{align}

Thus, we obtain by integrating over $[0,T]$
\begin{align*}
&\frac{1}{2}\int_\Omega \ri |\D_t^{2r} v|^2dx+\frac{1}{2}\int_\Omega \left[\abs{\nb_\eta \D_t^{2r-1}v}^2+\abs{\dveta \D_t^{2r-1} v}^2-\abs{\curleta \D_t^{2r-1}v}^2\right]\ri^2  J^{-1}dx\\
  =&\frac{1}{2}\int_\Omega \ri |\D_t^{2r} v(0)|^2dx+\frac{1}{2}\int_\Omega \left[\abs{\nb \D_t^{2r-1}v(0)}^2+\abs{\dv \D_t^{2r-1} v(0)}^2-\abs{\curl  \D_t^{2r-1}v(0)}^2\right]\ri^2  dx\\
 &+\frac{1}{2}\intt  \left[\abs{\nb_\eta \D_t^{2r-1}v}^2+\abs{\dveta \D_t^{2r-1} v}^2-\abs{\curleta \D_t^{2r-1}v}^2\right]\ri^2  \D_t J^{-1}dxdt\\
  &+\frac{1}{2}\intt  \D_t^{2r-1}\pd{v^l}{k} \pd{\D_t^{2r-1}v^i}{j}\D_t(A_i^kA_l^j)\ri^2  J^{-1}dxdt\\
  &+\frac1{2}\intt  \D_t^{2r-1}\pd{v^l}{k} \pd{\D_t^{2r-1}v^i}{j}\D_t(A_l^k A_{i}^j)\ri^2  J^{-1}dxdt \\
  &-\int_0^T\eqref{eq.2r.3}dt-\int_0^T\eqref{eq.2r.4}dt+\int_0^T\I_3dt.
\end{align*}
The  first three space-time double integrals can be absorbed by the left hand side as long as $T$ is sufficiently small.

\textbf{Step 3. Analysis of the remainder $\int_0^T \eqref{eq.2r.3}dt$.} Integration by parts with respect to time gives
\begin{align}
-\int_0^T \eqref{eq.2r.3}dt
  =&\sum_{s=1}^{2r-1}C_{2r-1}^s\intt  \D_t^{2r-1-s}\pd{v^l}{k}\D_t^s(J(A_i^jA_l^k-A_i^kA_l^j)) \pd{\D_t^{2r}v^i}{j}\ri^2  J^{-2}dxdt\no\\
  =&\sum_{s=1}^{2r-1}C_{2r-1}^s\int_\Omega \D_t^{2r-s}\pd{\eta^l}{k}\D_t^s(J(A_i^jA_l^k-A_i^kA_l^j)) \pd{\D_t^{2r}\eta^i}{j}\ri^2  J^{-2}dx\Big|_0^T \label{eq.2r.5}\\
  &-\sum_{s=1}^{2r-1}C_{2r-1}^s\intt  \D_t^{2r+1-s}\pd{\eta^l}{k}\D_t^s(J(A_i^jA_l^k-A_i^kA_l^j)) \pd{\D_t^{2r}\eta^i}{j}\ri^2  J^{-2}dxdt\label{eq.2r.6}\\
  &-\sum_{s=1}^{2r-1}C_{2r-1}^s\intt  \D_t^{2r-s}\pd{\eta^l}{k}\D_t^{s+1}(J(A_i^jA_l^k-A_i^kA_l^j)) \pd{\D_t^{2r}\eta^i}{j}\ri^2  J^{-2}dxdt\label{eq.2r.7}\\
  &-\sum_{s=1}^{2r-1}C_{2r-1}^s\intt  \D_t^{2r-s}\pd{\eta^l}{k}\D_t^{s}(J(A_i^jA_l^k-A_i^kA_l^j)) \pd{\D_t^{2r}\eta^i}{j}\ri^2  \D_tJ^{-2}dxdt.\label{eq.2r.8}
\end{align}

We first consider \eqref{eq.2r.6}. For the cases $s=1,2$, it is easy to get the desired bounds by the H\"older inequality and the Sobolev embedding theorem. For the case $s=3$,
\begin{align}
  &\intt \D_t^{2r-2}\pd{\eta^l}{k}\D_t^3(J(A_i^jA_l^k-A_i^kA_l^j)) \pd{\D_t^{2r}\eta^i}{j}\ri^2  J^{-2}dxdt\label{eq.2r.9}\\
  =&\intt \D_t^{2r-2}\pd{\eta^l}{k}\D_t^3(J(A_i^\beta A_l^k-A_i^kA_l^\beta)) \pd{\D_t^{2r}\eta^i}{\beta}\ri^2  J^{-2}dxdt\label{eq.2r.10}\\
  &+\intt \D_t^{2r-2}\pd{\eta^l}{\beta}\D_t^3(J(A_i^3A_l^\beta-A_i^\beta A_l^3)) \pd{\D_t^{2r}\eta^i}{3}\ri^2  J^{-2}dxdt.\label{eq.2r.11}
\end{align}
It is clear that \eqref{eq.2r.11} is easy to deal with by an $L^6$-$L^3$-$L^2$ H\"older inequality. For \eqref{eq.2r.10}, integration by parts yields
\begin{align}
  \eqref{eq.2r.10}=&-\intt \D_t^{2r-2}\pd{\eta^l}{k\beta}\D_t^3(J(A_i^\beta A_l^k-A_i^kA_l^\beta)) \D_t^{2r}\eta^i\ri^2  J^{-2}dxdt\label{eq.2r.12}\\
  &-\intt \D_t^{2r-2}\pd{\eta^l}{k}\D_t^3\pd{(J(A_i^\beta A_l^k-A_i^kA_l^\beta))}{\beta} \D_t^{2r}\eta^i\ri^2  J^{-2}dxdt\label{eq.2r.13}\\
  &-\intt \D_t^{2r-2}\pd{\eta^l}{k}\D_t^3(J(A_i^\beta A_l^k-A_i^kA_l^\beta)) \D_t^{2r}\eta^i\pd{(\ri^2  J^{-2})}{\beta}dxdt,\label{eq.2r.14}
\end{align}
which, then, can be controlled easily by the desired bound by an $L^2$-$L^3$-$L^6$ H\"older inequality, in addition for \eqref{eq.2r.13}, with the help of
\begin{align}
  &\norm{\D_t^{2m}D^{r-1-m}\eta_t}_{L^3([0,T]\times\Omega)}^2\no\\
  \ls& CT^{2/3}\Big[\norm{\D_t^{2m}D^{r-1-m}v(0)}_{L^3(\Omega)}^2 +\sup_{[0,T]}\norm{ \D_t^{2m}D^{r-1-m}\eta}_1\norm{\D_t^{2m+2}D^{r-1-m}\eta}_0\Big]\no\\
  \ls& CT^{2/3}\Big[\norm{\D_t^{2m}D^{r-1-m}v(0)}_{L^3(\Omega)}^2 +\sup_{[0,T]}\norm{ \D_t^{2m}\eta}_{r-m}\norm{\D_t^{2m+2}\eta}_{r-1-m}\Big]\no\\
  \ls& M_0+CTP(\sup_{[0,T]}E_r(t)).\label{eq.vL3}
\end{align}
for the integer $0\ls m \ls r-1$ due to \eqref{eq.interpolation}.

For the case $s=4$,
\begin{align}
  &\intt \D_t^{2r-3}\pd{\eta^l}{k}\D_t^4(J(A_i^jA_l^k-A_i^kA_l^j)) \pd{\D_t^{2r}\eta^i}{j}\ri^2  J^{-2}dxdt\label{eq.2r.15}\\
  =&\intt \D_t^{2r-3}\pd{\eta^l}{k}\D_t^4(J(A_i^\beta A_l^k-A_i^kA_l^\beta)) \pd{\D_t^{2r}\eta^i}{\beta}\ri^2  J^{-2}dxdt\label{eq.2r.16}\\
  &+\intt \D_t^{2r-3}\pd{\eta^l}{\beta}\D_t^4(J(A_i^3A_l^\beta-A_i^\beta A_l^3)) \pd{\D_t^{2r}\eta^i}{3}\ri^2  J^{-2}dxdt.\label{eq.2r.17}
\end{align}
It is clear that \eqref{eq.2r.17} is easy to be dealt with by an $L^6$-$L^3$-$L^2$ H\"older inequality. For \eqref{eq.2r.16}, integration by parts yields
\begin{align}
  \eqref{eq.2r.16}=&-\intt \D_t^{2r-3}\pd{\eta^l}{k\beta}\D_t^4(J(A_i^\beta A_l^k-A_i^kA_l^\beta)) \D_t^{2r}\eta^i\ri^2  J^{-2}dxdt\label{eq.2r.18}\\
  &-\intt \D_t^{2r-3}\pd{\eta^l}{k}\D_t^4\pd{(J(A_i^\beta A_l^k-A_i^kA_l^\beta))}{\beta} \D_t^{2r}\eta^i\ri^2  J^{-2}dxdt\label{eq.2r.19}\\
  &-\intt \D_t^{2r-3}\pd{\eta^l}{k}\D_t^4(J(A_i^\beta A_l^k-A_i^kA_l^\beta)) \D_t^{2r}\eta^i\pd{(\ri^2  J^{-2})}{\beta}dxdt.\label{eq.2r.20}
\end{align}
By an $L^2$-$L^3$-$L^6$ H\"older inequality, \eqref{eq.2r.18} and \eqref{eq.2r.20} can be easily estimated. We can also use an $L^3$-$L^2$-$L^6$ H\"older inequality to control \eqref{eq.2r.19} since $\D_t^{2r-4}\nb\eta_t\in L^3([0,T]\times\Omega)$ due to \eqref{eq.vL3}.

For the case $s=5$,
\begin{align}
  &\intt \D_t^{2r-4}\pd{\eta^l}{k}\D_t^5(J(A_i^jA_l^k-A_i^kA_l^j)) \pd{\D_t^{2r}\eta^i}{j}\ri^2  J^{-2}dxdt\label{eq.2r.21}\\
  =&\intt \D_t^{2r-4}\pd{\eta^l}{k}\D_t^5(J(A_i^\beta A_l^k-A_i^kA_l^\beta)) \pd{\D_t^{2r}\eta^i}{\beta}\ri^2  J^{-2}dxdt\label{eq.2r.22}\\
  &+\intt \D_t^{2r-4}\pd{\eta^l}{\beta}\D_t^5(J(A_i^3A_l^\beta-A_i^\beta A_l^3)) \pd{\D_t^{2r}\eta^i}{3}\ri^2  J^{-2}dxdt.\label{eq.2r.23}
\end{align}
It is easy to see that \eqref{eq.2r.23} is well estimated by an $L^6$-$L^3$-$L^2$ H\"older inequality. For \eqref{eq.2r.22}, integration by parts implies
\begin{align}
  \eqref{eq.2r.22}=&-\intt \D_t^{2r-4}\pd{\eta^l}{k\beta}\D_t^5(J(A_i^\beta A_l^k-A_i^kA_l^\beta)) \D_t^{2r}\eta^i\ri^2  J^{-2}dxdt\label{eq.2r.24}\\
  &-\intt \D_t^{2r-4}\pd{\eta^l}{k}\D_t^5\pd{(J(A_i^\beta A_l^k-A_i^kA_l^\beta))}{\beta} \D_t^{2r}\eta^i\ri^2  J^{-2}dxdt\label{eq.2r.25}\\
  &-\intt \D_t^{2r-4}\pd{\eta^l}{k}\D_t^5(J(A_i^\beta A_l^k-A_i^kA_l^\beta)) \D_t^{2r}\eta^i\pd{(\ri^2  J^{-2})}{\beta}dxdt.\label{eq.2r.26}
\end{align}
By an $L^2$-$L^3$-$L^6$ H\"older inequality, \eqref{eq.2r.24} and \eqref{eq.2r.26} can be easily estimated. We can use an $L^3$-$L^2$-$L^6$ H\"older inequality to control \eqref{eq.2r.19} because of $\ri \D_t^5\hd\nb\eta\in L^2(\Omega)$ in view of the fundamental theorem of calculus.

For the case $s=6$, integration by parts gives
\begin{align}
  &\intt \D_t^{2r-5}\pd{\eta^l}{k}\D_t^6(a_i^jA_l^k-A_i^ka_l^j) \pd{\D_t^{2r}\eta^i}{j}\ri^2  J^{-2}dxdt\label{eq.2r.27}\\
  =&-\intt \D_t^{2r-5}\pd{\eta^l}{kj}\D_t^6(a_i^jA_l^k-A_i^ka_l^j) \D_t^{2r}\eta^i\ri^2  J^{-2}dxdt\label{eq.2r.29}\\
  &-\intt \D_t^{2r-5}\pd{\eta^l}{k}\D_t^6(a_i^jA_l^k-A_i^ka_l^j) \D_t^{2r}\eta^i\pd{(\ri^2  J^{-2})}{j}dxdt.\label{eq.2r.28}
\end{align}
For \eqref{eq.2r.29}, it can be controlled by  the bound $M_0+\delta\sup_{[0,T]}E_r(t) +CTP(\sup_{[0,T]}E_r(t))$ by an $L^3$-$L^2$-$L^6$ H\"older inequality with the help of \eqref{eq.vL3}. For \eqref{eq.2r.28}, it is easily to be controlled by the desired bound.

For the cases $s=7$, and $s=8,9$ with $r=5$, they are easy to be controlled by the desired bound via the H\"older inequality.

Next, we consider \eqref{eq.2r.7}. For the case $s=2r-1$, it is easy to get the desired bounds by the H\"older inequality and the Sobolev embedding theorem. For other cases of $s$, it is the same as the cases of $s-1$ in \eqref{eq.2r.6}. Thus, we get the desired bounds for \eqref{eq.2r.7}.

The spatial integral \eqref{eq.2r.5} can be treated similarly as for \eqref{eq.2r.7} with the help of the fundamental theorem of calculus for one lower order term in order to get the factor $T$. For the last integral \eqref{eq.2r.8}, it is much easier to get the bound than \eqref{eq.2r.7}, thus we omit the details. Therefore, we have obtained
\begin{align*}
  \labs{\int_0^T\eqref{eq.2r.3}dt}\ls M_0+\delta \sup_{[0,T]}E_5(t)+CTP(\sup_{[0,T]}E_5(t)).
\end{align*}

\textbf{Step 4. Analysis of the remainder $\int_0^T\eqref{eq.2r.4}dt$.} By integration by parts with respect to time, we get
\begin{align}
-\int_0^T\eqref{eq.2r.4}dt=&-2\sum_{s=0}^{2r-2}C_{2r-1}^s \intt\D_t^{2r}\pd{v^i}{j}\D_t^{2r-1-s} (J^{-2}A_l^k)\D_t^s \pd{v^l}{k}\ri^2  J A^j_i dxdt\no\\
  =&-2\sum_{s=0}^{2r-2}C_{2r-1}^s \int_\Omega\D_t^{2r}\pd{\eta^i}{j}\D_t^{2r-1-s} (J^{-2}A_l^k)\D_t^s \pd{v^l}{k}\ri^2  J A^j_i dx\Big|_0^T\label{eq.2r.40}\\
  &+2\sum_{s=0}^{2r-2}C_{2r-1}^s \intt\D_t^{2r}\pd{\eta^i}{j}\D_t^{2r-s} (J^{-2}A_l^k)\D_t^s \pd{v^l}{k}\ri^2  J A^j_i dxdt\label{eq.2r.41}\\
  &+2\sum_{s=0}^{2r-2}C_{2r-1}^s \intt\D_t^{2r}\pd{\eta^i}{j}\D_t^{2r-1-s} (J^{-2}A_l^k)\D_t^{s+1} \pd{v^l}{k}\ri^2  J A^j_i dxdt\label{eq.2r.42}\\
  &+2\sum_{s=0}^{2r-2}C_{2r-1}^s \intt\D_t^{2r}\pd{\eta^i}{j}\D_t^{2r-1-s} (J^{-2}A_l^k)\D_t^s \pd{v^l}{k}\ri^2  \D_t(J A^j_i) dxdt.\label{eq.2r.43}
\end{align}

We first consider the double integral \eqref{eq.2r.41}. For the case $s=0$, we write it as
\begin{align*}
  &\intt\D_t^{2r}\pd{\eta^i}{j}\D_t^{2r} (J^{-2}A_l^k) \pd{v^l}{k}\ri^2  J A^j_i dxdt\\
  =&\sum_{m=0}^{2r}C_{2r}^m\intt\D_t^{2r}\pd{\eta^i}{j}\D_t^{m} J^{-2}\D_t^{2r-m}A_l^k \pd{v^l}{k}\ri^2  J A^j_i dxdt.
\end{align*}
For $m=0,3, 2r-1,2r$, we can use the H\"older inequality, \eqref{eq.vL3} and the Sobolev embedding theorem to get the desired bound. In particular, we have to use an $L^2$-$L^3$-$L^6$ H\"older inequality and \eqref{eq.vL3} to deal with the integral involving the terms of the form $\D_t^{2r}\nb\eta \D_t^{2r-3}\nb\eta \D_t^3\nb\eta$ in order to get the bound. For $m=1,2,4,\cdots, 2r-2$, we can only apply the H\"older inequality and the Sobolev embedding theorem to get the desired bound by noticing that $\ri \D_t^{2\ell}J^{-2}\in H^{r-\ell}(\Omega)$ for $0\ls \ell\ls r-1$.

For the case $s=1$, since $v_t\in H^{r-1}(\Omega)$, or $\nb v_t\in L^\infty(\Omega)$, it is similar to and easier than those of the case $s=0$. We omit the details.

For the case $s=2$, we have
\begin{align*}
  &\intt\D_t^{2r}\pd{\eta^i}{j}\D_t^{2r-2} (J^{-3}a_l^k)\D_t^2 \pd{v^l}{k}\ri^2  a^j_i dxdt\\
  =&\sum_{m=0}^{2r-2}C_{2r-2}^m\intt\D_t^{2r}\pd{\eta^i}{j} \D_t^m J^{-3}\D_t^{2r-2-m} a_l^k\D_t^2 \pd{v^l}{k}\ri^2  a^j_i dxdt.
\end{align*}
For $m=0$, i.e.,
\begin{align}\label{eq.why5order}
  \intt\D_t^{2r}\pd{\eta^i}{j}\D_t^{2r-2} a_l^k\D_t^2 \pd{v^l}{k}\ri^2  a^j_i J^{-3} dxdt,
\end{align}
we must use $E_5(t)$ to control $\norm{\D_t^2\nb v}_{L^\infty(\Omega)}$ when $r=4$; while it is easy to get the desired bound for $r=5$. For $m=1$, we can use an $L^2$-$L^3$-$L^6$ H\"older inequality and \eqref{eq.vL3} to obtain the desired bound, i.e., $M_0+\delta\sup_{[0,T]}E_r(t)+CTP(\sup_{[0,T]}E_r(t))$. For $m=2,\cdots, 2r-2$, they are controlled by the desired bounds by using the H\"older inequality and the Sobolev embedding theorem.

For the case $s=3$, we get
\begin{align*}
  &\intt\D_t^{2r}\pd{\eta^i}{j}\D_t^{2r-3} (J^{-2}A_l^k)\D_t^3 \pd{v^l}{k}\ri^2  J A^j_i dxdt\\
  =&\sum_{m=0}^{2r-3}C_{2r-3}^m\intt\D_t^{2r}\pd{\eta^i}{j} \D_t^m J^{-2}\D_t^{2r-3-m} A_l^k\D_t^3 \pd{v^l}{k}\ri^2  J A^j_i dxdt.
\end{align*}
For $m=0$, we can use an $L^2$-$L^3$-$L^6$ H\"older inequality and \eqref{eq.vL3} to obtain the desired bound. For $m=1,\cdots, 2r-3$, they are controlled by the desired bounds by using the H\"older inequality and the Sobolev embedding theorem.

For the case $s=4$, we can use an $L^2$-$L^6$-$L^3$ H\"older inequality, and \eqref{eq.vL3} for $r=4$ additionally, to obtain the desired bound. For the cases $s=5,\cdots, 2r-2$, we can use the H\"older inequality and the Sobolev embedding theorem to get the desired bounds.

Next, we consider the integral \eqref{eq.2r.42}. For the case $s=0$, we have
\begin{align*}
  &\intt\D_t^{2r}\pd{\eta^i}{j}\D_t^{2r-1} (J^{-2}A_l^k)\D_t \pd{v^l}{k}\ri^2  J A^j_i dxdt\\
  =&\sum_{m=0}^{2r-1}C_{2r-1}^m\intt\D_t^{2r}\pd{\eta^i}{j} \D_t^mJ^{-2}\D_t^{2r-1-m} A_l^k\D_t \pd{v^l}{k}\ri^2  J A^j_i dxdt.
\end{align*}
which can be controlled by the desired bound by using the H\"older inequality and the Sobolev embedding theorem.

For the case $s=1$, it follows that
\begin{align*}
  &\intt\D_t^{2r}\pd{\eta^i}{j}\D_t^{2r-2} (J^{-2}A_l^k)\D_t^2 \pd{v^l}{k}\ri^2  J A^j_i dxdt\\
  =&\sum_{m=0}^{2r-2}C_{2r-2}^m\intt\D_t^{2r}\pd{\eta^i}{j} \D_t^mJ^{-2}\D_t^{2r-2-m} A_l^k\D_t^2 \pd{v^l}{k}\ri^2  J A^j_i dxdt.
\end{align*}
For $m=0$, we can have to use the fact $v\in H^4(\Omega)$, which is contained in $E_5(t)$, for all cases $r=4,5$ to obtain
\begin{align*}
  \labs{\intt\D_t^{2r}\pd{\eta^i}{j} J^{-2}\D_t^{2r-2} A_l^k\D_t^2 \pd{v^l}{k}\ri^2  J A^j_i dxdt}\ls M_0+\delta\sup_{[0,T]}E_5(t)+CTP(\sup_{[0,T]}E_5(t)).
\end{align*}
For $m=1$, we can use an $L^2$-$L^3$-$L^6$ H\"older inequality, \eqref{eq.vL3} and the Sobolev embedding theorem to get the desired bound. For other cases of $m$, we can use the H\"older inequality and the Sobolev embedding theorem to get the desired bound by noticing that $\ri \D_t^{2\ell}J^{-2}\in H^{r-\ell}(\Omega)$ for $0\ls \ell\ls r-1$ with the help of the fundamental theorem of calculus if necessary.

For the case $s=2$, we get
\begin{align*}
  &\intt\D_t^{2r}\pd{\eta^i}{j}\D_t^{2r-3} (J^{-2}A_l^k)\D_t^3 \pd{v^l}{k}\ri^2  J A^j_i dxdt\\
  =&\sum_{m=0}^{2r-3}C_{2r-3}^m\intt\D_t^{2r}\pd{\eta^i}{j} \D_t^mJ^{-2}\D_t^{2r-3-m} A_l^k\D_t^3 \pd{v^l}{k}\ri^2  J A^j_i dxdt,
\end{align*}
which can be controlled by the desired bound by using the H\"older inequality and the Sobolev embedding theorem, in addition, with the help of \eqref{eq.vL3} for $m=0$.

For other cases of $s$, we can use similar argument to get the desired bounds and omit the details.

For the spatial integral \eqref{eq.2r.40}, we can use the same argument as for \eqref{eq.2r.42} to get the desired bound with the help of the fundamental theorem of calculus. For the double integral \eqref{eq.2r.43}, it is easier to get the bound than either \eqref{eq.2r.41} or \eqref{eq.2r.42} and thus we omit the details. Therefore, we obtain the estimates for $\int_0^T\eqref{eq.2r.4}dt$, i.e.,
\begin{align*}
  \labs{\int_0^T\eqref{eq.2r.4}dt}\ls M_0+\delta \sup_{[0,T]}E_5(t)+CTP(\sup_{[0,T]}E_5(t)).
\end{align*}

\textbf{Step 5. Analysis of the remainder $\int_0^T\I_3dt$.}
By integration by parts with respect to the spatial variables and the time variable, respectively, we obtain
\begin{align}
  \int_0^T\I_3dt=&\sum_{l=1}^{2r-1}C_{2r}^l\intt  \D_t^{2r-l}a_i^j \D_t^l J^{-2}\pd{\D_t^{2r}v^i}{j}\ri^2  dxdt\no\\
  =&\sum_{l=1}^{2r-1}C_{2r}^l\int_\Omega \D_t^{2r-l}a_i^j \D_t^l J^{-2}\pd{\D_t^{2r}\eta^i}{j}\ri^2  dx\Big|_0^T\label{eq.2r.46}\\
  &-\sum_{l=1}^{2r-1}C_{2r}^l\intt  \D_t^{2r+1-l}a_i^j \D_t^l J^{-2}\pd{\D_t^{2r}\eta^i}{j}\ri^2  dxdt\label{eq.2r.47}\\
  &-\sum_{l=1}^{2r-1}C_{2r}^l\intt  \D_t^{2r-l}a_i^j \D_t^{l+1} J^{-2} \pd{\D_t^{2r}\eta^i}{j}\ri^2  dxdt,\label{eq.2r.48}
\end{align}
due to $\D_t^{2r-l}a_i^3\D_t^{2r}v^i=0$ on $\Gamma_0$.

We first consider \eqref{eq.2r.47}. For the case $l=1$,
\begin{align}
  &\intt  \D_t^{2r}a_i^j \D_t J^{-2}\pd{\D_t^{2r}\eta^i}{j}\ri^2  dxdt\no\\
  =&\intt  \D_t^{2r-1}(\pd{\D_t\eta^p}{q}J^{-1}(a_i^ja_p^q-a_i^qa_p^j)) \D_t J^{-2}\pd{\D_t^{2r}\eta^i}{j}\ri^2  dxdt\no\\
  =&\sum_{s=1}^{2r-1}C_{2r-1}^s\intt  \pd{\D_t^{2r-s}\eta^p}{q}\D_t^s(J^{-1}(a_i^ja_p^q-a_i^qa_p^j)) \pd{\D_t^{2r}\eta^i}{j}\ri^2 \D_t J^{-2} dxdt\label{eq.2r.49}\\
  &+\intt  \pd{\D_t^{2r}\eta^p}{q} \pd{\D_t^{2r}\eta^i}{j}\ri^2 \D_t J^{-2}J(A_i^jA_p^q-A_i^qA_p^j) dxdt.\label{eq.2r.50}
\end{align}
Since $\D_t J^{-2}\in L^\infty(\Omega)$, we can use a similar argument as in \eqref{eq.2r.3} to get the estimates of \eqref{eq.2r.49}. For \eqref{eq.2r.50}, we easily have
\begin{align*}
  \abs{\eqref{eq.2r.50}}\ls&CT\sup_{[0,T]} \norm{\ri \D_t^{2r}\nb\eta}_0^2 \norm{\nb v}_2P(E_4(t))
  \ls CTP(\sup_{[0,T]}E_r(t)).
\end{align*}
For the case $l=2$, similar to the case $l=1$, we can get the bound easily since $\D_t^2 J^{-2}\in L^\infty(\Omega)$ and we omit the details. For the case $l=3$, we get
\begin{align}
  &\intt  \D_t^{2r-2}a_i^j \D_t^3 J^{-2}\pd{\D_t^{2r}\eta^i}{j}\ri^2  dxdt\no\\
  =&\intt  \D_t^{2r-3}(\pd{\D_t\eta^p}{q}J^{-1}(a_i^ja_p^q-a_i^qa_p^j)) \D_t^3 J^{-2}\pd{\D_t^{2r}\eta^i}{j}\ri^2  dxdt\no\\
  =&\sum_{s=1}^{2r-3}C_{2r-3}^s\intt  \pd{\D_t^{2r-2-s}\eta^p}{q}\D_t^s(J^{-1}(a_i^ja_p^q-a_i^qa_p^j)) \D_t^3 J^{-2}\pd{\D_t^{2r}\eta^i}{j}\ri^2  dxdt\label{eq.2r.51}\\
  &+\intt  \pd{\D_t^{2r-2}\eta^p}{q} \D_t^3 J^{-2}\pd{\D_t^{2r}\eta^i}{j}\ri^2  J^{-1}(a_i^ja_p^q-a_i^qa_p^j) dxdt.\label{eq.2r.52}
\end{align}
In \eqref{eq.2r.51}, we use an $L^3$-$L^6$-$L^2$ H\"older inequality and \eqref{eq.vL3} for the higher order terms of the cases $s=1$ and $s=2r-3$ and  an $L^6$-$L^6$-$L^6$-$L^2$ H\"older inequality for the other cases to get the desired bounds. For \eqref{eq.2r.52}, since $\ri \D_t^4 J^{-2}\in H^{r-2}(\Omega)\subset L^\infty(\Omega)$, we can get the desired bound easily in view of the H\"older inequality, the Sobolev embedding theorem and the fundamental theorem of calculus.
For the case $l=4$, we have the desired bound as a similar argument as for the case $l=3$. For the case $l=2r-3$, we can use an $L^6$-$L^3$-$L^2$ H\"older inequality and \eqref{eq.vL3} to get the desired bound. For the case $l=2r-2$, we use an $L^3$-$L^6$-$L^2$ H\"older inequality and the Sobolev embedding theorem to get the bound due to $\ri \D_t^{2(r-1)} J^{-2}\in H^1(\Omega)\subset L^6(\Omega)$. For the case $l=2r-1$, it is similar to the case $s=1$ in \eqref{eq.2r.41} and we omit the details. For the other cases, we can easily get the desired bounds by using the H\"older inequality and the Sobolev embedding theorem.

Next, we consider \eqref{eq.2r.48}. Since the cases $1\ls l\ls 2r-2$ are identical to the cases $2\ls l\ls 2r-1$ of \eqref{eq.2r.47} estimated just discussed up to some constant multipliers, we only need to consider the remainder case $l=2r-1$. We can apply \eqref{eq.Jt} to split the integral of the case $l=2r-1$ into two integrals. One of them can be used an $L^2$-$L^2$ H\"older inequality to get the estimates, the other one can be dealt with as the same arguments as for the case $l=2r-1$ of \eqref{eq.2r.47} or the case $s=1$ in \eqref{eq.2r.41}. Thus, we omit the details.

For the spatial integral \eqref{eq.2r.46}, it can be estimated as the same arguments as for \eqref{eq.2r.47} or \eqref{eq.2r.48} with the help of the fundamental theorem of calculus whose details are omitted.

\textbf{Step 6. Summing inequalities.} As the same argument as in the estimates of the horizontal derivatives, we can obtain the desired result by combining the previous inequalities.
\end{proof}

\section{The estimates for the mixed time-horizontal derivatives}\label{sec.mix}

We have the following estimates.

\begin{proposition}\label{prop.mix1}
  Let $r\in\{4,5\}$ and $1\ls m\ls r-1$. For $\delta>0$ and the constant $M_0$ depend on $1/\delta$, we have
  \begin{align*}
    &\sup_{[0,T]}\Big[\norm{\ri^{1/2}\D_t^{2m}\hd^{r-m} v}_0^2+\norm{\ri \nb\D_t^{2m}\hd^{r-m}\eta}_0^2 +\norm{\ri \dv\D_t^{2m}\hd^{r-m}\eta}_0^2+\abs{\D_t^{2m}\eta^\alpha}_{r-m-1/2}^2\Big]\\
    \ls & M_0+\delta\sup_{[0,T]}E_r(t)+CTP(\sup_{[0,T]}E_r(t)).
  \end{align*}
\end{proposition}

\begin{proof}
  Applying the differential operator $\D_t^{2m}\hd^{r-m}$ on \eqref{eu4.2} and taking the $L^2(\Omega)$-inner product with $\D_t^{2m}\hd^{r-m} v^i$, we have, by integration by parts, \eqref{eu4.6} and \eqref{eu4.7}, that
  \begin{align*}
  \frac{1}{2}\frac{d}{dt}\int_\Omega\ri \abs{\D_t^{2m}\hd^{r-m} v}^2 dx+\I_1+\I_2  =\I_3,
\end{align*}
where
\begin{align}
  \I_1:=&-\int_\Omega \D_t^{2m}\hd^{r-m}a^j_i\pd{\D_t^{2m}\hd^{r-m} v^i}{j} \ri^2 J^{-2}dx,\no\\
  \I_2:=&-\int_\Omega a^j_i \D_t^{2m}\hd^{r-m}J^{-2}\pd{\D_t^{2m}\hd^{r-m} v^i}{j} \ri^2  dx,\no\\
  \I_3:=&-\sum_{l=0}^{r-m-1} C_{r-m}^l\int_\Omega\hd^{r-m-l}\ri \hd^l\D_t^{2m} v_t^i\D_t^{2m}\hd^{r-m} v^i dx\label{eq.mix.1}\\
  &+\sum_{s=1}^{2m}\sum_{l=0}^{r-m-1}\sum_{k=0}^l C_{2m}^s C_{r-m}^l C_l^k \int_\Omega\hd^{l-k}\ri^2  \hd^{r-m-l}\D_t^{2m-s}a^j_i \D_t^s\hd^kJ^{-2}\pd{\D_t^{2m}\hd^{r-m} v^i}{j} dx\label{eq.mix.2}\\
  &+\sum_{s=1}^{2m}\sum_{k=0}^{r-m-1} C_{2m}^s  C_{r-m}^k \int_\Omega\hd^{r-m-k}\ri^2  \D_t^{2m-s}a^j_i \D_t^s\hd^kJ^{-2}\pd{\D_t^{2m}\hd^{r-m} v^i}{j} dx\label{eq.mix.3}\\
  &+\sum_{l=1}^{r-m}\sum_{k=0}^l  C_{r-m}^l C_l^k \int_\Omega\hd^{l-k}\ri^2  \hd^{r-m-l}\D_t^{2m}a^j_i \hd^kJ^{-2}\pd{\D_t^{2m}\hd^{r-m} v^i}{j} dx.\label{eq.mix.4}
\end{align}

By \eqref{eq.adt} and \eqref{eq.curletaF2}, it follows that
\begin{align}
  \I_1=&-\int_\Omega \D_t^{2m}\hd^{r-m}\pd{\eta^p}{q}\pd{\D_t^{2m}\hd^{r-m} v^i}{j} \ri^2 J^{-1}(A^j_iA^q_p-A^q_iA^j_p)dx\no\\
  &-\int_\Omega \pd{v^p}{q}\D_t^{2m-1}\hd^{r-m}[J(A^j_iA^q_p-A^q_iA^j_p)]\pd{\D_t^{2m}\hd^{r-m} v^i}{j} \ri^2 J^{-2}dx\label{eq.mix.5}\\
  =&-\frac{1}{2}\frac{d}{dt}\int_\Omega \D_t^{2m}\hd^{r-m}\pd{\eta^p}{q}\pd{\D_t^{2m}\hd^{r-m} \eta^i}{j}(A^j_iA^q_p-A^q_iA^j_p) \ri^2 J^{-1}dx+\eqref{eq.mix.5}\no\\
  &+\frac{1}{2}\int_\Omega \D_t^{2m}\hd^{r-m}\pd{\eta^p}{q}\pd{\D_t^{2m}\hd^{r-m} \eta^i}{j} \ri^2 \D_t[J^{-1}(A^j_iA^q_p-A^q_iA^j_p)]dx\label{eq.mix.6}\\
  =&\frac{1}{2}\frac{d}{dt}\int_\Omega \Big[\abs{\nb_\eta \D_t^{2m}\hd^{r-m} \eta}^2-\abs{\dveta \D_t^{2m}\hd^{r-m} \eta}^2-\abs{\curleta \D_t^{2m}\hd^{r-m} \eta}^2\Big]\ri^2 J^{-1}dx\no\\
  &+\eqref{eq.mix.5}+\eqref{eq.mix.6}.\no
\end{align}

By \eqref{eq.Jt}, we have
\begin{align}
  \I_2=&2\int_\Omega  A^j_i \D_t^{2m}\hd^{r-m}\pd{\eta^p}{q}\pd{\D_t^{2m}\hd^{r-m} v^i}{j} \ri^2  J^{-1}A^q_p dx\no\\
  &+2\int_\Omega  a^j_i\pd{v^p}{q} \D_t^{2m-1}\hd^{r-m}(J^{-2}A^q_p)\pd{\D_t^{2m}\hd^{r-m} v^i}{j} \ri^2  dx\label{eq.mix.7}\\
  =&\frac{d}{dt}\int_\Omega \abs{\dveta D_t^{2m}\hd^{r-m} \eta}^2\ri^2  J^{-1}dx+\eqref{eq.mix.7}\no\\
  &-\int_\Omega  \D_t^{2m}\hd^{r-m}\pd{\eta^p}{q}\pd{\D_t^{2m}\hd^{r-m} \eta^i}{j} \ri^2  \D_t(J^{-1}A^j_i A^q_p) dx.\label{eq.mix.8}
\end{align}

Thus, we get
\begin{align*}
  \I_1+\I_2=&\frac{1}{2}\frac{d}{dt}\int_\Omega \Big[\abs{\nb_\eta D_t^{2m}\hd^{r-m} \eta}^2+\abs{\dveta D_t^{2m}\hd^{r-m} \eta}^2-\abs{\curleta D_t^{2m}\hd^{r-m} \eta}^2\Big]\ri^2 J^{-1}dx\\
  &+\eqref{eq.mix.5}+\eqref{eq.mix.6}+\eqref{eq.mix.7}+\eqref{eq.mix.8}.
\end{align*}

Now, we analyze the integration with respect to time of the remainder integrals \eqref{eq.mix.1}-\eqref{eq.mix.8} and $\I_4$.

By the higher order Hardy inequality, the H\"older inequality, we have
\begin{align*}
  \labs{\int_0^T\eqref{eq.mix.1} dt}\ls &CT\norm{\ri}_{r+1}^2\norm{\ri}_2^{1/2}\sup_{[0,T]}\norm{\ri\hd^l\D_t^{2m+2}\nb\eta}_0 \norm{\ri^{1/2}\D_t^{2m}\hd^{r-m}v}_0\\
  \ls &M_0+\delta\sup_{[0,T]}\norm{\ri^{1/2}\D_t^{2m}\hd^{r-m}v}_0^2+ CTP(\sup_{[0,T]} E_r(t)).
\end{align*}
As a similar arguments as for the remainder integrals in Section \ref{sec.7}, we can get the integrations over $[0,T]$ of \eqref{eq.mix.2}-\eqref{eq.mix.8} can be bounded by
$$M_0+\delta\sup_{[0,T]}E_r(t)+ CTP(\sup_{[0,T]} E_r(t)).$$

Therefore, we can get the desired estimates by similar arguments as in Sections \ref{sec.7} and \ref{sec.2r}.

Finally, we show the boundary estimates.
By Lemma \ref{lem.tangential.trace}, \eqref{w-em}, the fundamental theorem of calculus, H\"older's inequality and curl estimates, we obtain
\begin{align*}
  &\abs{\D_t^{2m}\hd^{r-m} \eta^\alpha}_{-1/2}^2\lesssim \norm{\hd^{r-m} \D_t^{2m}\eta}_0^2+\norm{\curl\D_t^{2m}\hd^{r-m-1}\eta}_0^2\\
  \lesssim &\norm{\ri  \D_t^{2m}\hd^{r-m} \eta}_0^2+\norm{\ri \nb\D_t^{2m}\hd^{r-m} \eta}_0^2+\norm{\curl\D_t^{2m}\hd^{r-m-1}\eta}_0^2\\
  \lesssim & \norm{\ri  \D_t^{2m}\hd^{r-m} \eta(0)}_0^2+\lnorm{\ri \int_0^t\D_t^{2m}\hd^{r-m} v dt}_0^2+\norm{\ri \nb\D_t^{2m}\hd^{r-m} \eta}_0^2+\norm{\curl\D_t^{2m}\hd^{r-m-1}\eta}_0^2\\
  \lesssim &M_0+\norm{\ri}_{L^\infty(\Omega)} T^2\sup_{[0,T]} \norm{\ri^{1/2}\D_t^{2m}\hd^{r-m} v}_0^2+\norm{\ri \nb\D_t^{2m}\hd^{r-m} \eta}_0^2+\norm{\curl\D_t^{2m}\hd^{r-m-1}\eta}_0^2\\
  \ls &M_0+\delta\sup_{[0,T]}(E_r(t) +CTP(\sup_{[0,T]}(E_r(t)).
\end{align*}
Thus, we complete the proof.
\end{proof}

\section{The elliptic-type estimates for the normal derivatives}\label{sec.8}

Our energy estimates provide a priori control of horizontal and time derivatives of $\eta$; it remains to gain a priori control of the normal derivatives of $\eta$. This is accomplished via a bootstrapping procedure relying on the fact $\D_t^{9}v(t)$ is bounded in $L^2(\Omega)$.

\begin{proposition}\label{prop.dt7v}
  For $t\in [0,T]$, it holds that
  \begin{align*}
    \sup_{[0,T]}\Big[\norm{\D_t^{7} v(t)}_1^2+\norm{\ri  \D_t^8J^{-2}(t)}_1^2\Big]\ls  M_0+\delta \sup_{[0,T]} E_5(t)+CTP(\sup_{[0,T]}E_5(t)).
  \end{align*}
\end{proposition}

\begin{proof}
From \eqref{eu4.2}, we have for $\beta=1,2$
\begin{align}\label{eu0}
&\ri a_i^3 \pd{J^{-2}}{3}+2 a_i^3\pd{\ri}{3}J^{-2}
=-v_t^i-\ri a_i^\beta\pd{J^{-2}}{\beta}-2  a_i^\beta\pd{\ri}{\beta}J^{-2}.
\end{align}

Acting $\D_t^8$ on \eqref{eu0}, we get
\begin{align*}
&\ri  a_i^3 \D_t^8\pd{J^{-2}}{3}+2 a_i^3\pd{\ri}{3} \D_t^8 J^{-2}\\
=&-\D_t^9v^i-\ri  \D_t^8(a_i^\beta\pd{J^{-2}}{\beta})-2 \pd{\ri}{\beta} \D_t^8(a_i^\beta J^{-2})\\
&-\D_t^8 a_i^3[ \ri  \pd{J^{-2}}{3}+2 \pd{\ri}{3}J^{-2}]-\sum_{l=1}^7 C_8^l\D_t^{8-l} a_i^3\D_t^l[ \ri  \pd{J^{-2}}{3}+2 \pd{\ri}{3}J^{-2}].
\end{align*}

By \eqref{w-em}, the fundamental theorem of calculus and H\"older's inequality, we have
\begin{align*}
\norm{\D_t^9v(t)}_0^2\ls& C\int_\Omega \ri^2 \left(\abs{\D_t^9 v}^2+\abs{\nb\D_t^9 v}^2\right)dx\\
  \ls &C\int_\Omega \ri^2 \labs{\int_0^t \D_t^{10} vdt'+\D_t^9 v(0)}^2 dx+C\norm{\ri \nb\D_t^9v}_0^2\\
  \ls& C t\int_\Omega \ri^2 \int_0^t \abs{\D_t^{10} v}^2 dt'+C\int_\Omega \ri^2 \abs{\D_t^9 v(0)}^2dx+C\norm{\ri  \D_t^9\nb v}_0^2\\
  \ls&Ct^2\norm{\ri}_{L^\infty(\Omega)}\sup_{[0,t]} \norm{\ri^{1/2}\D_t^{10} v}_0^2+\norm{\ri}_{L^\infty(\Omega)}^{2}\norm{\D_t^9 v(0)}_0^2+C\norm{\ri  \D_t^9\nb v}_0^2.
\end{align*}

By the fundamental inequality of algebra, the fundamental theorem of calculus, the H\"older inequality and the Sobolev embedding theorems, we see that
\begin{align*}
  \norm{\ri  \D_t^8(a_i^\beta\pd{J^{-2}}{\beta})}_0^2
  \ls&C\sum_{l=0}^4\norm{\ri  \D_t^{8-2l} a_i^\beta \D_t^{2l}\pd{J^{-2}}{\beta}}_0^2\\
  &+C\sum_{l=1}^4\lnorm{\ri  \D_t^{9-2l} a_i^\beta \left(\int_0^t\D_t^{2l}\pd{J^{-2}}{\beta}dt' +\D_t^{2l}\pd{J^{-2}}{\beta}(0)\right)}_0^2\\
  \ls M_0+&Ct^2\sup_{[0,t]}\sum_{l=0}^4\norm{\ri  \D_t^{2l}J^{-2}}_{5-l}^2 P(\norm{\D_t^{8-2l}\eta}_{l+1},\cdots,\norm{\eta_{tt}}_4, \norm{\eta_t}_4,\norm{\eta}_5).
\end{align*}
Similarly, we also have, that
\begin{align*}
  &\norm{\D_t^8 a_i^3 \ri  \pd{J^{-2}}{3}}_0^2
  \ls M_0+Ct^2\sup_{[0,t]}\norm{\ri J^{-2}}_{5}^2 P(\norm{\D_t^{8}\eta}_{1},\cdots,\norm{\eta_{tt}}_4, \norm{\eta_t}_4,\norm{\eta}_5),
\end{align*}
and
\begin{align*}
  &\norm{ \pd{\ri}{\beta} \D_t^8(a_i^\beta J^{-2})}_0^2\\
  \ls&M_0+Ct^2\lnorm{\frac{\pd{\ri}{\beta}}{\ri}}_{L^\infty(\Omega)}^2\sup_{[0,t]}\sum_{l=0}^4 \norm{\ri  \D_t^{2l}J^{-2}}_{5-l}^2  P(\norm{\D_t^{8-2l}\eta}_{l+1},\cdots,\norm{\eta_{tt}}_4, \norm{\eta_t}_4,\norm{\eta}_5).
\end{align*}
By the fundamental theorem of calculus, we get
\begin{align*}
\norm{\D_t^8 a_i^3 \pd{\ri}{3}J^{-2}}_0^2
  \ls M_0+C&t^2\norm{\pd{\ri}{3}}_{L^\infty(\Omega)}^2 \\
  &\sup_{[0,t]}P(\norm{\D_t^{8}\eta}_{1},\cdots, \norm{\D_t^{8-2l}\eta}_{l+1},\cdots, \norm{\eta_{tt}}_4, \norm{\eta_t}_4,\norm{\eta}_5).
\end{align*}
Similarly, it follows that
\begin{align*}
  &\sum_{l=1}^7\lnorm{\D_t^{8-l} a_i^3\D_t^l\left[ \ri  \pd{J^{-2}}{3}+2 \pd{\ri}{3}J^{-2}\right]}_0^2\\
  \ls&M_0+Ct^2 \sup_{[0,t]}\sum_{l=0}^4\norm{\ri  \D_t^{2l}J^{-2}}_{5-l}^2  P(\norm{\D_t^{8-2l}\eta}_{l+1}^2,\cdots,\norm{\eta_{tt}}_4, \norm{\eta_t}_4,\norm{\eta}_5)\\
  &+Ct^2\norm{\pd{\ri}{3}}_{L^\infty(\Omega)}^2 \sup_{[0,t]}P(\norm{\D_t^{8}\eta}_{1},\cdots,\norm{\D_t^{8-2l}\eta}_{l+1},\cdots, \norm{\eta_{tt}}_4, \norm{\eta_t}_4,\norm{\eta}_5).
\end{align*}
Thus, we have obtained, for all $t\in [0,T]$, that
\begin{align*}
  \lnorm{\ri  a_i^3 \D_t^8\pd{J^{-2}}{3}+2 a_i^3\pd{\ri}{3} \D_t^8 J^{-2}}_0^2\ls M_0+\delta \sup_{[0,T]} E_5(t)+CTP(\sup_{[0,T]}E_5(t)).
\end{align*}
It follows that
\begin{align*}
  &\norm{\ri  a_\cdot^3 \D_t^8\pd{J^{-2}}{3}}_0^2+4\norm{ a_\cdot^3\pd{\ri}{3} \D_t^8 J^{-2}}_0^2\\
  \ls& M_0+\delta \sup_{[0,T]} E_5(t)+CTP(\sup_{[0,T]}E_5(t))-4\int_\Omega \ri\pd{\ri}{3} \abs{a_\cdot^3}^2 \D_t^8\pd{J^{-2}}{3} \D_t^8 J^{-2} dx.
\end{align*}
By H\"older's inequality and the fundamental theorem of calculus, we get
\begin{align*}
  &\labs{-4\int_\Omega \ri \pd{\ri}{3} \abs{a_\cdot^3}^2 \D_t^8\pd{J^{-2}}{3} \D_t^8 J^{-2} dx}\\
  \ls& M_0+\delta \norm{\ri \pd{\D_t^8 J^{-2}}{3}}_0^2+\delta \norm{\pd{\ri}{3}\D_t^8 J^{-2}}_0^2+CTP(\sup_{[0,T]}E_5(t)),
\end{align*}
and then, by the fundamental theorem of calculus once again,
\begin{align*}
  &\norm{\ri   \D_t^8\pd{J^{-2}}{3}}_0^2+\norm{ \pd{\ri}{3} \D_t^8 J^{-2}}_0^2
  \ls M_0+\delta \sup_{[0,T]} E_5(t)+CTP(\sup_{[0,T]}E_5(t)).
\end{align*}
By integration by parts with respect to $x_3$, the H\"older inequality and noticing that $\ri=0$ on $\Gamma_1$, it holds
\begin{align*}
\norm{\ri  \D_t^8J^{-2}(t)}_0^2=&\int_\Omega \D_{3}x_3 (\ri  \D_t^8J^{-2})^2dx\\
  =&-2\int_\Omega x_3\ri  \D_t^8J^{-2}(\pd{\ri}{3}\D_t^8J^{-2} +\ri \pd{\D_t^8J^{-2}}{3})dx\\
  \ls &2\norm{\ri  \D_t^8J^{-2}(t)}_0\big[ \norm{\pd{\ri}{3}\D_t^8J^{-2}}_0 +\norm{\ri \pd{\D_t^8J^{-2}}{3}}_0\big],
\end{align*}
which implies that
\begin{align}\label{eq.dt6j0}
  \norm{\ri  \D_t^8J^{-2}(t)}_0^2\ls &8\Big[ \norm{\pd{\ri}{3}\D_t^8J^{-2}}_0^2 +\norm{\ri \pd{\D_t^8J^{-2}}{3}}_0^2\Big]\no\\
  \ls& M_0+\delta \sup_{[0,T]} E_5(t)+CTP(\sup_{[0,T]}E_5(t)).
\end{align}
Since we can get, by Proposition \ref{prop.mix1}, that
\begin{align*}
  \norm{\ri \hd\D_t^8 J^{-2}}_0^2\ls M_0+\delta \sup_{[0,T]} E_5(t)+CTP(\sup_{[0,T]}E_5(t)),
\end{align*}
we have
\begin{align*}
  \norm{\ri  \D_t^8 J^{-2}}_1^2\ls M_0+\delta \sup_{[0,T]} E_5(t)+CTP(\sup_{[0,T]}E_5(t)).
\end{align*}
It follows from \eqref{w-em} and \eqref{eq.dt6j0} that
\begin{align*}
  \norm{\D_t^8 J^{-2}}_0^2\ls& C\norm{\ri  \D_t^8 J^{-2}}_0^2+C\norm{\ri \nb\D_t^8 J^{-2}}_0^2
  \ls  M_0+\delta \sup_{[0,T]} E_5(t)+CTP(\sup_{[0,T]}E_5(t)).
\end{align*}

Due to \eqref{eq.Jt}, we see that
\begin{align*}
  \D_t^8 J^{-2}=&-2 \D_t^7(J^{-2}A_i^j\pd{v^i}{j})
  =-2 J^{-2}A_i^j\D_t^7\pd{v^i}{j}-2 \D_t^7 (J^{-2}A_i^j)\pd{v^i}{j}-2 \sum_{l=1}^6 C_7^l \D_t^l( J^{-2}A_i^j)\D_t^{7-l}\pd{v^i}{j},
\end{align*}
namely, in view of the fundamental theorem of calculus,
\begin{align*}
  \dv\D_t^7 v=&-\frac{1}{2}\D_t^8 J^{-2}-\D_t^7\pd{v^i}{j}\int_0^t(J^{-2}A_i^j)_tdt'- \D_t^7 (J^{-2}A_i^j)\pd{v^i}{j}
  - \sum_{l=1}^6 C_7^l \D_t^l( J^{-2}A_i^j)\D_t^{7-l}\pd{v^i}{j}.
\end{align*}
We can easily estimate last three terms by using the fundamental theorem of calculus and the H\"older inequality. Thus, we obtain
\begin{align*}
  \norm{\dv\D_t^7 v}_0^2\ls M_0+\delta \sup_{[0,T]} E_5(t)+CTP(\sup_{[0,T]}E_5(t)).
\end{align*}
According to Proposition \ref{prop.curl}, we have
\begin{align*}
  \norm{\curl\D_t^7 v}_0^2\ls M_0+CTP(\sup_{[0,T]}E_5(t)).
\end{align*}
With the boundary estimates on $\D_t^7 v^\alpha$ or  $\D_t^8 \eta^\alpha$ given by Proposition \ref{prop.mix1}, we obtain, from Lemma \ref{lem.Hodge}, that
\begin{align*}
  \norm{\D_t^7 v}_1^2\ls M_0+\delta \sup_{[0,T]} E_5(t)+CTP(\sup_{[0,T]}E_5(t)).
\end{align*}
Thus, we complete the proof.
\end{proof}

Having a good bound for $\D_t^7 v(t)$ in $H^1(\Omega)$, we proceed with the bootstrapping.

\begin{proposition}\label{prop.vt5}
  For $t\in [0,T]$, it holds that
  \begin{align*}
    \sup_{[0,T]}\Big[\norm{ \D_t^5 v(t)}_2^2+\norm{\ri  \D_t^6J^{-2}(t)}_2^2\Big]\ls  M_0+\delta \sup_{[0,T]} E_5(t)+CTP(\sup_{[0,T]}E_5(t)).
  \end{align*}
\end{proposition}

\begin{proof}
  Applying the differential operator $\D_t^6$  on \eqref{eu0},  we have
\begin{align}\label{eq.vt3}
    &\ri  a_i^3 \D_t^6\pd{J^{-2}}{3}+2 a_i^3\pd{\ri}{3} \D_t^6 J^{-2}\no\\
    =&-\D_t^7v^i-\ri  \D_t^6(a_i^\beta\pd{J^{-2}}{\beta})-2 \pd{\ri}{\beta} \D_t^6(a_i^\beta J^{-2})-\D_t^6 a_i^3\left[ \ri  \pd{J^{-2}}{3}+2 \pd{\ri}{3}J^{-2}\right]\no\\
    &-\sum_{l=1}^5 C_6^l\D_t^{6-l} a_i^3\D_t^l\left[ \ri  \pd{J^{-2}}{3}+2 \pd{\ri}{3}J^{-2}\right].
\end{align}

We first estimate horizontal derivatives of $\pd{\ri}{3} \D_t^6 J^{-2}(t)$ in $L^2(\Omega)$ to consider for $\alpha=1,2$,
\begin{align}\label{eq.v3.1}
  &\ri  a_i^3 \D_t^6\pd{J^{-2}}{3\alpha}+2 a_i^3\pd{\ri}{3} \pd{\D_t^6 J^{-2}}{\alpha}\no\\
    =&\Big[-\D_t^7v^i-\ri  \D_t^6(a_i^\beta\pd{J^{-2}}{\beta})-2 \pd{\ri}{\beta} \D_t^6(a_i^\beta J^{-2})-\D_t^6 a_i^3[ \ri  \pd{J^{-2}}{3}+2 \pd{\ri}{3}J^{-2}]\no\\
    &-\sum_{l=1}^5 C_6^l\D_t^{6-l} a_i^3\D_t^l[ \ri  \pd{J^{-2}}{3}+2 \pd{\ri}{3}J^{-2}]\pd{\Big]}{\alpha}-\pd{\ri }{\alpha}a^3_i\D_t^6\pd{J^{-2}}{3}-\ri  \pd{a_i^3}{\alpha} \D_t^6\pd{J^{-2}}{3}\no\\
    &-2\pd{\ri}{3\alpha} a_i^3\D_t^6 J^{-2}-2 \pd{a_i^3}{\alpha}\pd{\ri}{3}\D_t^6 J^{-2}.
\end{align}

Now, we estimate the $L^2(\Omega)$ norms of the right hand side. From Proposition \ref{prop.dt7v}, we know that
\begin{align*}
  \norm{\D_t^7\hd v}_0^2\ls \norm{\D_t^7 v}_1^2\ls M_0+\delta \sup_{[0,T]} E_5(t)+CTP(\sup_{[0,T]}E_5(t)).
\end{align*}
Since
\begin{align*}
  \lnorm{\pd{\Big[\ri  \D_t^6(a_i^\beta\pd{J^{-2}}{\beta})\Big]}{\alpha}}_0^2
  \lesssim &\norm{\pd{\ri }{\alpha}\D_t^6(a_i^\beta\pd{J^{-2}}{\beta})}_0^2 +\norm{\ri \pd{\D_t^6(a_i^\beta\pd{J^{-2}}{\beta})}{\alpha}}_0^2\\
  \lesssim& \norm{\pd{\ri }{\alpha}\D_t^6(J^{-1}A_i^\beta A^l_k\pd{\eta^k}{l\beta})}_0^2 +\norm{\ri \pd{\D_t^6(J^{-1}A_i^\beta A^l_k\pd{\eta^k}{l\beta})}{\alpha}}_0^2,
\end{align*}
we consider the last term involving the highest order derivatives
\begin{align*}
  \pd{\D_t^6(J^{-1}A_i^\beta A^l_k\pd{\eta^k}{l\beta})}{\alpha}
  =&\D_t^6\big[-J^{-1}A^p_q\pd{\eta^q}{p\alpha}A_i^\beta A^l_k\pd{\eta^k}{l\beta}+J^{-1}A^\beta_q\pd{\eta^q}{p\alpha}A^q_i A^l_k\pd{\eta^k}{l\beta}\\
  &+J^{-1}A^\beta_i A^l_q\pd{\eta^q}{p\alpha}A^q_k\pd{\eta^k}{l\beta}+J^{-1}A_i^\beta A^l_k\pd{\eta^k}{l\beta\alpha}\big].
\end{align*}
For the highest order derivatives term, by the fundamental theorem of calculus, Sobolev's embedding theorem and Proposition \ref{prop.mix1}, we have for $T>0$ small enough that
\begin{align*}
  \norm{\ri J^{-1}A_i^\beta A^l_k\pd{\D_t^6\eta^k}{l\beta\alpha}}_0^2
  \lesssim& \norm{\ri  \dv \D_t^6\hd^2\eta}_0^2+\norm{\ri  \int_0^t \D_t(J^{-1}A_i^\beta A_{k}^l)dt' \pd{ \D_t^6\hd\eta^k}{l\beta}}_0^2\\
  \ls& M_0+\delta \sup_{[0,T]} E_5(t)+CTP(\sup_{[0,T]}E_5(t))\\
  &+CT^2P(\sup_{[0,T]}\norm{\eta}_5)\norm{v}_3^2 \norm{\ri \nb\D_t^6\hd^2\eta}_0^2\\
  \ls& M_0+\delta \sup_{[0,T]} E_5(t)+CTP(\sup_{[0,T]}E_5(t)).
\end{align*}
For the lower order derivatives terms, we use a similar argument to get the bound. Thus, we can get
\begin{align*}
  \lnorm{\pd{\Big[\ri  \D_t^6(a_i^\beta\pd{J^{-2}}{\beta})\Big]}{\alpha}}_0^2\ls M_0+\delta \sup_{[0,T]} E_5(t)+CTP(\sup_{[0,T]}E_5(t)).
\end{align*}
We can deal with the remainder terms in the right hand side of \eqref{eq.v3.1} by the same argument to get the bound in $L^2(\Omega)$ and thus we have
\begin{align*}
  \lnorm{\ri  a_i^3 \D_t^6\pd{J^{-2}}{3\alpha}+2 a_i^3\pd{\ri}{3} \pd{\D_t^6 J^{-2}}{\alpha}}_0^2\ls M_0+\delta \sup_{[0,T]} E_5(t)+CTP(\sup_{[0,T]}E_5(t)).
\end{align*}
It follows that
\begin{align*}
  &\norm{\ri  a_i^3 \D_t^6\pd{J^{-2}}{3\alpha}}_0^2+4\norm{ a_i^3\pd{\ri}{3} \pd{\D_t^6 J^{-2}}{\alpha}}_0^2\\
  \ls& M_0+\delta \sup_{[0,T]} E_5(t)+CTP(\sup_{[0,T]}E_5(t))-4 \int_\Omega \ri\pd{\ri}{3}\abs{a_\cdot^3}^2\D_t^6\pd{J^{-2}}{3\alpha}\pd{\D_t^6 J^{-2}}{\alpha}dx.
\end{align*}
By H\"older's inequality and the fundamental theorem of calculus, we get
\begin{align*}
  &\labs{-4\int_\Omega \ri^{22-3}\pd{\ri}{3} \abs{a_\cdot^3}^2 \D_t^6\pd{J^{-2}}{3\alpha}\pd{\D_t^6 J^{-2}}{\alpha} dx}\\
  \ls &M_0+\delta \norm{\ri \pd{\D_t^6 J^{-2}}{3\alpha}}_0^2+\delta \norm{\pd{\ri}{3}\pd{\D_t^6 J^{-2}}{\alpha}}_0^2+CTP(\sup_{[0,T]}E_5(t)),
\end{align*}
and then, by the fundamental theorem of calculus once again,
\begin{align*}
  &\norm{\ri   \D_t^6\pd{J^{-2}}{3\alpha}}_0^2+\norm{\pd{\ri}{3} \pd{\D_t^6 J^{-2}}{\alpha}}_0^2
  \ls M_0+\delta \sup_{[0,T]} E_5(t)+CTP(\sup_{[0,T]}E_5(t)).
\end{align*}
Similar to \eqref{eq.dt6j0}, we get
\begin{align*}
  \norm{\ri  \D_t^6\pd{J^{-2}}{\alpha}}_0^2\ls M_0+\delta \sup_{[0,T]} E_5(t)+CTP(\sup_{[0,T]}E_5(t)).
\end{align*}
From Proposition \ref{prop.mix1}, we can obtain that
\begin{align*}
  \norm{\ri  \D_t^6\hd^2 J^{-2}}_0^2\ls M_0+\delta \sup_{[0,T]} E_5(t)+CTP(\sup_{[0,T]}E_5(t)).
\end{align*}
Thus, we have
\begin{align*}
  \norm{\ri \pd{\D_t^6 J^{-2}}{\alpha}}_1^2\lesssim &\norm{\ri \pd{\D_t^6 J^{-2}}{\alpha}}_0^2+\lnorm{\frac{\hd\ri}{\ri }}_{L^\infty(\Omega)}^2 \norm{\ri \pd{\D_t^6 J^{-2}}{\alpha}}_0^2\\
  &+\norm{\pd{\ri}{3} \pd{\D_t^6 J^{-2}}{\alpha}}_0^2+\norm{\ri \nb\pd{\D_t^6 J^{-2}}{\alpha}}_0^2\\
  \ls&M_0+\delta \sup_{[0,T]} E_5(t)+CTP(\sup_{[0,T]}E_5(t)),
\end{align*}
and by using the inequality \eqref{w-em},
\begin{align*}
  \norm{\pd{\D_t^6 J^{-2}}{\alpha}}_0^2\lesssim & \norm{\ri \pd{\D_t^6 J^{-2}}{\alpha}}_0^2+\norm{\ri \nb\pd{\D_t^6 J^{-2}}{\alpha}}_0^2
  \ls M_0+\delta \sup_{[0,T]} E_5(t)+CTP(\sup_{[0,T]}E_5(t)).
\end{align*}

Due to \eqref{eq.Jt}, we see that
\begin{align*}
  \pd{\D_t^6 J^{-2}}{\alpha}=&-2 \pd{\D_t^5(J^{-2}A_i^j\pd{v^i}{j})}{\alpha}\\
  =&-2 J^{-2}A_i^j\D_t^5\pd{v^i}{j\alpha}-2 \pd{(J^{-2}A_i^j)}{\alpha}\D_t^5\pd{v^i}{j}-2 \pd{\D_t^5 (J^{-2}A_i^j)}{\alpha}\pd{v^i}{j}\\
  &-2 \D_t^5 (J^{-2}A_i^j)\pd{v^i}{j\alpha}-2 \sum_{l=1}^4 C_5^l \pd{[\D_t^l( J^{-2}A_i^j)\D_t^{5-l}\pd{v^i}{j}]}{\alpha},
\end{align*}
namely, in view of the fundamental theorem of calculus,
\begin{align*}
  \dv\pd{\D_t^5v}{\alpha}=&-\frac{1}{2}\pd{\D_t^6 J^{-2}}{\alpha}-\int_0^t(J^{-2}A_i^j)_tdt'\D_t^5\pd{v^i}{j\alpha}- \pd{(J^{-2}A_i^j)}{\alpha}\D_t^5\pd{v^i}{j}\\
  &- \pd{\D_t^5 (J^{-2}A_i^j)}{\alpha}\pd{v^i}{j}- \D_t^5 (J^{-2}A_i^j)\pd{v^i}{j\alpha}- \sum_{l=1}^4 C_5^l \pd{[\D_t^l( J^{-2}A_i^j)\D_t^{5-l}\pd{v^i}{j}]}{\alpha}.
\end{align*}
We can easily estimate last three terms by using the fundamental theorem of calculus and the H\"older inequality. Thus, we obtain
\begin{align*}
  \norm{\dv\pd{ \D_t^5v}{\alpha}}_0^2\ls M_0+\delta \sup_{[0,T]} E_5(t)+CTP(\sup_{[0,T]}E_5(t)).
\end{align*}
According to Proposition \ref{prop.curl}, we have
\begin{align*}
  \norm{\curl\hd \D_t^5v}_0^2\ls\norm{\curl\D_t^5 v}_1^2\ls M_0+CTP(\sup_{[0,T]}E_5(t)).
\end{align*}
With the boundary estimates on $\D_t^5 \hd v^\beta$  given by Proposition \ref{prop.mix1}, we obtain, from Lemma \ref{lem.Hodge}, that
\begin{align*}
  \norm{\pd{\D_t^5 v}{\alpha}}_1^2\ls M_0+\delta \sup_{[0,T]} E_5(t)+CTP(\sup_{[0,T]}E_5(t)).
\end{align*}
Thus, we have proved for $\alpha=1,2$
\begin{align}\label{eq.v3.3}
  &\sup_{[0,T]}\left[\norm{\pd{\D_t^5 v}{\alpha}(t)}_1^2+\norm{\ri \pd{\D_t^6 J^{-2}}{\alpha}(t)}_1^2\right]
  \ls M_0+\delta \sup_{[0,T]} E_5(t)+CTP(\sup_{[0,T]}E_5(t)).
\end{align}

We next differentiate \eqref{eq.vt3} in the vertical direction $x_3$ to obtain
\begin{align}\label{eq.v3.2}
  &\ri  a_i^3 \D_t^6\pd{J^{-2}}{33}+3\pd{\ri }{3} a_i^3 \pd{\D_t^6 J^{-2}}{3}\no\\
    =&\Big[-\D_t^7v^i-\ri  \D_t^6(a_i^\beta\pd{J^{-2}}{\beta})-2 \pd{\ri}{\beta} \D_t^6(a_i^\beta J^{-2})-\D_t^6 a_i^3[ \ri  \pd{J^{-2}}{3}+2 \pd{\ri}{3}J^{-2}]\no\\
    &-\sum_{l=1}^5 C_6^l\D_t^{6-l} a_i^3\D_t^l[ \ri  \pd{J^{-2}}{3}+2 \pd{\ri}{3}J^{-2}]\pd{\Big]}{3} -\ri  \pd{a_i^3}{3} \D_t^6\pd{J^{-2}}{3}\no\\
    &-2\pd{\ri }{33} a_i^3\D_t^6 J^{-2}-2 \pd{a_i^3}{3}\pd{\ri}{3}\D_t^6 J^{-2}.
\end{align}
Propositions \ref{prop.dt7v} and \ref{prop.mix1} together with the inequality \eqref{eq.v3.3} show that the right hand side of \eqref{eq.v3.2} is bounded in $L^2(\Omega)$ by $M_0+\delta \sup_{[0,T]} E_5(t)+CTP(\sup_{[0,T]}E_5(t))$. It follows that
\begin{align*}
  \lnorm{\ri  a_i^3 \D_t^6\pd{J^{-2}}{33}+3\pd{\ri }{3} a_i^3 \pd{\D_t^6 J^{-2}}{3}}_0^2\ls M_0+\delta \sup_{[0,T]} E_5(t)+CTP(\sup_{[0,T]}E_5(t)).
\end{align*}
Thus, by the H\"older inequality and the fundamental theorem of calculus, we get
\begin{align*}
  &\norm{\ri  a_i^3 \D_t^6\pd{J^{-2}}{33}}_0^2+9\norm{\pd{\ri }{3} a_i^3 \pd{\D_t^6 J^{-2}}{3}}_0^2\\
  \ls& M_0+\delta \sup_{[0,T]} E_5(t)+CTP(\sup_{[0,T]}E_5(t))-6\int_\Omega\ri  \pd{\ri }{3} \abs{a_\cdot^3}^2\pd{\D_t^6 J^{-2}}{3}\D_t^6\pd{J^{-2}}{33}dx\\
  \ls &M_0+\delta \sup_{[0,T]} E_5(t)+CTP(\sup_{[0,T]}E_5(t))+\delta \norm{\ri  \D_t^6\pd{J^{-2}}{33} }_0^2+\delta\norm{\pd{\ri }{3}\pd{\D_t^6 J^{-2}}{3}}_0^2,
\end{align*}
and then, by the fundamental theorem of calculus once again,
\begin{align*}
  \norm{\ri  \D_t^6\pd{J^{-2}}{33}}_0^2+\norm{\pd{\ri }{3} \pd{\D_t^6 J^{-2}}{3}}_0^2
  \ls M_0+\delta \sup_{[0,T]} E_5(t)+CTP(\sup_{[0,T]}(t)).
\end{align*}
Similar to \eqref{eq.dt6j0}, we get, by integration by parts and Cauchy's inequality, that
\begin{align*}
\norm{\ri  \D_t^6\pd{J^{-2}}{3}}_0^2=&\int_\Omega \ri^2  (\D_t^6\pd{J^{-2}}{3})^2dx\\
  =&-2\int_\Omega x_3[\ri\pd{\ri}{3}(\D_t^6\pd{J^{-2}}{3})^2+\ri^2  \D_t^6\pd{J^{-2}}{3}\D_t^6\pd{J^{-2}}{33}]dx\\
  \ls& 2\norm{\ri  \D_t^6\pd{J^{-2}}{3}}_0\norm{\pd{\ri }{3} \pd{\D_t^6 J^{-2}}{3}}_0+2\norm{\ri  \D_t^6\pd{J^{-2}}{3}}_0\norm{\ri  \D_t^6\pd{J^{-2}}{33}}_0,
\end{align*}
which implies, by eliminating one $\norm{\ri  \D_t^4\pd{J^{-2}}{3}}_0$ from both sides and using the Cauchy inequality, that
\begin{align*}
  \norm{\ri  \D_t^6\pd{J^{-2}}{3}}_0^2\ls & 8\left(\norm{\ri  \D_t^6\pd{J^{-2}}{33}}_0^2+\norm{\pd{\ri }{3} \pd{\D_t^6 J^{-2}}{3}}_0^2\right)\\
  \ls&  M_0+\delta \sup_{[0,T]} E_5(t)+CTP(\sup_{[0,T]}E_5(t)).
\end{align*}
Then, by \eqref{eq.v3.3} and \eqref{w-em}, we can obtain
\begin{align*}
  \norm{\ri  \D_t^6 J^{-2}}_2^2+\norm{\D_t^6 J^{-2}}_1^2\ls M_0+\delta \sup_{[0,T]} E_5(t)+CTP(\sup_{[0,T]}E_5(t)).
\end{align*}
Thus, we can infer that
\begin{align*}
  \norm{\dv \D_t^5 v(t)}_1^2\ls M_0+\delta \sup_{[0,T]} E_5(t)+CTP(\sup_{[0,T]}E_5(t)).
\end{align*}
According to Proposition \ref{prop.curl}, we have
\begin{align*}
  \norm{\curl\D_t^5 v}_1^2\ls M_0+CTP(\sup_{[0,T]}E_5(t)).
\end{align*}
With the boundary estimates on $\D_t^5 v^\beta$  given by Proposition \ref{prop.mix1}, we obtain, from Lemma \ref{lem.Hodge}, that
\begin{align*}
  \norm{\D_t^5v}_2^2\ls M_0+\delta \sup_{[0,T]} E_5(t)+CTP(\sup_{[0,T]}E_5(t)).
\end{align*}
This completes the proof.
\end{proof}

We also have the following propositions.

\begin{proposition}
  For $t\in [0,T]$, it holds that
  \begin{align*}
    \sup_{[0,T]}\Big[\norm{ v_{ttt}(t)}_3^2+\norm{\ri  \D_t^4J^{-2}(t)}_3^2\Big]\ls  M_0+\delta \sup_{[0,T]} E_5(t)+CTP(\sup_{[0,T]}E_5(t)).
  \end{align*}
\end{proposition}

\begin{proof}
  Applying $\D_t^4$ on \eqref{eu0}, we have
  \begin{align}\label{eq.vt1}
    &\ri  a_i^3 \D_t^4\pd{J^{-2}}{3}+2 a_i^3\pd{\ri}{3} \D_t^4 J^{-2}\no\\
    =&-\D_t^5v^i-\ri  \D_t^4(a_i^\beta\pd{J^{-2}}{\beta})-2 \pd{\ri}{\beta} \D_t^4(a_i^\beta J^{-2})\no\\
    &-\D_t^4 a_i^3\left[ \ri  \pd{J^{-2}}{3}+2 \pd{\ri}{3}J^{-2}\right]-\sum_{l=1}^3\D_t^{4-l} a_i^3\D_t^l\left[ \ri  \pd{J^{-2}}{3}+2 \pd{\ri}{3}J^{-2}\right].
\end{align}
The same argument used in the proof of Proposition \ref{prop.vt5} yields the desired results.
\end{proof}

\begin{proposition}
  For $t\in [0,T]$, it holds that
  \begin{align*}
    \sup_{[0,T]}\Big[\norm{ v_{t}(t)}_4^2+\norm{\ri  \D_t^2J^{-2}(t)}_4^2\Big]\ls  M_0+\delta \sup_{[0,T]} E_5(t)+CTP(\sup_{[0,T]}E_5(t)).
  \end{align*}
\end{proposition}

\begin{proof}
  Applying $\D_t^2$ on \eqref{eu0}, we have
  \begin{align}\label{eq.vt11}
    &\ri  a_i^3 \D_t^2\pd{J^{-2}}{3}+2 a_i^3\pd{\ri}{3} \D_t^2 J^{-2}\no\\
    =&-\D_t^3v^i-\ri  \D_t^2(a_i^\beta\pd{J^{-2}}{\beta})-2 \pd{\ri}{\beta} \D_t^2(a_i^\beta J^{-2})\no\\
    &-\D_t^2 a_i^3\left[ \ri  \pd{J^{-2}}{3}+2 \pd{\ri}{3}J^{-2}\right]-2\D_t a_i^3\D_t\left[ \ri  \pd{J^{-2}}{3}+2 \pd{\ri}{3}J^{-2}\right].
\end{align}
The same argument used in the proof of Proposition \ref{prop.vt5} yields the desired results.
\end{proof}

\begin{proposition}
  For $t\in [0,T]$, it holds that
  \begin{align*}
    \sup_{[0,T]}\Big[\norm{ \eta(t)}_5^2+\norm{\ri  J^{-2}(t)}_5^2\Big]\ls  M_0+\delta \sup_{[0,T]} E_5(t)+CTP(\sup_{[0,T]}E_5(t)).
  \end{align*}
\end{proposition}

\begin{proof}
  We use directly the identity \eqref{eu0}, i.e.,
  \begin{align*}
&\ri a_i^3 \pd{J^{-2}}{3}+2 a_i^3\pd{\ri}{3}J^{-2}
=-v_t^i-\ri a_i^\beta\pd{J^{-2}}{\beta}-2 a_i^\beta\pd{\ri}{\beta}J^{-2}.
\end{align*}
The same argument used in the proof of Proposition \ref{prop.vt5} implies the desired results.
\end{proof}

\section{The a priori bound}\label{sec.9}

For the estimates of $\norm{\curleta v(t)}_4^2$ and $\norm{\ri\hd^5\curleta v(t)}_0^2$, it is similar to those in \cite{CLS10} so that we omit the details. Combining the inequalities provided by energy estimates, the additional elliptic estimates and the curl estimates shows, with the help of Proposition \ref{prop.interpolation}, that
\begin{align*}
  \sup_{[0,T]}E_5(t)\ls C(E_5(0))+CTP(\sup_{[0,T]}E_5(t)).
\end{align*}
According to the polynomial-type inequality \eqref{f}, by taking $T>0$ sufficiently small, we obtain the a priori bound
\begin{align*}
  \sup_{[0,T]}E_5(t)\ls 2C(E_5(0)).
\end{align*}
Therefore, we complete the proof of the main theorem.

\appendix

\section{Preliminaries}\label{sec.pre}

\subsection{Notations and weighted Sobolev spaces}\label{sec.2}

We make use of the Levi-Civita permutation symbol
\begin{align*}
  \eps_{ijk}=\frac{1}{2}(i-j)(j-k)(k-i)=\left\{\begin{array}{ll}
  1,\quad &\text{even permutation of } \{1,2,3\},\\
  -1,&  \text{odd permutation of } \{1,2,3\},\\
  0,& \text{otherwise},
  \end{array}\right.
\end{align*}
and the basic identity regarding the $i^{\mathrm{th}}$ component of the curl of a vector field $u$:
\begin{align*}
  (\curl u)_i=\eps_{ijk} \pd{u^k}{j},
\end{align*}
where it means that we have taken the sum with respect to the repeated scripts $j$ and $k$. Since $\pd{v^i}{j}=\D u^i/\D \eta^k\cdot \D\eta^k/\D x^j$, it follows that $\D u^i/\D \eta^k=\D x^j/\D\eta^k\cdot\pd{v^i}{j}$, i.e., $\pd{u^i}{k}=A_k^j\pd{v^i}{j}$. The chain rule shows that
\begin{align*}
  (\curl u(\eta))_i=(\curleta v)_i:=\eps_{ijk} \pd{v^k}{s}A_j^s,
\end{align*}
where the right-hand side defines the Lagrangian curl operator $\curleta$. Similarly, we have
\begin{align*}
  \dv u(\eta)=\dveta v:=\pd{v^j}{s}A_j^s,
\end{align*}
and the right-hand side defines the Lagrangian divergence operator $\dveta$. We also use the notation for any vector field $F$
\begin{align}
  (\nb_\eta F)_j^i=\pd{F^i}{s}A_j^s .
\end{align}

For any vector field $F$, we have
\begin{align*}
  \curleta F=&\curl(FA)=\begin{pmatrix}
                          \pd{F^3}{j}A^j_2-\pd{F^2}{j}A^j_3 \\
                          \pd{F^1}{j}A^j_3-\pd{F^3}{j}A^j_1 \\
                          \pd{F^2}{j}A^j_1-\pd{F^1}{j}A^j_2 \\
                        \end{pmatrix},
\end{align*}
and then
\begin{align}\label{eq.curletaF2}
  \abs{\nb_\eta F}^2=\abs{\curleta F}^2+(\nb_\eta F)\cdot(\nb_\eta F)^\top,
\end{align}
where the superscript $\top$ denotes the transpose of the matrix.

As a generalization of the standard Gronwall inequality, we introduce a polynomial-type inequality. For a constant $M_0\gs 0$,  suppose that $f(t)\gs 0$,
$t \mapsto f(t)$ is continuous,  and
\begin{align}\label{f}
f(t) \ls M_0 + C\,t\, P(f(t))\,,
\end{align}
where $P$ denotes a polynomial function,  and  $C$ is a generic constant.
Then for $t$ taken sufficiently small, we have the bound  (cf. \cite{CLS10,CS06})
$$
f(t) \ls 2M_0.
$$

For integers $k\gs 0$ and a smooth, open domain $\Omega$ of $\R^3$,
we define the Sobolev space $H^k(\Omega)$ ($H^k(\Omega; \R^3)$) to
be the completion of $C^\infty(\Omega)$ ($C^\infty(\Omega; \R^3)$)
in the norm
$$\|u\|_k := \left( \sum_{|\alpha|\le k}\int_\Omega \left|   D^\alpha u(x)
\right|^2 dx\right)^{1/2},$$
for a multi-index $\alpha \in {\mathbb Z}^3_+$, with the standard convention   $|\alpha|=\alpha_1 +\alpha_2+ \alpha_3$.
For real numbers $s\gs 0$, the Sobolev spaces $H^s(\Omega)$ and the norms $\| \cdot \|_s$ are defined by interpolation.
We will  write $H^s(\Omega)$ instead of $H^s(\Omega;{\mathbb R} ^3)$
for vector-valued functions.  In the case that $s\gs 3$, the above definition also holds for domains $\Omega$  of class $H^s$.

Our analysis will often make use of the following subspace of $H^1(\Omega)$:
$$
\dot{H}_0^1 = \{ u \in H^1(\Omega) \ : \  u=0 \text{ on } {\Gamma}, (x_1,x_2) \mapsto u(x_1,x_2) \text{ is periodic} \},
$$
where, as usual, the vanishing of $u$ on ${\Gamma}$ is understood in the sense of trace.

For  functions $u\in H^k({\Gamma})$, $k \gs 0$,  we set
$$|u|_k := \left( \sum_{|\alpha|\ls k}\int_{\Gamma} \abs{  \hd^\alpha  u(x)
}^2 dx\right)^{1/2},$$
for a multi-index $\alpha  \in {\mathbb Z}^2_+$.   For real $s \gs 0$, the Hilbert space $H^s({\Gamma})$ and the boundary norm $| \cdot |_s$ is defined by interpolation.  The negative-order Sobolev spaces $H^{-s}({\Gamma})$ are defined via duality: for  real $s \gs 0$,
$$
H^{-s}({\Gamma}) := [ H^s({\Gamma})]'.
$$

The derivative loss inherent to this degenerate problem is a consequence of the weighted embedding we now describe.

Using $d$ to denote the distance function to the boundary ${\Gamma_1}$, i.e., $d(x)=\dist(x,{\Gamma_1})$,  and letting $p=1$ or $2$, the weighted Sobolev space  $H^1_{d^p}(\Omega)$, with norm given by
$$\int_\Omega d(x)^p (|F(x)|^2+| \nb F (x)|^2 ) dx$$ for any $F \in H^1_{d^p}(\Omega)$,
satisfies the following embedding:
$$H^1_{d^p}(\Omega) \hookrightarrow   H^{1 - \frac{p}{2}}(\Omega);$$
that is, there is a constant $C>0$ depending only on $\Omega$ and $p$, such that
\begin{equation}\label{w-embed}
\|F\|_{1-p/2} ^2 \ls C \int_\Omega d(x)^p \bigl(|F(x)|^2 + \left|\nb F(x) \right|^2\bigr)  dx.
\end{equation}
See, for example,  Section 8.8 in Kufner \cite{Ku85}. From this embedding relation and \eqref{eq.ricon}, we obtain
\begin{align}\label{w-em}
  \norm{F}_0^2\ls C\int_\Omega \ri^2 \bigl(|F(x)|^2 + \left|\nb F(x) \right|^2\bigr)  dx.
\end{align}

\subsection{Higher-order Hardy-type inequality and Hodge-type elliptic estimates}\label{sec.3}

We will make fundamental use of the following generalization of the well-known Hardy
inequality to higher-order derivatives, see \cite[Lemma 3.1]{CS12}:
\begin{lemma}[Higher-order Hardy-type inequality] \label{lem.Hardy}
  Let $s\gs 1$ be a given integer, $\Omega$ and $d(x)$ be defined as above, and suppose that
  \begin{align*}
    u\in H^s(\Omega)\cap \dot{H}_0^1(\Omega),
  \end{align*}
  then $\frac{u}{d}\in H^{s-1}(\Omega)$ and
  \begin{align}\label{Hardy.1}
    \lnorm{\frac{u}{d}}_{s-1}\ls C\norm{u}_s.
  \end{align}
\end{lemma}

%

The normal trace theorem  provides  the existence of the
normal trace $w \cdot N$ of a velocity field $w\in L^2(\Omega)$ with
$\dv w \in L^2(\Omega) $ (see, e.g., \cite{Temam84}). For our purposes, the following
form is most useful (see, e.g., \cite{CCS08}):

\begin{lemma}[Normal trace theorem]\label{lem: normal trace} Let $w$ be a vector field defined on $\Omega$ such that $\hd w\in L^2(\Omega)$ and $\dv w\in L^2(\Omega)$, and let $N$ denote the outward unit normal vector to $\Gamma$. Then the normal trace $\hd w\cdot N$ exists in $H^{-1/2}(\Gamma)$ with the estimate
\begin{align}
	\abs{\hd w\cdot N}_{-1/2}^2 \le C \Big[\norm{\hd w}_{L^2(\Omega)}^2 +\norm{\dv w}_{L^2(\Omega)}^2\Big],\label{normaltrace}
\end{align}
for some constant $C$ independent of $w$.
\end{lemma}

\begin{lemma}[Tangential trace theorem]\label{lem.tangential.trace}
Let $\hd w\in L^2(\Omega)$ so that $\curl  w\in L^2(\Omega) $, and let
$T_1$, $T_2$ denote the unit tangent vectors on ${\Gamma}$,
so that any vector field $u$ on ${\Gamma}$ can be uniquely written as $u^\alpha
T_\alpha$. Then
\begin{align}
\abs{\hd w\cdot T_\alpha}_{-1/2}^2 \le C
\Big[\|\hd w\|^2_{L^2(\Omega)} + \| \curl
w\|^2_{ L^2(\Omega) }\Big]\,,\qquad\alpha=1,2, \label{tangentialtrace}
\end{align}
for some constant $C$ independent of $w$.
\end{lemma}

Combining (\ref{normaltrace}) and (\ref{tangentialtrace}), we have
\begin{align}
\abs{\hd w}_{-1/2} \le C\Big[\|\hd w\|_{L^2(\Omega)} + \|\dv
w\|_{L^2(\Omega) } + \|\curl  w\|_{L^2(\Omega) }\Big]
\label{tracetemp}
\end{align}
for some constant $C$ independent of $w$.

The construction of our higher-order energy function is based on the following Hodge-type elliptic estimate (see, e.g., \cite{CLS10}):

\begin{lemma}\label{lem.Hodge}
  For an $H^r$ domain $\Omega$, $r\gs 3$, if $F\in L^2(\Omega;\R^3)$ with $\curl F\in H^{s-1}(\Omega)$, $\dv F\in H^{s-1}(\Omega)$, and $F\cdot N|_{\Gamma}\in H^{s-1/2}({\Gamma})$ for $1\ls s\ls r$, then there exists a constant $\bar{C}>0$ depending only on $\Omega$ such that
  \begin{align*}
      \norm{F}_s\ls &\bar{C}\left(\norm{F}_0+\norm{\curl F}_{s-1}+\norm{\dv F}_{s-1}+\abs{\hd F\cdot N}_{s-3/2}\right),\\
      \norm{F}_s\ls &\bar{C}\Big(\norm{F}_0+\norm{\curl F}_{s-1}+\norm{\dv F}_{s-1}+\sum_{\alpha=1}^2\abs{\hd F\cdot T_\alpha}_{s-3/2}\Big),
  \end{align*}
  where $N$ denotes the outward unit normal to ${\Gamma}$ and $T_\alpha$ are tangent vectors for $\alpha=1,2$.
\end{lemma}

\subsection{Properties of $J$, $A$ and $a$}\label{sec.4}

From the derivative formula of matrices and determinants, we have
\begin{align}
  \pd{A_i^k}{s}=&-A_l^k\pd{\eta^l}{js}A_i^j,\label{eq.Ad}\\
  A^k_{ti}=&\D_tA^k_i=-A^k_l\pd{v^l}{j}A^j_i,\label{eq.At}\\
  \pd{J}{s}=&JA_l^j\pd{\eta^l}{js}=a_l^j\pd{\eta^l}{js}.\label{eq.Jd}
\end{align}
It follows from $a=JA$, \eqref{eq.Ad} and \eqref{eq.Jd} that the columns of every adjoint matrix are divergence-free, i.e., the Piola identity,
\begin{align}\label{eq.Piola}
  \pd{a_i^k}{k}=0,
\end{align}
which will play a vital role in our energy estimates. We also have
\begin{align}
  \pd{a_i^k}{s}=&J\pd{\eta^l}{js}(A_i^kA_l^j-A_i^jA_l^k) =J^{-1}\pd{\eta^l}{js}(a_i^ka_l^j-a_i^ja_l^k),\label{eq.ahd}\\
  a_{ti}^k=&J \pd{v^l}{j}(A_i^kA_l^j-A_i^jA_l^k) =J^{-1}\pd{v^l}{j}(a_i^ka_l^j-a_i^ja_l^k).\label{eq.adt}
\end{align}

\medskip

\textbf{Acknowledgement.}
  The author would like to thank Professor T. Luo for his valuable
comments and helpful discussions.


\end{document}